\documentclass[a4paper,11pt]{article}
\usepackage{amsmath,amsthm}
\usepackage{aligned-overset}
\usepackage{authblk}
\usepackage[style=numeric-comp,maxbibnames=99,bibencoding=utf8,firstinits=true]{biblatex}
\usepackage{booktabs}
\usepackage{cancel}
\usepackage[nomarkers,nolists]{endfloat}
\usepackage[hmargin={26mm,26mm},vmargin={30mm,35mm}]{geometry}
\usepackage{nicefrac}
\usepackage{paralist}
\usepackage[colorlinks, allcolors=blue]{hyperref}
\usepackage{loglogslopetriangle}
\usepackage{subcaption}
\usepackage[compact,small]{titlesec}

\usepackage{tikz-cd}
\usepackage{pgfplotstable}
\usepackage{newtxmath}
\usepackage{newtxtext}

\usepackage{graphicx}
\graphicspath{{./data/meshes/}}

\bibliography{divdivsymm3d}

\newcommand{\email}[1]{\href{mailto:#1}{#1}}

\allowdisplaybreaks

\numberwithin{equation}{section}

\usepackage[normalem]{ulem}
\definecolor{violet}{rgb}{0.580,0.,0.827}

\newtheorem{theorem}{Theorem}

\newtheorem{lemma}[theorem]{Lemma}

\theoremstyle{remark}
\newtheorem{remark}[theorem]{Remark}
\theoremstyle{definition}

\DeclareRobustCommand{\bvec}[1]{\boldsymbol{#1}}
\newcommand{\uvec}[1]{\underline{\bvec{#1}}}

\newcommand{\cvec}[1]{\bvec{\mathcal{#1}}}

\DeclareRobustCommand{\btens}[1]{\boldsymbol{#1}}
\newcommand{\utens}[1]{\underline{\bvec{#1}}}

\newcommand{\ctens}[1]{\bvec{\mathcal{#1}}}

\newcommand{\Real}{\mathbb{R}}

\newcommand{\Symm}{\mathbb{S}}
\newcommand{\Matr}{\Real^{3\times 3}}
\newcommand{\Tless}{\mathbb{T}}
\newcommand{\ID}[1]{\boldsymbol{I}_{#1}}

\newcommand{\st}{\,:\,}
\DeclareMathOperator{\tdot}{\bf :}
\DeclareMathOperator{\tr}{tr}
\DeclareMathOperator{\Tr}{tr}
\DeclareMathOperator{\DEV}{\bf dev}
\DeclareMathOperator{\SYM}{\bf sym}

\DeclareMathOperator{\DIM}{dim}
\DeclareMathOperator{\Ker}{Ker}
\DeclareMathOperator{\Image}{Im}
\newcommand{\term}{\mathfrak{T}}

\newcommand{\Ctensor}{\mathbb{C}}
\newcommand{\dotp}[1]{\partial_{\tangent_E}#1}
\newcommand{\dotpp}[1]{\partial_{\tangent_E}^2#1}

\DeclareMathOperator{\DIV}{div}
\DeclareMathOperator{\VDIV}{\bf div}
\DeclareMathOperator{\CURL}{\bf curl}
\DeclareMathOperator{\GRAD}{\bf grad}
\DeclareMathOperator{\HESS}{\bf hess}

\newcommand{\sym}{{\rm sym}} 
\newcommand{\compl}{{\rm c}}
\newcommand{\trimmed}{{-}}

\newcommand{\jump}[2]{\llbracket#2\rrbracket_{#1}}

\newcommand{\Poly}[2][]{\mathcal{P}_{#1}^{#2}}
\newcommand{\vPoly}[2][]{\cvec{P}_{#1}^{#2}}
\newcommand{\tPoly}[2][]{\ctens{P}_{#1}^{#2}}

\newcommand{\Roly}[1]{\boldsymbol{\mathcal{R}}^{#1}}
\newcommand{\cRoly}[1]{\boldsymbol{\mathcal{R}}^{\compl,#1}}

\newcommand{\cGoly}[1]{\boldsymbol{\mathcal{G}}^{\compl,#1}}
\newcommand{\Holy}[1]{\cvec{H}^{#1}}
\newcommand{\cHoly}[1]{\cvec{H}^{\compl,#1}}
\newcommand{\CGoly}[1]{\cvec{CG}^{#1}} 
\newcommand{\cCGoly}[1]{\cvec{CG}^{\compl,#1}}
\newcommand{\SRoly}[1]{\cvec{SR}^{#1}}
\newcommand{\cSRoly}[1]{\cvec{SR}^{\compl,#1}}

\newcommand{\RT}[1]{\boldsymbol{\mathcal{R}}^{\trimmed,#1}}
\newcommand{\HtrimPoly}[1]{\cvec{H}^{\trimmed,#1}}
\newcommand{\SRtrimPoly}[1]{\cvec{SR}^{\trimmed,#1}}
\newcommand{\CGtrimPoly}[1]{\cvec{CG}^{\trimmed,#1}}
\newcommand{\Xtrim}[1]{\cvec{X}^{\trimmed,#1}}

\newcommand{\lproj}[2]{\pi_{\mathcal{P},#2}^{#1}}
\newcommand{\vlproj}[2]{\bvec{\pi}_{\cvec{P},#2}^{#1}}
\newcommand{\tlproj}[2]{\btens{\pi}_{\ctens{P},#2}^{#1}}

\newcommand{\Rproj}[2][T]{\bvec{\pi}_{\cvec{R},#1}^{#2}}
\newcommand{\cRproj}[2][T]{\bvec{\pi}_{\cvec{R},#1}^{\compl,#2}}

\newcommand{\CGproj}[2][F]{\btens{\pi}_{\cvec{CG},#1}^{#2}}
\newcommand{\cCGproj}[2][F]{\btens{\pi}_{\cvec{CG},#1}^{\compl,#2}}
\newcommand{\SRproj}[2][T]{\btens{\pi}_{\cvec{SR},#1}^{#2}}
\newcommand{\cSRproj}[2][T]{\btens{\pi}_{\cvec{SR},#1}^{\compl,#2}}

\newcommand{\RTproj}[2][T]{\btens{\pi}_{\boldsymbol{\mathcal{R}},#1}^{\trimmed,#2}}
\newcommand{\Htrimproj}[2][T]{\btens{\pi}_{\cvec{H},#1}^{\trimmed,#2}}
\newcommand{\SRtrimproj}[2][T]{\btens{\pi}_{\cvec{SR},#1}^{\trimmed,#2}}
\newcommand{\CGtrimproj}[2][F]{\btens{\pi}_{\cvec{CG},#1}^{\trimmed,#2}}
\newcommand{\Xtrimproj}[1]{\btens{\pi}_{\cvec{X},Y}^{\trimmed,#1}}

\newcommand{\faces}[1]{\mathcal{F}_{#1}}
\newcommand{\edges}[1]{\mathcal{E}_{#1}}
\newcommand{\vertices}[1]{\mathcal{V}_{#1}}

\newcommand{\FT}{\faces{T}}
\newcommand{\ET}{\edges{T}}
\newcommand{\EF}{\edges{F}}

\newcommand{\VE}{\vertices{E}}
\newcommand{\VF}{\vertices{F}}
\newcommand{\VT}{\vertices{T}}

\newcommand{\normal}{\bvec{n}}
\newcommand{\tangent}{\bvec{t}}

\newcommand{\tE}{\tangent_E}
\newcommand{\tFo}{\tangent_{F,1}}
\newcommand{\tFd}{\tangent_{F,2}}
\newcommand{\nEo}{\normal_{E,1}}
\newcommand{\nEd}{\normal_{E,2}}
\newcommand{\nFE}{\normal_{FE}}
\newcommand{\nF}{\normal_F}

\newcommand{\wFE}{\omega_{FE}}
\newcommand{\wTF}{\omega_{TF}}

\newcommand{\Mh}{\mathcal{M}_h}
\newcommand{\Th}{\mathcal{T}_h}
\newcommand{\Fh}{\mathcal{F}_h}
\newcommand{\Eh}{\mathcal{E}_h}
\newcommand{\Vh}{\mathcal{V}_h}

\newcommand{\trtF}[1]{{#1}_{\tangent,F}}
\newcommand{\trnnF}[1]{{#1}_{\normal\normal,F}}
\newcommand{\trntF}[1]{{#1}_{\normal\tangent,F}}
\newcommand{\trttF}[1]{{#1}_{\tangent\tangent,F}}

\newcommand{\trnE}[1]{{#1}_{\normal,E}}
\newcommand{\trnnE}[1]{{#1}_{\normal\normal,E}}
\newcommand{\trntE}[1]{{#1}_{\normal\tangent,E}}
\newcommand{\trttE}[1]{{#1}_{\tangent\tangent,E}}

\newcommand{\trItF}[1]{{\widehat{#1}}_{\tangent,F}}
\newcommand{\trInF}[1]{{\widehat{#1}}_{\normal,F}}

\newcommand{\sumFT}{\sum_{F\in\FT}\wTF\int_F}
\newcommand{\sumEF}{\sum_{E\in\EF}\wFE\int_E}
\newcommand{\sumEFT}{\sum_{F\in\FT}\wTF\sum_{E\in\EF}\wFE\int_E}

\newcommand{\DIVF}{\DIV_F}
\newcommand{\VDIVF}{\VDIV_F}
\newcommand{\CURLF}{\CURL_F}
\DeclareMathOperator{\ROTF}{rot_{F}}
\DeclareMathOperator{\VROT}{\mathbf{rot}}
\DeclareMathOperator{\VROTF}{\VROT_F}
\newcommand{\GRADF}{\GRAD_F}

\newcommand{\nSPoly}[2]{\tPoly{#1}(#2;\SYM\Real^{2\times 2})}
\newcommand{\nPoly}[2]{\tPoly{#1}(#2;\Real^{2\times 2})}

\newcommand{\Hdevgrad}[2]{\bvec{H}^1(#1;#2)}

\newcommand{\Hsymcurl}[2]{\bvec{H}(\SYM\CURL,#1;#2)}

\newcommand{\Hdivdiv}[2]{\bvec{H}(\DIV\VDIV,#1;#2)}

\newcommand{\norm}[2]{\|#2\|_{#1}}
\newcommand{\seminorm}[2]{|#2|_{#1}}
\newcommand{\vvvert}{\vert\kern-0.25ex\vert\kern-0.25ex\vert}
\newcommand{\tnorm}[2]{\vvvert #2\vvvert_{#1}}

\newcommand{\Eprod}{\mathcal{E}_{{\rm prod},h}}
\newcommand{\Edivdiv}{\widehat{\mathcal{E}}_{\DIV\VDIV,h}}
\newcommand{\Err}{\mathcal{E}_h}

\newcommand{\uHdevgrad}[1]{\uvec{X}_{\DEV\GRAD,#1}^k}
\newcommand{\uHsymcurl}[1]{\uvec{X}_{\SYM\CURL,#1}^k}
\newcommand{\uHdivdiv}[1]{\uvec{X}_{\DIV\VDIV,#1}^k}

\newcommand{\IDGrad}[1]{\uvec{I}_{\DEV\GRAD,#1}^k}
\newcommand{\ISCurl}[1]{\utens{I}_{\SYM\CURL,#1}^k}
\newcommand{\IDivDiv}[1]{\utens{I}_{\DIV\VDIV,#1}^k}

\newcommand{\DG}[1]{\utens{DG}_{#1}^k}
\newcommand{\DGT}{\btens{DG}_{T}^k}
\newcommand{\DGntF}{\bvec{DG}_{\normal\tangent,F}^k}

\newcommand{\DGttF}{\btens{DG}_{\tangent\tangent,F}^k}
\newcommand{\SC}[1]{\utens{SC}_{#1}^k}
\newcommand{\SCT}{\btens{SC}_{T}^k}
\newcommand{\SCE}{\btens{SC}_{E}^{k+1}}
\newcommand{\SCnnF}{{SC}_{\normal\normal,F}^{k+1}}
\newcommand{\SCddofF}{{SC}_{D,F}^{k+1}}
\newcommand{\DD}[1]{DD_{#1}^k}

\newcommand{\gammaF}{\gamma_{\normal\normal,F}^{k}}
\newcommand{\TP}[1][T]{\btens{P}_{#1}^k}
\newcommand{\DerE}[1]{\mathfrak{G}_E^{#1}}

\newcommand{\htvec}[2][T]{\bvec{#2}_{\cvec{H},#1}}
\newcommand{\srtvec}[2][T]{\bvec{#2}_{\cvec{SR},#1}}
\newcommand{\cgtvec}[2][F]{\bvec{#2}_{\cvec{CG},#1}}
\newcommand{\rttvec}[2][T]{\bvec{#2}_{\boldsymbol{\mathcal{R}},#1}}
\newcommand{\srvec}[2][T]{\bvec{#2}_{\cvec{SR},#1}}
\newcommand{\cgvec}[2][F]{\bvec{#2}_{\cvec{CG},#1}}

\newcommand{\gdofV}[2][V]{\btens{G}_{\bvec{#2},#1}}
\newcommand{\gdofE}[1]{\btens{G}_{\bvec{#1},E}}
\newcommand{\gdofF}[1]{G_{\bvec{#1},F}}
\newcommand{\cdofE}[1]{\btens{C}_{\btens{#1},E}}
\newcommand{\doftE}[1]{{#1}_{\tangent,E}}
\newcommand{\hdoftE}[1]{\widehat{#1}_{\tangent,E}}

\newcommand{\dofnE}[1]{\bvec{#1}_{\normal,E}}
\newcommand{\hdofnE}[1]{\widehat{\bvec{#1}}_{\normal,E}}
\newcommand{\dofnF}[1]{{#1}_{\normal,F}}
\newcommand{\hdoftF}[1]{\widehat{\bvec{#1}}_{\tangent,F}}
\newcommand{\doftF}[1]{\bvec{#1}_{\tangent,F}}

\title{A discrete three-dimensional divdiv complex on polyhedral meshes with application to a mixed formulation of the biharmonic problem}
\author{Daniele A. Di Pietro}
\author{Marien-Lorenzo Hanot}
\affil{%
  IMAG, Univ Montpellier, CNRS, Montpellier, France \\
  \email{daniele.di-pietro@umontpellier.fr}, \email{marien-lorenzo.hanot@umontpellier.fr}%
}

\begin{document}

\maketitle

\begin{abstract}
  In this work, following the Discrete de Rham (DDR) paradigm, we develop an arbitrary-order discrete divdiv complex on general polyhedral meshes.
  The construction rests 1) on discrete spaces that are spanned by vectors of polynomials whose components are attached to mesh entities and 2) on discrete operators obtained mimicking integration by parts formulas.
  We provide an in-depth study of the algebraic properties of the local complex, showing that it is exact on mesh elements with trivial topology.
  The new DDR complex is used to design a numerical scheme for the approximation of biharmonic problems, for which we provide detailed stability and convergence analyses.
  Numerical experiments complete the theoretical results.
\end{abstract}



\section{Introduction}

Let $\Omega\subset\Real^3$ be a polyhedral domain with boundary $\partial\Omega$.
Denote by $\SYM$ and $\DEV$ the symmetrisation and deviator operators such that, for any matrix $\btens{M}\in\Real^{d\times d}$, $\SYM\btens{M}\coloneq\frac12\left(\btens{M} + \btens{M}^\top\right)$ and $\DEV\btens{M}\coloneq\btens{M} - \frac1d(\tr\btens{M})\ID{d}$.
We construct a discrete counterpart of the three-dimensional divdiv complex
\begin{small}
  \begin{equation} \label{eq:complex}
    \begin{tikzcd}
      \RT{1}(\Omega) \arrow[r,"i_\Omega"]
      & \Hdevgrad{\Omega}{\Real^3} \arrow[r,"\DEV\GRAD"]
      & \Hsymcurl{\Omega}{\Tless} \arrow[r,"\SYM\CURL"]
      &\Hdivdiv{\Omega}{\Symm} \arrow[r,"\DIV\VDIV"]
      & L^2(\Omega) \arrow[r,"0"]
      & \lbrace 0 \rbrace,
    \end{tikzcd}
  \end{equation}%
\end{small}%
where $\RT{1}\coloneq\vPoly{0}(\Omega) + \bvec{x}\Poly{0}(\Omega)$ is the lowest-order Raviart--Thomas space,
$\Hdevgrad{\Omega}{\Real^3}$ is spanned by vector-valued functions that are square-integrable over $\Omega$ along with their gradient,
$\Hsymcurl{\Omega}{\Tless}$ by functions taking values in $\Tless \coloneq \DEV \Real^{3\times 3}$ that are square-integrable over $\Omega$ along with the symmetric part of their curl,
and $\Hdivdiv{\Omega}{\Symm}$ by functions taking values in $\Symm \coloneq \SYM \Real^{3\times 3}$ that are square-integrable together with the divergence of their (row-wise) divergence.
The divdiv complex can be derived from the de Rham complex through the BGG construction \cite{Arnold.Hu:21}, which offers a powerful framework to study its theoretical properties, but still lacks a generic blueprint for the construction of discrete complexes.

The main difficulty in the numerical approximation of the complex \eqref{eq:complex} is related to the algebraic constraints that appear in both the spaces and the operators.
Finite element approximations of the spaces appearing in the complex have been developed in \cite{Chen.Huang:22*1,Adams.Cockburn:05,Arnold.Awanou.ea:08,Arnold.Winther:02}.
The discretization of the full complex is, on the other hand, much more recent \cite{Hu.Liang.ea:22}.

The above references are concerned with spaces built on standard (matching simplicial) meshes.
In this work, following the Discrete de Rham (DDR) paradigm of \cite{Di-Pietro.Droniou.ea:20,Di-Pietro.Droniou:23} (see also \cite{Bonaldi.Di-Pietro.ea:23} for a very recent generalization using differential forms), we address the discretization of the divdiv complex \eqref{eq:complex} on more general meshes made of polyhedral elements and possibly featuring non-matching interfaces.
The support of such meshes provides great flexibility in the approximation of the domain and enables an efficient use of computational resources through non-conforming local mesh refinement and agglomeration \cite{Bassi.Botti.ea:12,Antonietti.Giani.ea:13,Antonietti.Cangiani.ea:16}.
  Polytopal methods additionally benefit from a higher-level point of view, which enables unknowns-reduction strategies such as serendipity \cite{Beirao-da-Veiga.Brezzi.ea:18,Beirao-da-Veiga.Mascotto:23,Chen.Sukumar:23}; see also \cite{Di-Pietro.Droniou:23*2} for a general framework in the context of discrete complexes.
The key idea of DDR methods consists in replacing both the spaces and the operators in the complex with discrete counterparts.
Discrete spaces are spanned by vectors of polynomials with components attached to the mesh entities, while discrete operators are obtained mimicking integration by parts formula.
Applying this paradigm to the discretization of the divdiv complex involves a number of subtleties, from the decomposition of traces of tensor-valued fields to the identification of the appropriate integration by parts formulas.
We provide a complete study of the algebraic properties of the local complex showing how the design of the spaces and operators fits to ensure exactness on mesh elements with trivial topology.
Local exactness is one of the key ingredients to prove algebraic properties of the global complex following, e.g., the paradigm of \cite{Di-Pietro.Droniou.ea:23}.

The DDR divdiv complex is then used as a starting point to design a numerical scheme for the following fourth-order problem in mixed formulation:
Given $f:\Omega\to\Real$, find $\btens{\sigma}\in\Hdivdiv{\Omega}{\Symm}$ and $u\in L^2(\Omega)$ such that
\begin{equation}\label{eq:weak}
  \begin{alignedat}{4}
    \int_\Omega\btens{\sigma}:\btens{\tau}
    + \int_\Omega\DIV\VDIV\btens{\tau}~u &= 0 &\qquad& \forall\btens{\tau}\in\Hdivdiv{\Omega}{\Symm},
    \\
    -\int_\Omega\DIV\VDIV\btens{\sigma}~v &= \int_\Omega fv
    &\qquad&
    \forall v\in L^2(\Omega).
  \end{alignedat}
\end{equation}
Previous results in the (significantly easier) two dimensional case include the design of a DDR complex along with its application to Kirchhoff--Love plates \cite{Di-Pietro.Droniou:23*1} and its serendipity variant \cite{Botti.Di-Pietro.ea:23}.
Based on the properties of the new three-dimensional divdiv complex, we prove stability of the DDR scheme for problem \eqref{eq:weak}, along with its convergence in $h^{k+1}$, with $h$ denoting the meshsize and $k$ the polynomial degree of the complex.

The rest of the paper is organized as follows.
In Section \ref{sec:setting} we establish the setting, including the relevant integration by parts formulas and trimmed polynomial spaces.
The discrete divdiv complex along with its algebraic properties make the object of Section \ref{sec:discrete.complex}.
Section \ref{sec:biharmonic} contains the DDR scheme for problem \eqref{eq:weak} as well as its stability and convergence analyses.
Sections \ref{sec:complex.property} and \ref{sec:exactness} contain the most technical proofs of algebraic properties of the DDR complex.
Finally, results on local polynomial space of general scope are presented in Appendix \ref{sec:results.polynomial.spaces}.


\section{Setting}\label{sec:setting}

\subsection{Mesh}

For any (measurable) set $Y\subset\Real^3$, we denote by $h_Y$ its diameter.
We consider meshes $\Mh\coloneq\Th\cup\Fh\cup\Eh\cup\Vh$ of $\Omega$, where:
$\Th$ is a finite collection of open disjoint polyhedral elements such that $\overline{\Omega} = \bigcup_{T\in\Th}\overline{T}$ and $h=\max_{T\in\Th}h_T>0$;
$\Fh$ is a finite collection of open planar faces;
$\Eh$ is the set collecting the open edges of the faces;
$\Vh$ is the set collecting the edge endpoints.
It is assumed, in what follows, that $(\Th,\Fh)$ matches the conditions in \cite[Definition 1.4]{Di-Pietro.Droniou:20}, so that the faces form a partition of the mesh skeleton $\bigcup_{T\in\Th}\partial T$.

Given a mesh edge $E\in\Eh$, we denote by $V_1(E)$ and $V_2(E)$ the vertices in $\Vh$ corresponding to its endpoints and ordered so that $\tangent_E = h_E^{-1}(\bvec{x}_{V_2(E)} - \bvec{x}_{V_1(E)})$.
For the sake of conciseness, whenever no ambiguity can arise, we avoid specifying the edge and simply write $V_1$ and $V_2$.
For any face $F\in\Fh$, we fix a unit normal vector $\normal_F$ and, for any edge $E\in\EF$,we denote by $\normal_{FE}$ the vector normal to $E$ in the plane containing $F$ and oriented such that $(\tangent_E,\normal_{FE},\normal_F)$ forms a right-handed system of coordinates.
Depending on the context, the vectors $\tangent_E$ and $\normal_{FE}$ may be regarded as embedded in the plane containing $F$ or in the three-dimensional space.

The set collecting the mesh faces that lie on the boundary of a mesh element $T\in\Th$ is denoted by $\FT$.
For any $Y\in\Th\cup\Fh$, we denote by $\edges{Y}$ the set of edges of $Y$.
Similarly, for all $Y\in\Th\cup\Fh\cup\Eh$, $\vertices{Y}$ denotes the set of vertices of $Y$.

For each mesh element or face $Y\in\Th\cup\Fh$, we fix a point $\bvec{x}_Y\in Y$ such that there exists a ball centered in $\bvec{x}_Y$ contained in $Y$ and of diameter comparable to $h_Y$ uniformly in $h$ (when $\Mh$ belongs to a regular mesh sequence in the sense of \cite[Definition 1.9]{Di-Pietro.Droniou:20}).

Throughout the paper, $a\lesssim b$ stands for $a\le Cb$ with $C$ depending only on $\Omega$, the mesh regularity parameter and, when polynomial functions are involved, the corresponding polynomial degree.

\subsection{Local and broken polynomial spaces}

For given integers $n\ge 0$ and $\ell\ge 0$, $\mathbb{P}_n^\ell$ denotes the space of $n$-variate polynomials of total degree $\le\ell$, with the convention that
$\mathbb{P}_0^\ell \coloneq \Real$ for any $\ell$ and that $\mathbb{P}_n^{-1} \coloneq \{ 0 \}$ for any $n$.
Given $Y\in\Th\cup\Fh\cup\Eh$, we denote by $\Poly{\ell}(Y)$ the space spanned by the restriction to $Y$ of the functions in $\mathbb{P}_3^\ell$ and by $\lproj{\ell}{Y}$ the corresponding $L^2$-orthogonal projector.
  When $Y$ is a mesh edge $E\in\Eh$ or face $F\in\Fh$, whenever needed we will identify $\Poly{\ell}(E)$ and $\Poly{\ell}(F)$ with the spaces of one- and two-variate polynomials on $E$ and $F$, respectively.
Spaces of vector- or matrix-valued functions on $Y$ that have polynomial components of total degree $\le\ell$ are denoted in boldface and the codomain is specified.
At the global level, we define the broken polynomial space
\begin{equation}\label{eq:broken.Pl}
  \Poly{\ell}(\Th)\coloneq\left\{v_h\in L^2(\Omega)\st \text{$(v_h)_{|T}\in\Poly{\ell}(T)$ for all $T\in\Th$}\right\}.
\end{equation}

\subsection{Direct decompositions of local polynomial spaces}

For any mesh face $F\in\Fh$ and any integer $\ell\ge 0$, the following direct decomposition of vector-valued polynomial functions holds (cf. \cite{Arnold:18}):
\[
\begin{gathered}
  \vPoly{\ell}(F;\Real^2) = \Roly{\ell}(F) \oplus \cRoly{\ell}(F)
  \\
  \text{
    with
    $\Roly{\ell}(F)\coloneq\CURLF\Poly{\ell+1}(F)$ and 
    $\cRoly{\ell}(F)\coloneq (\bvec{x} - \bvec{x}_F)\Poly{\ell-1}(F)$.
  }
\end{gathered}
\]
The following lemma contains a new direct decomposition 
  that will be needed to design the discrete counterpart of $\Hsymcurl{\Omega}$.
\begin{lemma}[Direct decomposition of matrix-valued polynomial fields on faces] \label{lemma:Poly=CGoly+cCGoly}
  For all $F\in\Fh$ and all $\ell\ge 0$, the following direct decomposition holds:
  \begin{equation} \label{eq:Poly=CGoly+cCGoly}
    \begin{gathered}
      \tPoly{\ell}(F;\Real^{2\times 2}) = \CGoly{\ell}(F) \oplus \cCGoly{\ell}(F),
      \\
      \text{%
        with
        $\CGoly{\ell}(F) \coloneq \CURLF \vPoly{\ell+1}(F;\Real^2)$ and
        $\cCGoly{\ell}(F) \coloneq (\mathrm{Id} - \mathrm{adj})\left[\vPoly{\ell-1}(F;\Real^2)\otimes(\bvec{x}-\bvec{x}_F)^\intercal\right]$,
      }
    \end{gathered}
  \end{equation}
  where $\mathrm{adj}$ is the adjugate operator acting on $2\times 2$ matrices.
\end{lemma}

\begin{proof}
  See Appendix~\ref{sec:results.polynomial.spaces}.
\end{proof}

In what follows, we will also need the decompositions of matrix-valued polynomial functions on mesh elements $T\in\Th$ described hereafter.
We start by recalling the following results (cf. \cite[Lemma 4.4]{Chen.Huang:20*1} and \cite[Lemma 3.6]{Chen.Huang:22}, respectively:
\begin{align}\nonumber 
  \vPoly{\ell}(T;\Tless)
  &= \SRoly{\ell}(T) \oplus \cSRoly{\ell}(T),
  \\ \label{eq:vPoly=Holy+cHoly}
  \vPoly{\ell}(T;\Symm)
  &= \Holy{\ell}(T) \oplus \cHoly{\ell}(T),
\end{align}
with
\[
\begin{alignedat}{3}
  \SRoly{\ell}(T)&\coloneq\CURL\vPoly{\ell+1}(T;\Symm),&\qquad
  \cSRoly{\ell}(T)&\coloneq\DEV\left[\vPoly{\ell-1}(T;\Real^3)\otimes(\bvec{x}-\bvec{x}_T)^\intercal\right],
  \\
  \Holy{\ell}(T)&\coloneq\HESS\Poly{\ell+2}(T),&\qquad
  \cHoly{\ell}(T)&\coloneq\SYM\left[\vPoly{\ell-1}(T;\Tless)\times(\bvec{x}-\bvec{x}_T)\right],
\end{alignedat}
\]
where the cross product $\btens{A}\times\bvec{b}$ between a matrix $\btens{A}\in\Real^{3\times 3}$ and a vector $\bvec{v}\in\Real^3$ is performed row-wise.
The following lemma establishes a link between $\SRoly{\ell}(T)$ and $\Holy{\compl,\ell+1}(T)$.

\begin{lemma}[Link between $\SRoly{\ell}(T)$ and $\Holy{\compl,\ell+1}(T)$]\label{lemma:poly.SRfromH}
  It holds
  \begin{equation*}
    \SRoly{\ell}(T) = \CURL \Holy{\compl,\ell+1}(T).
  \end{equation*}
\end{lemma}

\begin{proof}
  Since $\Holy{\compl,\ell+1}(T)\subset\tPoly{\ell+1}(T;\Symm)$, $\CURL \Holy{\compl,\ell+1}(T) \subset \SRoly{\ell}(T)$.
  Let now $\btens{\sigma} \in \SRoly{\ell}(T)$.
  By definition, there is $\btens{\tau} \in \tPoly{\ell+1}(T;\Symm)$ such that $\btens{\sigma} = \CURL \btens{\tau}$.
  Recalling \eqref{eq:vPoly=Holy+cHoly}, $\btens{\tau}$ can be decomposed as $\btens{\tau} = \btens{\tau}_1 + \btens{\tau}_2$ with $(\btens{\tau}_1,\btens{\tau}_2) \in \Holy{\ell}(T)\times \Holy{\compl,\ell}(T)$.
  Using $\CURL \HESS = \bvec{0}$, we have
  $\btens{\sigma} = \CURL \btens{\tau} = \cancel{\CURL \btens{\tau}_1} + \CURL \btens{\tau}_2 \in\CURL\Holy{\compl,\ell+1}(T)$.
  Since $\btens{\sigma}$ is generic in $\SRoly{\ell}(T)$, this concludes the proof.
\end{proof}

\begin{remark}[Extension to negative indices]
  The definitions of $\Roly{\ell}(F)$, $\CGoly{\ell}(F)$, and $\SRoly{\ell}(T)$ naturally extend to $\ell = -1$ (in which case, all of these spaces become trivial).
  Similarly, the definition of $\Holy{\ell}(T)$ extends to $\ell = -2$ and $\ell = -1$, yielding the trivial space in both cases.
\end{remark}

\subsection{Trimmed local polynomial spaces}

For any integer $\ell\ge 0$, trimmed polynomial spaces are obtained from the direct decompositions described in the previous section by lowering the degree of the first component.
Based on this principle we define:
For all $F\in\Fh$, 
\begin{align}\label{eq:RT}  
  \RT{\ell}(F) &\coloneq \Roly{\ell-1}(F) \oplus \cRoly{\ell}(F),
  \\ \label{eq:CGtrimPoly}
  \CGtrimPoly{\ell}(F) &\coloneq \CGoly{\ell-1}(F) \oplus \cCGoly{\ell}(F),
\end{align}
and, for all $T\in\Th$, 
\begin{align}\label{eq:SRtrimPoly}
  \SRtrimPoly{\ell}(T) &\coloneq \SRoly{\ell-1}(T) \oplus \cSRoly{\ell}(T),
  \\ \label{eq:HtrimPoly}
  \HtrimPoly{\ell}(T) &\coloneq \Holy{\ell-2}(T) \oplus \cHoly{\ell}(T).
\end{align}
Notice that, for $\ell = 0$, all of the above spaces become trivial.
For any $(\cvec{X},Y)\in\left\{(\cvec{R},F), (\cvec{CG},F), (\cvec{SR},T), (\cvec{H},T)\right\}$, we denote by $\Xtrimproj{\ell}$ the $L^2$-orthogonal projection on $\Xtrim{\ell}(Y)$.

\subsection{Reconstruction of tangent derivatives on edges}

We will often need to reconstruct tangential derivatives of functions over edges based on their vertex values and $L^2$-orthogonal projections.
Specifically, letting $\ell\ge 0$ be an integer and denoting by $\dotp{}$ the derivative along $E$ in the direction of $\tangent_E$, the tangential derivative reconstruction 
$\DerE{\ell}: \Real\times\Real\times\Poly{\ell-1}(E) \rightarrow\Poly{\ell}(E)$ is such that,
for any $(v_{V_1},v_{V_2},v_E) \in \Real\times\Real\times\Poly{\ell-1}(E)$,
\begin{equation} \label{eq:IPP.E}
  \int_E \DerE{\ell}(v_{V_1},v_{V_2},v_E)\, r
  = - \int_E v_E\, \partial_{\tE} r  
  + \jump{E}{v_V\, r}
  \qquad \forall r \in \Poly{\ell}(E),
\end{equation}
where $\jump{E}{\cdot}$ denotes the difference between vertex values on an edge such that, for any function $\phi\in C^0(\overline{E})$ and any family $\{w_{V_1},w_{V_2}\}$ of vertex values (possibly such that $w_{V_1} = w_{V_2} = 1$)
\[
\jump{E}{w_V\phi}\coloneq w_{V_2} \phi(\bvec{x}_{V_2}) - w_{V_1} \phi(\bvec{x}_{V_1}).
\]
When the arguments are vector- or matrix-valued, $\DerE{\ell}$ acts component-wise.
Noticing that $\DerE{\ell}$ coincides with the one-dimensional HHO gradient (cf., e.g., \cite[Eq. (4.37)]{Di-Pietro.Droniou:20}), it is readily inferred that
\begin{equation}\label{eq:DerE:commutation.property}
  \DerE{\ell}(\phi(\bvec{x}_{V_1}),\phi(\bvec{x}_{V_2}),\lproj{\ell-1}{E}\phi)
  = \lproj{\ell}{E}\dotp{\phi}
  \qquad\forall\phi\in H^1(E).
\end{equation}

\subsection{Notation and basic results on traces}\label{sec:setting:shortcuts.traces}

Given a family of linearly independent orthonormal vectors $\bvec{w} = \lbrace \bvec{w}_i \rbrace_{i \in I} \subset \Real^3$, 
we define the trace of a vector $\bvec{u}$ with respect to this family by $\bvec{u}_{\bvec{w}} \coloneq (\bvec{u}\cdot\bvec{w}_i)_{i \in I} \in \Real^I$
We also consider its injection into the original space
\begin{equation}\label{eq:injection.3d}
  \widehat{\bvec{u}}_{\bvec{w}}
  \coloneq \sum_{i \in I} (\bvec{u}\cdot\bvec{w}_i) \bvec{w}_i.
\end{equation}
Likewise, for two families of linearly independent orthonormal vectors 
$\bvec{v} = \lbrace \bvec{v}_i \rbrace_{i \in I}$, $\bvec{w} = \lbrace \bvec{w}_j \rbrace_{j \in J}$, 
we define the trace of a matrix $\bvec{A}$ with respect to these families by
$\bvec{A}_{\bvec{v}\bvec{w}} \coloneq (\bvec{v}_i^\top\bvec{A}\bvec{w}_j)_{(i,j) \in I\times J} \in \Real^{I\times J}$.
We also consider its injection into the original space 
$\bvec{\widehat{A}}_{\bvec{v}\bvec{w}} \coloneq \sum_{(i,j) \in I\times J} (\bvec{v}_i^\top\bvec{A}\bvec{w}_j)  \bvec{v}_i \otimes\bvec{w}_j$.

The notations defined above are used in what follows for traces on faces and edges as described hereafter.
For a face $F$, we consider an orthonormal basis $\lbrace \tFo, \tFd \rbrace$ of the plane tangent to $F$, and, for any vector-valued field $\bvec{w}:F\to\Real^3$ and any matrix-valued field $\btens{\tau}:F\to\Real^{3\times 3}$, write $\trtF{\bvec{w}}$, $\trnnF{\tau}$, $\trntF{\bvec{\tau}}$, and $\trttF{\btens{\tau}}$ with $\normal = \lbrace \nF \rbrace$ and $\tangent = \lbrace \tFo, \tFd \rbrace$.
Similarly, for any edge $E$, we consider an orthonormal basis $\lbrace \nEo, \nEd \rbrace$ of the plane normal to $E$, and, for any vector-valued field $\bvec{w}:E\to\Real^3$ and any matrix-valued field $\btens{\tau}:E\to\Real^{3\times 3}$, write $\trnE{\bvec{w}}$, $\trnnE{\btens{\tau}}$, $\trntE{\bvec{\tau}}$, and $\trttE{\tau}$ where $\normal = \lbrace \nEo, \nEd \rbrace$ and $\tangent = \{\tE\}$.

The following lemma shows that traces of functions in trimmed spaces lie in trimmed spaces.
Its proof is similar to that of \cite[Proposition~8]{Di-Pietro.Droniou:23} and is omitted for the sake of conciseness.

\begin{lemma}[Traces of trimmed spaces]
  For any element $T\in\Th$ and any face $F\in\FT$, it holds
  \begin{alignat}{4}
    \trttF{(\btens{\sigma} \nF)}&\in\vPoly{k-1}(F;\Real^2)
    &\qquad&\forall \btens{\sigma}\in\SRtrimPoly{k}(T), \label{eq:poly.trace.SR}\\
    \trttF{(\btens{\sigma}\times\nF)} &\in \CGtrimPoly{k}(F)
    &\qquad&\forall\btens{\sigma}\in\HtrimPoly{k}(T). \label{eq:poly.trace.H}
  \end{alignat}
\end{lemma}

\subsection{Integration by parts formulas}

A key element for the DDR-inspired construction are the integration by parts formulas collected in this section, which are used both to identify the components of the discrete spaces and to reconstruct the discrete differentials and the corresponding potentials.

\subsubsection{Integration by parts formulas for $\Hdevgrad{T}{\Real^3}$}

Let $T\in\Th$ and let $\bvec{v}:T:\to\Real^3$ be a vector-valued function, which we assume as smooth as needed in what follows.
For all $\btens{\sigma}:T\to\Tless$ smooth enough, it holds
\begin{equation} \label{eq:IPP.DG.T}
  \begin{aligned}
    \int_T\DEV\GRAD{\bvec{v}} : \btens{\sigma}
    &=
    -\int_T\bvec{v}\cdot\VDIV\btens{\sigma}
    + \sum_{F\in\FT}\omega_{TF}\int_F \bvec{v}^\top \btens{\sigma} \nF \\
    &= -\int_T \bvec{v} \cdot \VDIV\btens{\sigma}
    + \sum_{F \in \FT} \wTF \int_F \left(
    \trInF{\bvec{v}}^\top \btens{\sigma} \nF
    + \trItF{\bvec{v}}^\top \btens{\sigma} \nF
    \right),
  \end{aligned}
\end{equation}
where we have used the decomposition $\bvec{v} = \trInF{\bvec{v}} + \trItF{\bvec{v}}$ of the trace of $\bvec{v}$ on $F$ into its normal and tangential components to pass to the second line.

Let now $F \in \FT$.
For all $\bvec{w} : F \to \Real^2$ smooth enough and valued in the tangent space of $F$, it holds
\begin{equation} \label{eq:IPP.DG.Ftn}
  \begin{aligned}
    \int_F \trntF{(\DEV\GRAD \bvec{v})} \cdot \bvec{w}
    &= \int_F \trntF{(\GRAD \bvec{v})} \cdot \bvec{w}
    =  \int_F \GRADF (\bvec{v}\cdot\nF) \cdot \bvec{w}
    \\
    &= - \int_F (\bvec{v}\cdot\nF) \DIVF \bvec{w} + \sum_{E \in \EF} \wFE \int_E (\bvec{v}  \cdot \nF)(\bvec{w}\cdot\nFE),
  \end{aligned}
\end{equation}
where, we have used the fact that the components extracted by $\trntF{(\DEV\GRAD \bvec{v})}$ are not affected by the $\DEV$ operator in the first equality
and a standard integration by parts on $F$ to conclude.
For all $\btens{\sigma} : F \to \Real^{2\times2}$ smooth enough and matrix-valued in the tangent space of $F$, on the other hand, we have
\begin{equation} \label{eq:IPP.DG.Ftt}
  \begin{aligned}
    \int_F \trttF{(\DEV\GRAD\bvec{v})} \tdot \bvec{\sigma} 
    &= \int_F \GRADF \trtF{\bvec{v}} \tdot \btens{\sigma} - \frac{1}{3} \int_F \tr (\GRAD \bvec{v}) \ID{2} \tdot \btens{\sigma}\\
    &= - \int_F \trtF{\bvec{v}} \cdot \VDIVF\btens{\sigma}
    - \frac{1}{3} \int_F \DIV \bvec{v}~ \tr \btens{\sigma}\\
    &\quad+ \sum_{E \in \EF} \wFE \int_E \left[
      (\bvec{v} \cdot \tE) \tE^\top\btens{\sigma}\nFE
      +(\bvec{v} \cdot \nFE) \nFE^\top\btens{\sigma}\nFE
      \right],
  \end{aligned}
\end{equation}
where, in the second equality, we have used an integration by parts for the first term and decomposed the tangent trace of $\bvec{v}$ on $F$
as $\trtF{\bvec{v}} = (\bvec{v}\cdot\tangent_E)\tangent_E + (\bvec{v}\cdot\normal_{FE})\normal_{FE}$ 
after noticing that $(\tangent_E,\normal_{FE})$ forms an orthonormal basis of the plane orthogonal to $E$ at each point of $E$.

\subsubsection{Integration by parts formulas for $\Hsymcurl{T}{\Tless}$}

Let $T \in \Th$.
For all $\btens{\tau}:T:\to\Tless$ and $\btens{\sigma}:T\to\Symm$ smooth enough, it holds
\begin{equation} \label{eq:IPP.SC.T}
  \begin{aligned}
    \int_T\SYM\CURL{\btens{\tau}}:\btens{\sigma}
    =
    \int_T\CURL{\btens{\tau}}:\btens{\sigma}
    &=
    \int_T\btens{\tau}:\CURL \btens{\sigma}
    - \sum_{F\in\FT}\omega_{TF}\int_F (\btens{\tau} \times \nF) : \btens{\sigma}\\
    &=
    \int_T\btens{\tau}:\CURL \btens{\sigma}
    + \sum_{F\in\FT}\omega_{TF}\int_F \btens{\tau} : (\btens{\sigma} \times \nF).
  \end{aligned}
\end{equation}

Let now $F\in\FT$.
We have, for $r:F\to\Real$ smooth enough,
\begin{equation} \label{eq:IPP.SC.Fnn}
  \int_F \trnnF{(\SYM\CURL\btens{\tau})}\, r
  = \int_F \ROTF \trntF{\tau}\, r
  = \int_F \trntF{\tau} \cdot \CURLF r - \sum_{E \in \EF} \wFE \int_E (\nF^\top\btens{\tau}\tE)\, r,
\end{equation}
where we have used the fact that the component extracted by $\trnnF{(\SYM\CURL\btens{\tau})}$ is on the diagonal, hence it is not affected by the $\SYM$ operator 
(so that, in particular, $\trnnF{(\SYM\CURL\btens{\tau})} = \trnnF{(\CURL\btens{\tau})}  = \ROTF \trntF{\tau}$).

For the tangential-tangential component of $\btens{\tau}$, standard integration by parts formulas on faces (corresponding, respectively, to \cite[Eqs.~(3.12) and (3.15)]{Di-Pietro.Droniou.ea:20}) give:
\begin{equation} \label{eq:IPP.SC.Ftt}
  \begin{aligned}
    \int_F \DIVF \VROTF \trttF{\tau} r 
    &= -\int_F \VROTF \trttF{\tau} \cdot \GRADF r
    + \sum_{E \in \EF} \wFE \int_E (\VROTF \trttF{\tau} \cdot \nFE) r \\
    &= -\int_F \trttF{\tau} \tdot \CURLF \GRADF r
    + \sum_{E \in \EF} \wFE \int_E  
    (\trttF{\tau} \tE)\cdot\GRADF r
    \\
    &\quad
    + \int_E (\VROTF \trttF{\tau} \cdot \nFE)\, r.
  \end{aligned}
\end{equation}
For all $E \in \EF$, it holds:
\begin{equation} \label{eq:IPP.SC.E1}
  \begin{aligned}
    \int_E (\trttF{\tau} \tE)\cdot\GRADF r
    &= \int_E\left[
      (\tE^\top  \btens{\tau}  \tE)\, \partial_{\tE} r 
      + (\nFE^\top  \btens{\tau}  \tE)\, \partial_{\nFE} r
      \right]
    \\
    &=
    \int_E\left[
      - \left ( \nFE^\top  \btens{\tau}  \nFE + \nF^\top  \btens{\tau}  \nF \right )\, \partial_{\tE} r
      +  (\nFE^\top  \btens{\tau}  \tE)\,\partial_{\nFE} r
      \right], 
  \end{aligned}
\end{equation}
where we have
  written $\GRADF r = \dotp{r}\, \tE + \partial_{\nFE} r\, \nFE$ in the first equality and
used the fact that $\btens{\tau}$ is traceless on the last line (so that, the trace being an invariant and $(\tangent_E,\normal_{FE},\normal_F)$ an orthonormal basis of $\Real^3$,
$\tE^\top  \btens{\tau}  \tE + \nFE^\top  \btens{\tau}  \nFE + \nF^\top  \btens{\tau}  \nF = 0$).
Moreover, we have
\begin{multline} \label{eq:IPP.SC.E2}
  \int_E (\VROTF \trttF{\tau} \cdot \nFE)\, r
  =  \int_E \ROTF (\nFE^\top \trttF{\tau})\, r
  = \int_E\left[
    \partial_{\tE} (\nFE^\top  \btens{\tau}  \nFE)\, r
    - \partial_{\nFE} (\nFE^\top  \btens{\tau}  \tE)\, r
    \right]
  \\
  =
  - \int_E (\nFE^\top  \btens{\tau}  \nFE)\,\partial_{\tE}r
  +  \jump{E}{(\nFE^\top  \btens{\tau}  \nFE)\, r}
  - \int_E\left[
    \nFE^\top   \GRAD (\btens{\tau}  \tE )  \nFE
    \right] r,
\end{multline}
where we have used an integration by parts on the first term to conclude.
Plugging \eqref{eq:IPP.SC.E1} and \eqref{eq:IPP.SC.E2} into \eqref{eq:IPP.SC.Ftt} finally gives
\begin{equation} \label{eq:IPP.SC.Ftt2}
  \begin{aligned}
    &\int_F \DIVF \VROTF \trttF{\tau} r 
    = -\int_F \trttF{\tau} \tdot \CURLF \GRADF r
    + \sumEF (\nFE^\top  \btens{\tau}  \tE)\,\partial_{\nFE} r \\
    &\qquad -\sumEF \left (2 \nFE^\top  \btens{\tau}  \nFE + \nF^\top  \btens{\tau}  \nF \right ) \partial_{\tE} r
    - \sumEF \left[
      \nFE^\top   \GRAD (\btens{\tau}  \tE )  \nFE
      \right] r  \\
    &\qquad + \sum_{E\in\EF}\wFE \jump{E}{(\nFE^\top  \btens{\tau}  \nFE)\, r}.
  \end{aligned}
\end{equation}

\subsubsection{Integration by parts formulas for $\Hdivdiv{T}{\Symm}$}

Let $T \in \Th$, for all $\btens{\upsilon}:T:\to\Symm$ and $v:T\to\Real$ smooth enough, it holds (cf. \cite[Lemma~4.1]{Chen.Huang:22})
\begin{multline} \label{eq:IPP.DD.T}
  \int_T\DIV\VDIV\btens{\upsilon}~v
  =
  \int_T\btens{\upsilon}:\HESS v
  - \sum_{F\in\FT}\omega_{TF}\sum_{E\in\EF}\omega_{FE}\int_E(\normal_{FE}^\top\btens{\upsilon}\normal_F)~v
  \\
  - \sum_{F\in\FT}\omega_{TF}\int_F\upsilon_{\normal\normal,F}~\partial_{\normal_F} v
  - \sum_{F\in\FT}\omega_{TF}\int_F\left[
    2\DIV_F(\trntF{\upsilon}) + \partial_{\normal_F}\upsilon_{\normal\normal,F}
    \right]~v.
\end{multline}


\section{Discrete divdiv complex}\label{sec:discrete.complex}

Throughout the rest of this work, we fix an integer $k\ge 0$ corresponding to the polynomial degree of the discrete complex.
The focus of this section is on the construction of the local DDR complex mimicking \eqref{eq:complex} on a mesh element $T\in\Th$ and the study of its algebraic properties.
The analytical properties for the divdiv operator are studied in Section \ref{sec:biharmonic} in the context of an application to a biharmonic problem.
An in-depth study of the analytical properties of the other spaces and operators is postponed to a future work.

\subsection{Local discrete spaces}

The discrete counterparts of the spaces $\Hdevgrad{T}{\Real^3}$, $\Hsymcurl{T}{\Tless}$, and $\Hdivdiv{T}{\Symm}$ are, respectively, the spaces $\uHdevgrad{T}$, $\uHsymcurl{T}$, and $\uHdivdiv{T}$ defined as follows:
\[
\begin{aligned}
  \uHdevgrad{T}
  \coloneq\bigg\{
  &\uvec{v}_T = \big(
  \bvec{v}_T, %
  (\dofnF{v}, \doftF{v},\gdofF{v})_{F\in\FT}, %
  (\doftE{v}, \dofnE{v}, \gdofE{v})_{E\in\ET}, %
  (\bvec{v}_V, \gdofV{v})_{V\in\VT}
  \big)\st\\
  &\qquad
  \text{$\bvec{v}_T \in \vPoly{k-1}(T;\Real^3)$,}\\
  &\qquad
  \text{
    $\dofnF{v} \in \Poly{k}(F)$, $\doftF{v} \in \vPoly{k-1}(F;\Real^2)$, 
    and $\gdofF{v} \in \Poly{k-1}(F)$ for all $F\in\FT$,
  }\\
  &\qquad
  \text{
    $\doftE{v} \in \Poly{k-1}(E)$, $\dofnE{v} \in \vPoly{k}(E; \Real^2)$ and $\gdofE{v} \in \nPoly{k}{E}$ for all $E\in\ET$,
  }\\
  &\qquad
  \text{
    $\bvec{v}_V\in\Real^{3}$ and $\gdofV{v} \in \Real^{3\times 3}$ for all $V\in\VT$
  }
  \bigg\},
\end{aligned}
\]
\[
\begin{aligned}
  \uHsymcurl{T}
  \coloneq\bigg\{
  &\utens{\tau}_T = \big(
  \srtvec{\tau},
  (\rttvec[F]{\tau},
  \cgtvec{\tau}
  )_{F\in\FT}, (\btens{\tau}_E, \doftE{\tau}, \cdofE{v})_{E\in\ET}, (\btens{\tau}_V)_{V\in\VT}
  \big)\st\\
  &\qquad
  \text{$\srtvec{\tau} \in \SRtrimPoly{k}(T)$,} \\
  &\qquad
  \text{$\rttvec[F]{\tau} \in \RT{k+1}(F)$ and $\cgtvec[F]{\tau} \in \CGtrimPoly{k}(F)$ for all $F\in\FT$,} \\
  &\qquad
  \text{$\btens{\tau}_E \in \nPoly{k}{E}$, $\doftE{\tau} \in \vPoly{k+1}(E;\Real^2)$ and 
    $\cdofE{\tau} \in \nPoly{k+1}{E}$ for all $E\in\ET$,} \\
  &\qquad
  \text{$\btens{\tau}_V \in \Tless$ for all $V\in\VT$}
  \bigg\},
\end{aligned}
\]
\[
\begin{aligned}
  \uHdivdiv{T}
  \coloneq\bigg\{
  &\utens{\upsilon}_T = \big(
  \htvec{\upsilon},
  (\upsilon_F, D_{\btens{\upsilon},F})_{F\in\FT},
  (\btens{\upsilon}_E)_{E\in\ET}
  \big)\st
  \\
  &\qquad
  \text{%
    $\htvec{\upsilon}\in\HtrimPoly{k}(T)$,
  }
  \\
  &\qquad
  \text{%
    $\upsilon_F\in\Poly{k+1}(F)$ and $D_{\btens{\upsilon},F}\in\Poly{k+1}(F)$ for all $F\in\FT$,
  }
  \\
  &\qquad
  \text{%
    and $\btens{\upsilon}_E\in\nSPoly{k+1}{E}$ for all $E\in\ET$
  }
  \bigg\}.
\end{aligned}
\]
The meaning of the polynomial components in these spaces is provided by the interpolators
$\IDGrad{T}: \bvec{C}^1(\overline{T}; \Real^3)\to\uHdevgrad{T}$,
$\ISCurl{T}: \btens{H}^3(T;\Tless)\to\uHsymcurl{T}$,
and $\IDivDiv{T}: \btens{H}^2(T;\Symm)\to\uHdivdiv{T}$ such that, for all $(\bvec{v},\btens{\tau},\btens{\upsilon})\in\bvec{C}^1(\overline{T}; \Real^3)\times \btens{H}^3(T;\Tless)\times \btens{H}^2(T;\Symm)$,
\begin{equation}\label{eq:IDGrad}
  \begin{aligned}
    \IDGrad{T}\bvec{v}
    \coloneq\Big(
    &\vlproj{k-1}{T}\bvec{v},
    \big(\lproj{k}{F}(\bvec{v}\cdot\nF),\vlproj{k-1}{F}(\trtF{\bvec{v}}),
    \lproj{k-1}{F}(\DIV \bvec{v})\big)_{F\in\FT},\\
    &\big(\lproj{k-1}{E}(\bvec{v} \cdot \tE),
    \vlproj{k}{E}(\trnE{\bvec{v}}),\tlproj{k}{E}(\trnnE{\GRAD \bvec{v})}\big)_{E\in\ET},
    \\
    &\big(\bvec{v}(\bvec{x}_V),{\GRAD \bvec{v}}(\bvec{x}_V)\big)_{V\in\VT}
    \Big),
  \end{aligned}
\end{equation}
\[ 
\begin{aligned}
  \ISCurl{T}\btens{\tau}
  \coloneq\Big(
  &\SRtrimproj{k}\btens{\tau},
  \big(\RTproj[F]{k+1}\trntF{\btens{\tau}},
  \CGtrimproj[F]{k}\trttF{\btens{\tau}}\big)_{F\in\FT},\\
  &\big(\tlproj{k}{E}\trnnE{\btens{\tau}},\vlproj{k+1}{E}\trntE{\btens{\tau}},\tlproj{k+1}{E}(\trnnE{(\GRAD (\btens{\tau}  \tE))}\big)_{E\in\ET},
  \big({\btens{\tau}}(\bvec{x}_V)\big)_{V\in\VT}
  \Big),
\end{aligned}
\] 
\begin{equation}\label{eq:IDivDiv}
  \IDivDiv{T}\btens{\upsilon}
  \coloneq\Big(
  \Htrimproj{k}\btens{\upsilon},
  \big(
  \lproj{k+1}{F}(\upsilon_{\normal\normal,F}),
  \lproj{k+1}{F}(2\DIV_F(\btens{\upsilon}_{|F}\normal_F) + \partial_{\normal_F}\upsilon_{\normal\normal,F})
  \big)_{F\in\FT},
  (\tlproj{k+1}{E}\btens{\upsilon})_{E\in\ET}
  \Big).
\end{equation}

In what follows, for $\bullet\in\{\DEV\GRAD,\SYM\CURL,\DIV\VDIV\}$ and any $Y\in\FT\cup\ET$, we denote by $\uvec{X}_{\bullet,Y}^k$ and $\uvec{I}_{\bullet,Y}^k$ the restrictions of $\uvec{X}_{\bullet,T}^k$ and $\uvec{I}_{\bullet,T}^k$ to $Y$, obtained collecting the polynomial components on $Y$ and its boundary.

\subsection{Local operators}

\subsubsection{Discrete devgrad operator}

Let $F\in\FT$.
The discrete counterpart of the normal-tangential component of the discrete devgrad operator is obtained mimicking \eqref{eq:IPP.DG.Ftn}. 
Specifically, we let $\DGntF:\uHdevgrad{F}\to\RT{k+1}(F)$ be such that,
for all $\uvec{v}_{F} \in \uHdevgrad{F}$ and all $\bvec{w} \in \RT{k+1}(F)$,
\begin{equation} \label{eq:defDGntF}
  \int_F \DGntF \uvec{v}_F \cdot \bvec{w}
  =
  - \int_F \dofnF{v} \DIVF \bvec{w}
  + \sum_{E \in \EF} \wFE \int_E 
  (\dofnE{v} \cdot \nF)
  (\bvec{w}\cdot\nFE).
\end{equation}

The discrete counterpart of the tangential-tangential component is, on the other hand,
obtained mimicking \eqref{eq:IPP.DG.Ftt}. 
Specifically, we let
$\DGttF:\uHdevgrad{F}\to\CGtrimPoly{k}(F)$ be such that, 
for all $\uvec{v}_F \in \uHdevgrad{F}$ and all $\btens{\sigma}\in\CGtrimPoly{k}(F)$,
\begin{equation}\label{eq:defDGttF}
  \begin{aligned}
    \int_F \DGttF\uvec{v}_F \tdot \btens{\sigma}
    &\quad=
    - \int_F \doftF{v} \cdot \VDIVF\btens{\sigma}
    - \frac13\int_F\gdofF{v}\tr\btens{\sigma}
    \\
    &\qquad
    + \sumEF 
    \doftE{v}
    \,\tE^\top\btens{\sigma}\nFE
    + \sumEF 
    (\dofnE{v} \cdot \nFE)
    \, \nFE^\top\btens{\sigma}\nFE.
  \end{aligned}
\end{equation}

Recalling \eqref{eq:IPP.DG.T}, the element devgrad operator $\DGT : \uHdevgrad{T}\to\SRtrimPoly{k}(T)$
is defined such that, for all $\uvec{v}_T \in \uHdevgrad{T}$ and all $\bvec{\sigma} \in \SRtrimPoly{k}(T)$,
\begin{equation}\label{eq:defDGT}
  \int_T \DGT\uvec{v}_T \tdot \btens{\sigma}
  = 
  -\int_T \bvec{v}_T \cdot \VDIV\btens{\sigma}
  + \sum_{F \in \FT} \wTF \int_F\left(
  \dofnF{v}\, \nF^\top \btens{\sigma} \nF + \hdoftF{v}^\top \btens{\sigma} \nF
  \right),
\end{equation}
where we remind the reader that $\hdoftF{v}$ denotes the injection of $\doftF{v}$ into $\Real^3$.

The discrete devgrad operator $\DG{T}:\uHdevgrad{T}\to\uHsymcurl{T}$ acting between spaces of the discrete complex is defined, for all $\uvec{v}_T \in \uHdevgrad{T}$, by
\begin{equation} \label{eq:defuDG}
  \begin{aligned}
    \DG{T} \uvec{v}_T
    \coloneq\bigg(
    &\DGT\uvec{v}_T,
    \Big(
    \DGntF\uvec{v}_F,
    \DGttF\uvec{v}_F
    \Big)_{F\in\FT},
    \\
    &\Big(
    \gdofE{v} - \tfrac13\big(
    \tr \gdofE{v} +
    \DerE{k}(\bvec{v}_{V_1} \cdot \tE, \bvec{v}_{V_2} \cdot \tE, \doftE{v})
    \big)
    \ID{2}, \\
    &\hphantom{\big(}
    \DerE{k+1}(\trnE{(\bvec{v}_{V_1})},\trnE{(\bvec{v}_{V_2})}, \dofnE{v}),
    \DerE{k+1}(\trnnE{(\gdofV[V_1]{v})},\trnnE{(\gdofV[V_2]{v})}, \gdofE{v})
    \Big)_{E\in\ET},
    \\
    &\left( \DEV \gdofV{v} \right)_{V\in\VT}
    \bigg).
  \end{aligned} 
\end{equation}
In what follows, the restriction of $\DG{T}$ to a face or edge $Y\in\FT\cup\ET$ of $T$, obtained collecting the components on $Y$ and its boundary, will be denoted by $\DG{Y}$.%

\begin{lemma}[Local commutation property]
  The following commutation property holds
  \begin{equation}\label{eq:DGT.commutativity}
    \DG{T}(\IDGrad{T}\bvec{v}) = \ISCurl{T}(\DEV \GRAD \bvec{v})\qquad
    \forall \bvec{v} \in \bvec{H}^3(T;\Real^3).
  \end{equation}
\end{lemma}

\begin{proof}
  Let $\bvec{v} \in \bvec{H}^3(T;\Real^3)$. 
  Let us check the relation on the normal-normal component on edges.
  By \eqref{eq:DerE:commutation.property} with $\ell = k$, 
  we have $\DerE{k}(\bvec{v}(\bvec{x}_{V_1}) \cdot \tE, \bvec{v}(\bvec{x}_{V_2}) \cdot \tE, \lproj{k-1}{E}(\bvec{v} \cdot \tE)) = \lproj{k}{E}(\dotp{\bvec{v}} \cdot \tE)$.
  Accounting for the previous relation
  we have, for all $\btens{\sigma}\in\tPoly{k}(E;\Real^{2\times2})$,
  \begin{multline} \label{eq:DGT.com.1}
    \int_E \left[
      \cancel{\tlproj{k}{E}}\trnnE{(\GRAD \bvec{v})} 
      - \frac13\left(\tr\cancel{\tlproj{k}{E}}\trnnE{(\GRAD \bvec{v})} +
      \cancel{\lproj{k}{E}}(\dotp{\bvec{v}} \cdot \tE)
      \right)\ID{2}
      \right] \tdot \btens{\sigma}
    \\
    = \int_E \trnnE{(\DEV\GRAD \bvec{v})} \tdot \btens{\sigma},
  \end{multline}
  where the cancellation of the projectors is made possible by their definition.
  Using again \eqref{eq:DerE:commutation.property}, this time with $\ell = k+1$, we infer that 
  \begin{align} \label{eq:DGT.com.2}
    \DerE{k+1}(\trnE{\bvec{v}(\bvec{x}_{V_1})}, \trnE{\bvec{v}(\bvec{x}_{V_2})}, \vlproj{k}{E}(\trnE{\bvec{v}}))
    ={}& \vlproj{k+1}{E}\trntE{(\DEV\GRAD\bvec{v})},\\
    \DerE{k+1}(\trnnE{(\GRAD\bvec{v}(\bvec{x}_{V_1}))}, \trnnE{(\GRAD\bvec{v}(\bvec{x}_{V_2}))}, 
    \tlproj{k}{E}\trnnE{(\GRAD \bvec{v})})
    ={}& \tlproj{k+1}{E}\trnnE{(\GRAD ((\DEV\GRAD\bvec{v}) \tE))}. \label{eq:DGT.com.3}
  \end{align}
  Combining \eqref{eq:DGT.com.1}, \eqref{eq:DGT.com.2} and \eqref{eq:DGT.com.3}, we obtain
  $\DG{E}(\IDGrad{E}\bvec{v}) = \ISCurl{E}(\DEV \GRAD \bvec{v})$.
  The commutation for the components of $\DG{T}$ on faces and on the element are proved in a similar fashion:
  first by removing the projections in the definition \eqref{eq:defDGntF} (respectively, \eqref{eq:defDGttF} and \eqref{eq:defDGT}), 
  and then concluding with the integration by parts formula \eqref{eq:IPP.DG.Ftn} (respectively, \eqref{eq:IPP.DG.Ftt} and \eqref{eq:IPP.DG.T}).
\end{proof}

\subsubsection{Discrete symcurl operator}

The discrete symmetric curl operator on edges $\SCE: \uHsymcurl{E}\to\nSPoly{k+1}{E}$ is defined,
for all $\uvec{\tau}_E \in \uHsymcurl{E}$, by:
\begin{equation} \label{eq:defSCE}
  \SCE\utens{\tau}_E \coloneq
  \Ctensor\left(\cdofE{\tau} -
  \DerE{k+1}(\trnnE{(\btens{\tau}_{V_1})},\trnnE{(\btens{\tau}_{V_2})}, \btens{\tau}_E)
  \right),
\end{equation}
where $\Ctensor$ is the fourth-order tensor such that
\begin{equation} \label{eq:defCtensor}
  \Ctensor \begin{pmatrix}
    \eta_{11} & \eta_{12} \\ \eta_{21} & \eta_{22}
  \end{pmatrix}
  = \begin{pmatrix}
    \eta_{12} & \frac{-\eta_{11} + \eta_{22}}{2} \\ \frac{-\eta_{11} + \eta_{22}}{2} & - \eta_{21}
  \end{pmatrix}\qquad\forall\btens{\eta}\in\Real^{2\times 2}.
\end{equation}

There are two components for the symmetric curl operator on faces.
The first one, $\SCnnF:\uHsymcurl{F}\to\Poly{k+1}(F)$, is defined mimicking \eqref{eq:IPP.SC.Fnn}:
For all $\uvec{\tau}_F \in \uHsymcurl{F}$ and all $r \in \Poly{k+1}(F)$,
\begin{equation} \label{eq:defSCnnF}
  \int_F \SCnnF\utens{\tau}_F \, r
  =
  \int_F \rttvec[F]{\tau} \cdot \CURLF r - \sum_{E \in \EF} \wFE \int_E (\doftE{\tau} \cdot \nF)\, r,
\end{equation}
while the second one, $\SCddofF:\uHsymcurl{F}\to\Poly{k+1}(F)$, is defined mimicking \eqref{eq:IPP.SC.Ftt2}:
For all $\uvec{\tau}_F \in \uHsymcurl{F}$ and all $r \in \Poly{k+1}(F)$,
\begin{equation} \label{eq:defSCddofF}
  \begin{aligned}
    \int_F \SCddofF\utens{\tau}_F\, r
    &=
    -\int_F \cgtvec[F]{\tau} \tdot \CURLF \GRADF r + 
    \sumEF (\doftE{\bvec{\tau}} \cdot \nFE)\, \partial_{\nFE} r
    \\
    &\quad
    - \sumEF \left (2 \nFE^\top \btens{\tau}_E \nFE + \nF^\top \btens{\tau}_E \nF \right)\, \partial_{\tE} r
    - \sumEF (\nFE^\top \cdofE{\tau} \, \nFE)\,r
    \\
    &\quad
    + \sum_{E\in\EF}\wFE\jump{E}{(\nFE^\top \, \btens{\tau}_V \, \nFE)\, r}.
  \end{aligned}
\end{equation}

Recalling \eqref{eq:IPP.SC.T}, the discrete symmetric curl operator on an element $T\in\Th$ is defined such that,
for all $\utens{\tau}_T \in \uHsymcurl{T}$ and all
$ \btens{\sigma} \in \HtrimPoly{k}(T)$,
\begin{equation} \label{eq:defSCT}
  \begin{aligned}
    \int_T \SCT\utens{\tau}_T \tdot \btens{\sigma}
    &\coloneq
    \int_T \srtvec{\tau} \tdot \CURL \btens{\sigma} 
    + \sum_{F \in \FT} \wTF \int_F \cgtvec[F]{\tau} \tdot \trttF{(\btens{\sigma} \times \nF)}
    \\
    &\quad
    + \sum_{F \in \FT} \wTF \int_F\rttvec[F]{\tau} \cdot \trntF{(\btens{\sigma}\times\nF)}.
  \end{aligned}
\end{equation}

Finally, the discrete symmetric curl operator $\SC{T}:\uHsymcurl{T}\to\uHdivdiv{T}$ acting between discrete spaces is defined, for all $\utens{\tau}_T \in \uHsymcurl{T}$, as the vector collecting the components defined above:
\begin{equation} \label{eq:defuSC}
  \SC{T}\utens{\tau}_T
  \coloneq\Big(
  \SCT \utens{\tau}_T,
  \big(
  \SCnnF\utens{\tau}_F,
  \SCddofF\utens{\tau}_F
  \big)_{F\in\FT},
  \big(
  \SCE\utens{\tau}_E
  \big)_{E\in\ET}
  \Big).
\end{equation}
It can be checked that the following commutation property with the interpolators holds:
\begin{equation} \label{eq:SCT.commutativity}
  \SC{T} \ISCurl{T} \btens{\tau} = \IDivDiv{T} (\SYM \CURL \btens{\tau})\qquad
  \forall \btens{\tau} \in \btens{H}^3(T;\Tless).
\end{equation}

\subsubsection{Discrete divdiv operator}

The discrete divdiv operator $\DD{T}:\uHdivdiv{T}\to\Poly{k}(T)$ is defined after \eqref{eq:IPP.DD.T}:
For all $\utens{\upsilon}_T\in\uHdivdiv{T}$ and all $v\in\Poly{k}(T)$,
\begin{equation}\label{eq:DT}
  \begin{aligned}
    \int_T\DD{T}\utens{\upsilon}_T~v
    &= \int_T\btens{\upsilon}_{\ctens{H},T}:\HESS v
    - \sum_{F\in\FT}\omega_{TF}\sum_{E\in\EF}\omega_{FE}\int_E(\normal_{FE}^\top\btens{\upsilon}_E\normal_F)~v
    \\
    &\quad
    -\sum_{F\in\FT}\omega_{TF}\int_F\upsilon_F~\partial_{\normal_F}v
    -\sum_{F\in\FT}\omega_{TF}\int_F D_{\btens{\upsilon},F}~v.  
  \end{aligned}
\end{equation}
By construction, it holds
\begin{equation}\label{eq:fortin:2}
  \DD{T}\IDivDiv{T}\btens{\upsilon}
  = \lproj{k}{T}(\DIV\VDIV\btens{\upsilon})\qquad
  \forall\btens{\upsilon}\in\btens{H}^2(T;\Symm),
\end{equation}
as can be checked using \eqref{eq:IDivDiv} in \eqref{eq:DT} written for $\utens{\upsilon}_T = \IDivDiv{T}\btens{\upsilon}$, cancelling the $L^2$-orthogonal projectors using their definitions,
and concluding with \eqref{eq:IPP.DD.T}.

\subsection{Local DDR complex and main results}
For a given mesh element $T\in\Th$, the spaces and operators defined above can be arranged to form the sequence
\begin{equation}\label{eq:ddr.complex}
  \begin{tikzcd}
    \RT{1}(T) \arrow[r,"\IDGrad{T}"]
    & \uHdevgrad{T} \arrow[r,"\DG{T}"]
    & \uHsymcurl{T}  \arrow[r,"\SC{T}"]
    & \uHdivdiv{T} \arrow[r,"\DD{T}"]
    & \Poly{k}(T) \arrow[r,"0"]
    & \lbrace 0 \rbrace.
  \end{tikzcd}
\end{equation}

\begin{theorem}[Local complex property and exactness]\label{th:local.complex}
  The sequence \eqref{eq:ddr.complex} forms a complex which is exact if the topology of $T$ is trivial and if $k \geq 1$.
\end{theorem}

\begin{remark}[Role of the condition $k\ge 1$]
  The condition $k \geq 1$ is only required for the exactness of the 
    tail of the complex (see \eqref{eq:exact.SCDD} and Remark \ref{rem:SCDD.k=0} below).
  The head of the complex is exact also for $k = 0$ (cf., in particular, \eqref{eq:exact.RTDG} and \eqref{eq:exact.DGSC} below).
\end{remark}

\begin{proof}
  The fact that the sequence \eqref{eq:ddr.complex} forms a complex is a consequence of the following relations:
  \begin{align} \label{eq:IGrad.RT.subset.DGT}
    \IDGrad{T}\RT{1}(T) &\subset \Ker \DG{T},
    \\ \label{eq:LC.DGSC}
    \Image \DG{T} &\subset \Ker \SC{T},
    \\ \label{eq:LC.SCDD}
    \Image \SC{T} &\subset \Ker \DD{T},
    \\ \label{eq:Image.DDT.subset.PolykT}
    \Image \DD{T} &= \Poly{k}(T).
  \end{align}
  The inclusion \eqref{eq:IGrad.RT.subset.DGT} is a straightforward consequence of the commutation  property \eqref{eq:DGT.commutativity} along with the fact that $\DEV\GRAD\RT{1}(T) = \bvec{0}$.
  The relation \eqref{eq:Image.DDT.subset.PolykT} classically follows from the surjectivity of $\DIV\VDIV:\Hdivdiv{T}{\Symm}\to L^2(T)$ along with \eqref{eq:fortin:2} (a more detailed argument is provided in Lemma \ref{lem:inf-sup:DDh} below for its global counterpart).
  Finally, properties \eqref{eq:LC.DGSC} and \eqref{eq:LC.SCDD} are proved in Section \ref{sec:complex.property}.

  The exactness of the complex when $T$ has a trivial topology translates into the following properties:
  \begin{subequations}\label{eq:exactness}
    \begin{alignat}{2} \label{eq:exact.RTDG}
      \IDGrad{T} \RT{1}(T) &= \Ker \DG{T}
      \\ \label{eq:exact.DGSC}
      \Image \DG{T} &= \Ker \SC{T}
      \\ \label{eq:exact.SCDD}
      \Image \SC{T} &= \Ker \DD{T}
      &\qquad& \text{if $k \geq 1$.}
    \end{alignat}
  \end{subequations}
  These properties are proved in Section~\ref{sec:exactness}.
\end{proof}


\section{A mixed method for biharmonic problems}\label{sec:biharmonic}

In this section we consider the application of the spaces at the tail of the above complex to the mixed discretization of the biharmonic problem \eqref{eq:weak}.
  Throughout this section, $k\ge 0$ is an integer corresponding to the polynomial degree of the scheme.

\subsection{Local component product}

We furnish $\uHdivdiv{T}$ with the component inner product such that, for all $(\utens{\upsilon}_T,\utens{\tau}_T)\in\uHdivdiv{T}\times\uHdivdiv{T}$,
\begin{equation}\label{eq:component.L2.prod}
  \lbrack\utens{\upsilon}_T,\utens{\tau}_T\rbrack_{\DIV\VDIV,T}
  \coloneq
  \int_T \btens{\upsilon}_{\ctens{H},T}:\btens{\tau}_{\ctens{H},T}
  + h_T\sum_{F\in\FT}\int_F\left(
  \upsilon_F~\tau_F + h_T^2 D_{\btens{\upsilon},F}~D_{\btens{\tau},F}
  \right)
  + h_T^2\sum_{E\in\ET}\int_E\btens{\upsilon}_E:\btens{\tau}_E.
\end{equation}
and we introduce the corresponding component norm such that
\begin{equation}\label{eq:tnorm}
  \tnorm{\DIV\VDIV,T}{\utens{\tau}_T}
  \coloneq
  \lbrack\utens{\tau}_T,\utens{\tau}_T\rbrack^{\frac12}
  \qquad\forall\utens{\tau}_T\in\uHdivdiv{T}.
\end{equation}

Using the boundedness of $L^2$-orthogonal projectors along with continuous trace inequalities on the faces and edges of $T$, it can be proved, similarly to \cite[Lemma 6]{Di-Pietro.Droniou:23}, that
\begin{equation}\label{eq:fortin:1}
  \tnorm{\DIV\VDIV,T}{\IDivDiv{T}\btens{\tau}}
  \lesssim
  \norm{\btens{L}^2(T;\Matr)}{\btens{\tau}}
  + h_T\seminorm{\btens{H}^1(T;\Matr)}{\btens{\tau}}
  + h_T^2\seminorm{\btens{H}^2(T;\Matr)}{\btens{\tau}}\qquad
  \forall\btens{\tau}\in\btens{H}^2(T;\Symm).
\end{equation}
Moreover, for all $\utens{\tau}_T\in\uHdivdiv{T}$, taking $v = \DD{T}\utens{\tau}_T$ in \eqref{eq:DT} and using Cauchy--Schwarz, discrete inverse and trace inequalities along with the definition \eqref{eq:tnorm} of $\tnorm{\DIV\VDIV,T}{{\cdot}}$, we get the following boundedness property:
\begin{equation}\label{eq:DD:boundedness}
  \norm{L^2(T)}{\DD{T}\utens{\tau}_T}
  \lesssim h_T^{-2}\tnorm{\DIV\VDIV,T}{\utens{\tau}_T}
  \qquad\forall\utens{\tau}_T\in\uHdivdiv{T}.
\end{equation}

\subsection{Discrete symmetric matrix potential}

In order to reconstruct a symmetric matrix potential, we first need to reconstruct face traces of degree $k$.
To this purpose, for each $F\in\FT$, we can apply the principles of the HHO potential reconstruction (see, e.g., \cite[Section 5.1.3]{Di-Pietro.Droniou:20}) to devise, from the vector of polynomials $(\upsilon_F, (\normal_F^\top\btens{\upsilon}_E\normal_F)_{E\in\ET})\in\Poly{k-1}(F)\times\left(\bigtimes_{E\in\EF}\Poly{k}(E)\right)$, a function in $\Poly{k}(F)$ that can be interpreted as the normal-normal trace of a symmetric matrix-valued field on $T$ (in passing, with these values one could actually compute a normal-normal trace in $\Poly{k+1}(F)$, but this will not be needed in what follows).
The corresponding reconstruction operator $\gammaF:\uHdivdiv{F}\to\Poly{k}(F)$ is, by construction, polynomially consistent: 
For all $\btens{\upsilon}\in\tPoly{k}(T;\Symm)$, $\gammaF\IDivDiv{F}\btens{\upsilon}_{|F} = \normal_F^\top\btens{\upsilon}_{|F}\normal_F$.
Moreover, the following boundedness property holds:
\begin{equation}\label{eq:gammaF:boundedness}
  \norm{L^2(F)}{\gammaF\utens{\upsilon}_F}
  \lesssim h_F^{-\frac12}\tnorm{\DIV\VDIV,T}{\utens{\upsilon}_T}
  \qquad\forall\utens{\upsilon}_T\in\uHdivdiv{T}.
\end{equation}

The symmetric matrix potential $\TP:\uHdivdiv{T}\to\tPoly{k}(T;\Symm)$ is then defined, mimicking \eqref{eq:IPP.DD.T},  such that, for all $\utens{\upsilon}_T\in\uHdivdiv{T}$:
For all $(v,\btens{\tau})\in\Poly{k+2}(T)\times\cHoly{k}(T)$,
\begin{multline}\label{eq:TP}
  \int_T\TP\utens{\upsilon}_T:(\HESS v + \btens{\tau})
  = \int_T\DD{T}\utens{\upsilon}_T~v
  + \sum_{F\in\FT}\omega_{TF}\sum_{E\in\EF}\omega_{FE}\int_E(\normal_{FE}^\top\btens{\upsilon}_E\normal_F)~v
  \\
  + \sum_{F\in\FT}\omega_{TF}\int_F\gammaF\utens{\upsilon}_F~\partial_{\normal_F}v
  - \sum_{F\in\FT}\omega_{TF}\int_F D_{\btens{\upsilon},F}~v
  + \int_T\btens{\upsilon}_{\ctens{H},T}:\btens{\tau}.
\end{multline}
By construction, the following polynomial consistency property holds:
\begin{equation}\label{eq:TP:polynomial.consistency}
  \TP\IDivDiv{T}\btens{\upsilon} = \btens{\upsilon}\qquad
  \forall\btens{\upsilon}\in\tPoly{k}(T;\Symm).
\end{equation}
Additionally, taking in \eqref{eq:TP} $(v,\btens{\tau})$ such that $\HESS v + \btens{\tau} = \TP\utens{\upsilon}_T$ 
(this is possible by virtue of the direct decomposition \eqref{eq:vPoly=Holy+cHoly}), 
using Cauchy--Schwarz, discrete trace, and inverse inequalities, 
and invoking the boundedness \eqref{eq:DD:boundedness} of $\DD{T}$ and \eqref{eq:gammaF:boundedness} of $\gammaF$, it is inferred:
\begin{equation}\label{eq:TP:boundedness}
  \norm{\btens{L}^2(T;\Matr)}{\TP\utens{\upsilon}_T}
  \lesssim\tnorm{\DIV\VDIV,T}{\utens{\upsilon}_T}
  \qquad\forall\utens{\upsilon}_T\in\uHdivdiv{T}.
\end{equation}

\begin{remark}[Polynomially consistent symmetric curl]
  For all $\btens{\tau}\in\tPoly{k+1}(T;\Tless)$, noticing that $\sym\CURL\btens{\tau}\in\tPoly{k}(T;\Symm)$, it holds
  \[
  \TP(\SC{T}\ISCurl{T}\btens{\tau})
  \overset{\eqref{eq:SCT.commutativity}}{=} \TP\IDivDiv{T}(\sym\CURL\btens{\tau})
  \overset{\eqref{eq:TP:polynomial.consistency}}{=} \sym\CURL\btens{\tau},
  \]
  showing that $\TP\circ\SC{T}$ provides a polynomially consistent approximation of the symmetric curl inside $T$.
  A similar construction can be repeated to obtain a consistent approximation of dev grad.
  Since this construction is not needed in the present discussion, we leave the details for a future work.
\end{remark}

\subsection{Global Hessian space, reconstructions, and discrete $L^2$-product}

A global space $\uHdivdiv{h}$ on the mesh $\Mh$ is obtained patching together the local spaces by enforcing the single-valuedness of the unknowns attached to edges and faces shared by multiple elements.
The global divdiv operator $\DD{h}:\uHdivdiv{h}\to\Poly{k}(\Th)$
and symmetric matrix potential operator $\TP[h]:\uHdivdiv{h}\to\tPoly{k}(\Th;\Symm)$ (with $\tPoly{k}(\Th;\Symm)$ symmetric matrix-valued version of the broken polynomial space \eqref{eq:broken.Pl}) are such that, for all $\utens{\upsilon}_h\in\uHdivdiv{h}$,
\[
\text{
  $(\DD{h}\utens{\upsilon}_h)_{|T}\coloneq\DD{T}\utens{\upsilon}_T$\quad
  and\quad $(\TP[h]\utens{\upsilon}_h)_{|T}\coloneq\TP\utens{\upsilon}_T$\quad
  for all $T\in\Th$.
}
\]
We define the following $L^2$-like product:
For all $(\utens{\upsilon}_h,\utens{\tau}_h)\in\uHdivdiv{h}$,
\[
(\utens{\upsilon}_h,\utens{\tau}_h)_{\DIV\VDIV,h}
\coloneq\sum_{T\in\Th}(\utens{\upsilon}_T,\utens{\tau}_T)_{\DIV\VDIV,T},
\]
where
\begin{equation}\label{eq:L2.prod:T}
  (\utens{\upsilon}_T,\utens{\tau}_T)_{\DIV\VDIV,T}
  \coloneq\int_T\TP\utens{\upsilon}_T:\TP\utens{\tau}_T
  + s_T(\utens{\upsilon}_T,\utens{\tau}_T).
\end{equation}
Above, $s_T$ is a symmetric positive semi-definite stabilisation bilinear form that ensures the positivity of $(\cdot,\cdot)_{\DIV\VDIV,T}$ while preserving polynomial consistency.
A possible expression for $s_T$ is the following:
\begin{equation}\label{eq:sT}
  s_T(\utens{\upsilon}_T,\utens{\tau}_T)
  = \lbrack\IDivDiv{T}\TP\utens{\upsilon}_T - \utens{\upsilon}_T,
  \IDivDiv{T}\TP\utens{\tau}_T - \utens{\tau}_T
  \rbrack_{\DIV\VDIV,T}.
\end{equation}
The following polynomial consistency property easily follows from \eqref{eq:TP:polynomial.consistency} and \eqref{eq:sT}:
\begin{equation}\label{eq:sT:polynomial.consistency}
  s_T(\IDivDiv{T}\btens{\upsilon}, \utens{\tau}_T) = 0
  \qquad\forall(\btens{\upsilon},\utens{\tau}_T)\in\tPoly{k}(T;\Symm)\times\uHdivdiv{T}.
\end{equation}

\begin{remark}[Difference between the component and discrete $L^2$-products]
  The main difference between the local component $L^2$-product defined by \eqref{eq:component.L2.prod} and the local discrete $L^2$-product defined by \eqref{eq:L2.prod:T} is that the latter is consistent whenever its arguments are interpolate of polynomial functions, i.e.,
  \[
  (\IDivDiv{T}\btens{\upsilon}, \IDivDiv{T}\btens{\tau})_{\DIV\VDIV,T}
  = \int_T\btens{\upsilon}:\btens{\tau}
  \qquad\forall(\btens{\upsilon},\btens{\tau})\in\tPoly{k}(T;\Symm).
  \]
\end{remark}

We close this section defining the norm induced by the $L^2$-product:
For $\bullet\in\Th\cup\{h\}$,
\begin{equation}\label{eq:discrete.L2.norm}
  \norm{\DIV\VDIV,\bullet}{\utens{\tau}_\bullet}
  \coloneq (\utens{\tau}_\bullet,\utens{\tau}_\bullet)_{\DIV\VDIV,\bullet}^{\nicefrac12}
  \qquad\forall\utens{\tau}_\bullet\in\uHdivdiv{\bullet}.
\end{equation}
The norm dual to $\norm{\DIV\VDIV,h}{{\cdot}}$ is denoted by $\norm{\DIV\VDIV,h,*}{{\cdot}}$.

\subsection{Discrete problem and main results}

Set, for the sake of brevity, $\mathcal{Z}_h^k\coloneq\uHdivdiv{h}\times\Poly{k}(\Th)$.
The discrete problem reads:
Find $(\utens{\sigma}_h,u_h)\in\mathcal{Z}_h^k$ such that
\begin{equation}\label{eq:discrete}
  \begin{alignedat}{4}
    (\utens{\sigma}_h,\utens{\tau}_h)_{\DIV\VDIV,h}
    + \int_\Omega\DD{h}\utens{\tau}_h~u_h
    &= 0
    &\qquad&\forall\utens{\tau}_h\in\uHdivdiv{h},
    \\
    -\int_\Omega\DD{h}\utens{\sigma}_h~v_h
    &= \int_\Omega f v_h
    &\qquad&\forall v_h\in\Poly{k}(\Th),
  \end{alignedat}
\end{equation}
or, equivalently:
Find $(\utens{\sigma}_h,u_h)\in\mathcal{Z}_h^k$ such that
\begin{equation}\label{eq:discrete:variational}
  \mathcal{A}_h((\utens{\sigma}_h,u_h),(\utens{\tau}_h,v_h))
  = \int_\Omega fv_h\qquad
  \forall(\utens{\tau}_h,v_h)\in\mathcal{Z}_h^k,
\end{equation}
with bilinear form $\mathcal{A}_h:\mathcal{Z}_h^k\times\mathcal{Z}_h^k\to\Real$ such that
\[
\mathcal{A}_h((\utens{\upsilon}_h,w_h),(\utens{\tau}_h,v_h))
\coloneq
(\utens{\upsilon}_h,\utens{\tau}_h)_{\DIV\VDIV,h}
+ \int_\Omega\DD{h}\utens{\tau}_h~w_h
-\int_\Omega\DD{h}\utens{\upsilon}_h~v_h.
\]
We state hereafter the main analysis results for the numerical scheme defined above.
To this purpose, we equip $\mathcal{Z}_h^k$ with the following norm:
\begin{equation}\label{eq:norm}
  \norm{\mathcal{Z},h}{(\utens{\tau}_h,v_h)}
  \coloneq\norm{\DIV\VDIV,h}{\utens{\tau}_h}
  + \norm{L^2(\Omega)}{v_h}
  \qquad\forall (\utens{\tau}_h,v_h)\in\mathcal{Z}_h^k.
\end{equation}

\begin{theorem}[Well-posedness]\label{thm:well-posedness}
  It holds
  \begin{equation}\label{eq:inf-sup:Ah}
    1\lesssim\inf_{(\utens{\upsilon}_h,w_h)\in\mathcal{Z}_h^k\setminus\{0\}}
    \sup_{(\utens{\tau}_h,v_h)\in\mathcal{Z}_h^k\setminus\{0\}}
    \frac{\mathcal{A}_h((\utens{\upsilon}_h,w_h),(\utens{\tau}_h,v_h))}{\norm{\mathcal{Z},h}{(\utens{\upsilon}_h,w_h)}\norm{\mathcal{Z},h}{(\utens{\tau}_h,v_h)}}.
  \end{equation}
  Moreover, problem \eqref{eq:discrete} (or, equivalently, \eqref{eq:discrete:variational}) admits a unique solution which satisfies
  \[
  \norm{\mathcal{Z},h}{(\utens{\sigma}_h,u_h)}
  \lesssim\norm{L^2(\Omega)}{f}.
  \]
\end{theorem}

\begin{proof}
  See Section \ref{sec:discrete.problem:well-posedness}.
\end{proof}

\begin{theorem}[Error estimate]\label{thm:error.estimate}
  Let $(\btens{\sigma},u)\in\Hdivdiv{\Omega}{\Symm}\times L^2(\Omega)$ denote the unique solution to the continuous problem \eqref{eq:weak}, and assume the additional regularity $\btens{\sigma}\in\btens{H}^2(\Omega;\Symm)\cap\btens{H}^{k+1}(\Th;\Symm)$ and $u\in H^{k+3}(\Th)$.
  Then, denoting by $(\utens{\sigma}_h,u_h)\in\uHdivdiv{h}\times\Poly{k}(\Th)$ the unique solution to the discrete problem \eqref{eq:discrete} (or, equivalently, \eqref{eq:discrete:variational}), it holds
  \begin{equation}\label{eq:error.estimate}
    \norm{\mathcal{Z},h}{(\utens{\sigma}_h - \IDivDiv{h}\btens{\sigma}, u_h - \lproj{k}{h} u)}
    \lesssim h^{k+1}\left(
    \seminorm{\btens{H}^{k+1}(\Th;\Matr)}{\btens{\sigma}}
    + \seminorm{H^{k+3}(\Th)}{u}
    \right).
  \end{equation}
\end{theorem}

\begin{proof}
  See Section \ref{sec:error.estimate}.
\end{proof}

\subsection{Numerical examples}
The numerical scheme \eqref{eq:discrete} was implemented using the HArDCore library (see \url{https://github.com/jdroniou/HArDCore}).
In order to validate the error estimate of Theorem \ref{thm:error.estimate}, we consider the following manufactured solution 
\[
u = x^2(1 - x)^2 y^2(1-y)^2 z^2(1-z)^2,\qquad
\btens{\sigma} = -\HESS u
\]
on the domain $\Omega = [0,1]^3$.
The method supports arbitrary polyhedral meshes, so we have considered three mesh sequences: cubic, tetrahedral (generated using Tetgen), and Voronoi.
In Figure \ref{fig:convrate3D} we depict, for each mesh sequence, the error measure in the left-hand side of \eqref{eq:error.estimate} as a function of the mesh size for polynomial degrees $k$ between 0 and 3.
The convergence plots show good agreement between the observed and predicted convergence rates.
For the Tetgen mesh family, a saturation of the error is observed for $k=0$.
A slight reduction of the convergence rate is also observed for the finest meshes of the Tetgen mesh family with $k=1$ and of the Voronoi mesh family for $k\in\{0,1\}$.
In both cases, however, the slope is still close to the theoretical one.

\begin{figure}
  \begin{center}
  \ref{br.conv.voro}\medskip
  \\
    \begin{minipage}[c]{0.30\columnwidth}\centering
      \includegraphics[width=0.9\columnwidth,height=0.2\textheight,keepaspectratio]{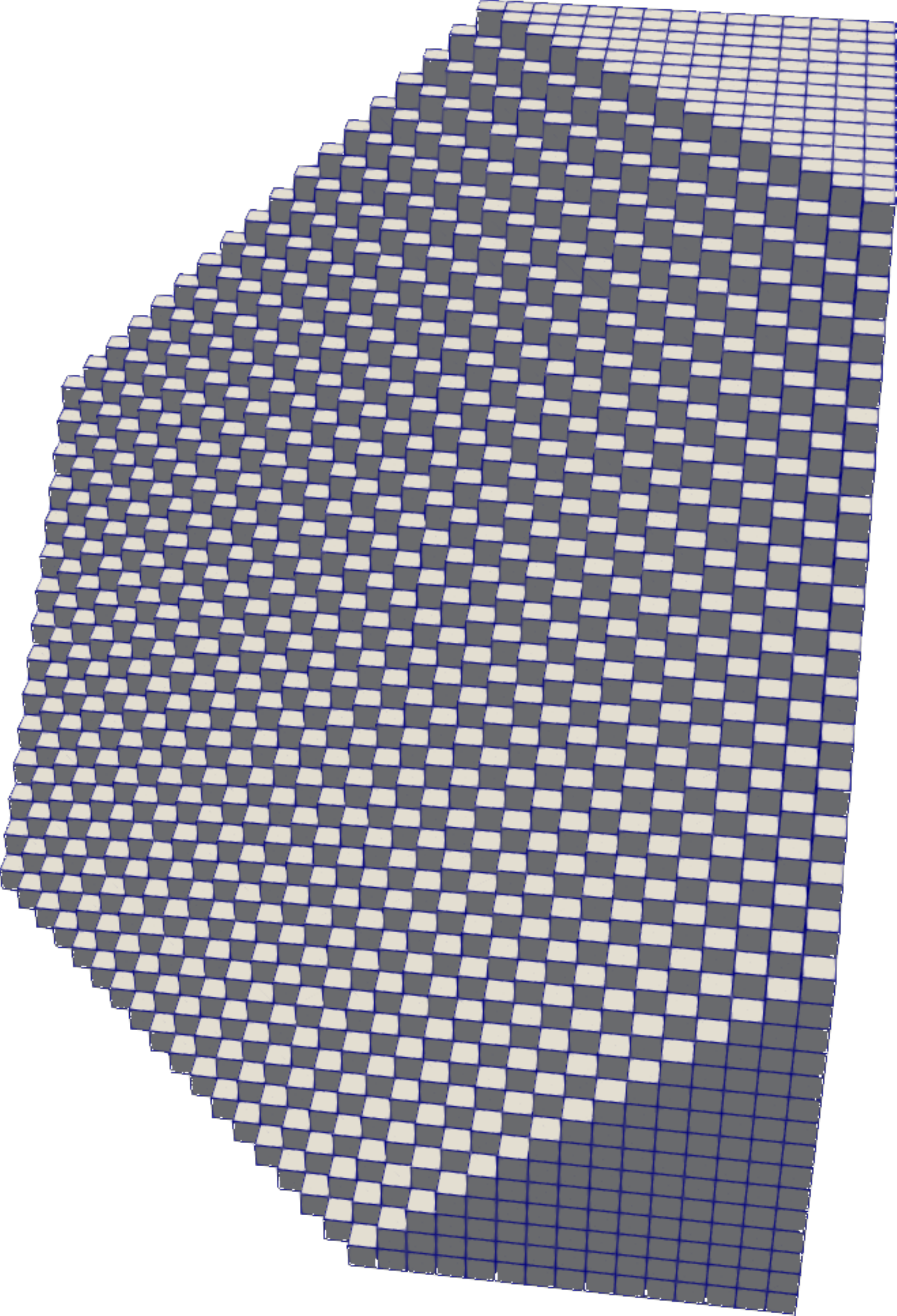}
    \end{minipage}
    \begin{minipage}[c]{0.50\columnwidth}\centering
      \begin{tikzpicture}
        \begin{loglogaxis}[legend columns=4, legend to name=br.conv.voro]
          \addplot +[mark=+, style=solid, color=blue] table[x index=0,y index=4] {data/3D/1/sol_1_k0_detailed.dat};
          \addplot +[mark=triangle, style=solid, color=red] table[x index=0,y index=4] {data/3D/1/sol_1_k1_detailed.dat};
          \addplot +[mark=square, style=solid, color=brown] table[x index=0,y index=4] {data/3D/1/sol_1_k2_detailed.dat};
          \addplot +[mark=pentagon, style=solid, color=darkgray] table[x index=0,y index=4] {data/3D/1/sol_1_k3_detailed.dat};
          \logLogSlopeTriangle{0.90}{0.25}{0.1}{1}{blue};
          \logLogSlopeTriangle{0.90}{0.25}{0.1}{2}{red};
          \logLogSlopeTriangle{0.90}{0.25}{0.1}{3}{brown};
          \logLogSlopeTriangle{0.90}{0.25}{0.1}{4}{darkgray};
          \legend{$k=0$,$k=1$,$k=2$,$k=3$};
        \end{loglogaxis}        
      \end{tikzpicture}
    \end{minipage}
    \subcaption{Cubic mesh family}\medskip
    \begin{minipage}[c]{0.30\columnwidth}\centering
      \includegraphics[width=0.9\columnwidth,height=0.2\textheight,keepaspectratio]{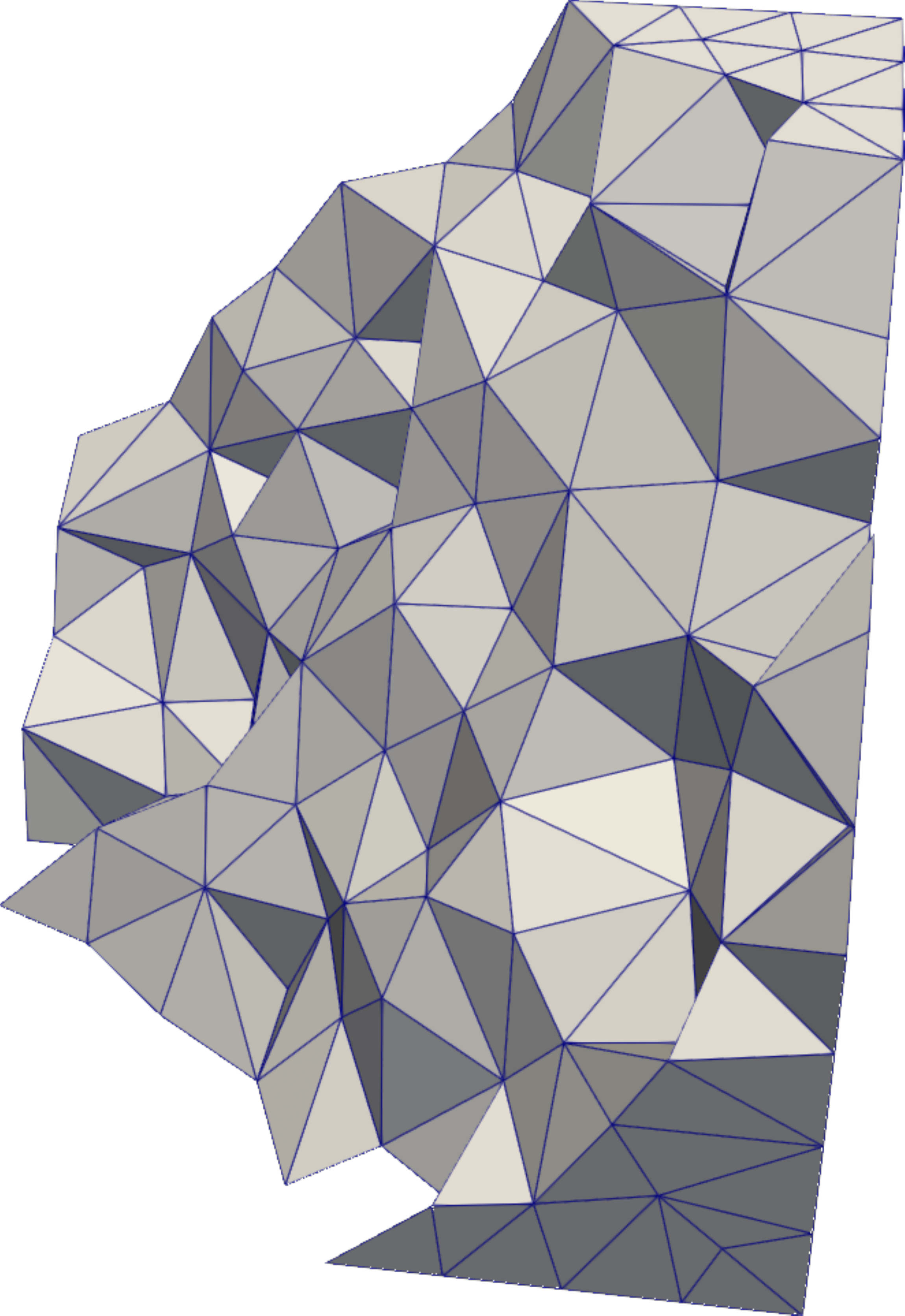}
    \end{minipage}
    \begin{minipage}[c]{0.50\columnwidth}\centering
      \begin{tikzpicture}
        \begin{loglogaxis}
          \addplot +[mark=triangle, style=solid, color=blue] table[x index=0,y index=4] {data/3D/4/sol_1_k0_detailed.dat};
          \addplot +[mark=triangle, style=solid, color=red] table[x index=0,y index=4] {data/3D/4/sol_1_k1_detailed.dat};
          \addplot +[mark=square, style=solid, color=brown] table[x index=0,y index=4] {data/3D/4/sol_1_k2_detailed.dat};
          \addplot +[mark=square, style=solid, color=darkgray] table[x index=0,y index=4] {data/3D/4/sol_1_k3_detailed.dat};
          \logLogSlopeTriangle{0.90}{0.4}{0.1}{1}{blue};
          \logLogSlopeTriangle{0.90}{0.4}{0.1}{2}{red};
          \logLogSlopeTriangle{0.90}{0.4}{0.1}{3}{brown};
          \logLogSlopeTriangle{0.90}{0.4}{0.1}{4}{darkgray};
        \end{loglogaxis}
      \end{tikzpicture}
    \end{minipage}
    \subcaption{Tetgen tetrahedral mesh family}\medskip
    \begin{minipage}[c]{0.30\columnwidth}\centering
      \includegraphics[width=0.9\columnwidth,height=0.2\textheight,keepaspectratio]{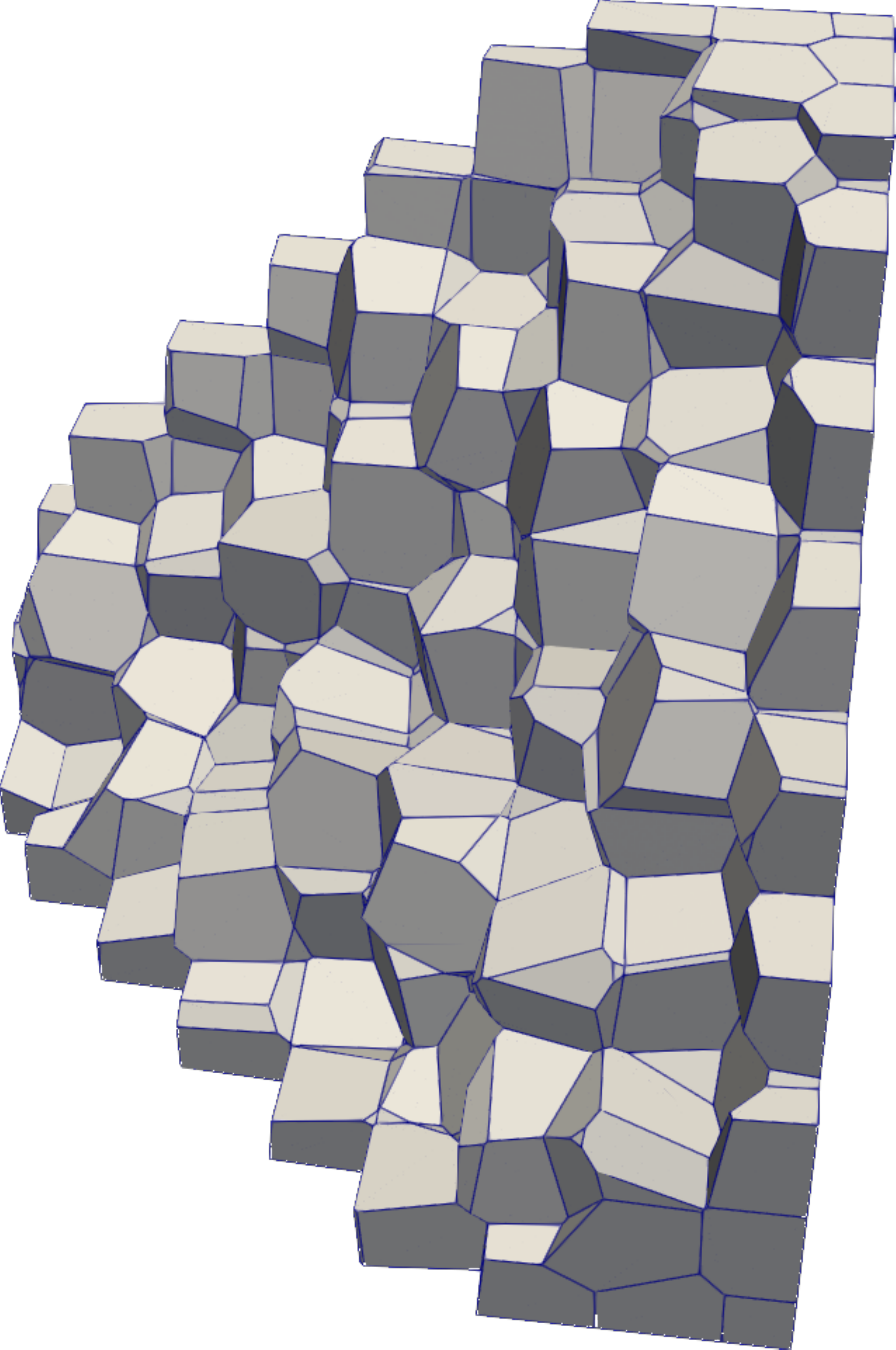}
    \end{minipage}
    \begin{minipage}[c]{0.50\columnwidth}\centering
      \begin{tikzpicture}
        \begin{loglogaxis}
          \addplot +[mark=+, style=solid, color=blue] table[x index=0,y index=4] {data/3D/5/sol_1_k0_detailed.dat};
          \addplot +[mark=triangle, style=solid, color=red] table[x index=0,y index=4] {data/3D/5/sol_1_k1_detailed.dat};
          \addplot +[mark=square, style=solid, color=brown] table[x index=0,y index=4] {data/3D/5/sol_1_k2_detailed.dat};
          \addplot +[mark=square, style=solid, color=darkgray] table[x index=0,y index=4] {data/3D/5/sol_1_k3_detailed.dat};
          \logLogSlopeTriangle{0.90}{0.25}{0.1}{1}{blue};
          \logLogSlopeTriangle{0.90}{0.25}{0.1}{2}{red};
          \logLogSlopeTriangle{0.90}{0.25}{0.1}{3}{brown};
          \logLogSlopeTriangle{0.90}{0.25}{0.1}{4}{darkgray};
        \end{loglogaxis}
      \end{tikzpicture}
    \end{minipage}
    \subcaption{Voronoi mesh family}
  \end{center}
  \caption{Sample mesh (left) and error $\norm{\mathcal{Z},h}{(\utens{\sigma}_h - \IDivDiv{h}\btens{\sigma}, u_h - \lproj{k}{h} u)}$ v. mesh size $h$ (right).}
  \label{fig:convrate3D}
\end{figure}

\subsection{Well-posedness}\label{sec:discrete.problem:well-posedness}

This section contains the proof Theorem \ref{thm:well-posedness} preceeded by two preliminary results: a uniform equivalence of discrete $L^2$-norms and an inf-sup condition on the discrete divdiv operator.

\begin{lemma}[Uniform norm equivalence]
  Recalling the definitions \eqref{eq:discrete.L2.norm} and \eqref{eq:tnorm} of the discrete $L^2$- and component norms, it holds
  \begin{equation}\label{eq:norm.equivalence}
    \norm{\DIV\VDIV,\bullet}{\utens{\tau}_\bullet}
    \lesssim \tnorm{\DIV\VDIV,\bullet}{\utens{\tau}_\bullet}
    \lesssim \norm{\DIV\VDIV,\bullet}{\utens{\tau}_\bullet}
    \qquad\forall\utens{\tau}_\bullet\in\uHdivdiv{\bullet}.
  \end{equation}
\end{lemma}

\begin{proof}
  It suffices to prove \eqref{eq:norm.equivalence} for $\bullet = T\in\Th$, as the result for $\bullet = h$ follows squaring, summing over $T\in\Th$, and passing to square roots.
  We start by proving that
  \begin{equation}\label{eq:norm.equivalence:norm.lesssim.tnorm}
    \norm{\DIV\VDIV,T}{\utens{\tau}_T}
    \lesssim\tnorm{\DIV\VDIV,T}{\utens{\tau}_T}
    \qquad\forall\utens{\tau}_T\in\uHdivdiv{T}.
  \end{equation}
  To this end, we take a generic $\utens{\tau}_T\in\uHdivdiv{T}$ and use \eqref{eq:sT} to write
  \begin{equation}\label{eq:norm.equivalence:norm.lesssim.tnorm:1}
    s_T(\utens{\tau}_T, \utens{\tau}_T)^{\frac12}
    = \tnorm{\DIV\VDIV,T}{\IDivDiv{T}\TP\utens{\tau}_T - \utens{\tau}_T}
    \le \tnorm{\DIV\VDIV,T}{\IDivDiv{T}\TP\utens{\tau}_T}
    + \tnorm{\DIV\VDIV,T}{\utens{\tau}_T},
  \end{equation}
  where the conclusion follows from a triangle inequality.
  We then use the boundedness \eqref{eq:fortin:1} of the interpolator, discrete inverse inequalities, and the boundedness \eqref{eq:TP:boundedness} of $\TP$ to write
  \begin{equation}\label{eq:norm.equivalence:norm.lesssim.tnorm:2}
    \begin{aligned}
      \tnorm{\DIV\VDIV,T}{\IDivDiv{T}\TP\utens{\tau}_T}
      &\lesssim
      \norm{\btens{L}^2(T;\Matr)}{\TP\utens{\tau}_T}
      + h_T\seminorm{\btens{H}^1(T;\Matr)}{\TP\utens{\tau}_T}
      + h_T^2\seminorm{\btens{H}^2(T;\Matr)}{\TP\utens{\tau}_T}
      \\
      &\lesssim\norm{\btens{L}^2(T;\Matr)}{\TP\utens{\tau}_T}
      \lesssim\tnorm{\DIV\VDIV,T}{\utens{\tau}_T}.
    \end{aligned}
  \end{equation}
  Plugging the above estimate into \eqref{eq:norm.equivalence:norm.lesssim.tnorm:1}, we get
  $s_T(\utens{\tau}_T, \utens{\tau}_T)^{\frac12}\lesssim\tnorm{\DIV\VDIV,T}{\utens{\tau}_T}$
  which, combined with the boundedness \eqref{eq:TP:boundedness} of $\TP$, yields \eqref{eq:norm.equivalence:norm.lesssim.tnorm}.
  \medskip
  
  Let us now prove the converse inequality
  \begin{equation}\label{eq:norm.equivalence:tnorm.lesssim.norm}
    \tnorm{\DIV\VDIV,T}{\utens{\tau}_T}
    \lesssim\norm{\DIV\VDIV,T}{\utens{\tau}_T}
    \qquad\forall\utens{\tau}_T\in\uHdivdiv{T}.
  \end{equation}
  To this purpose, we start using a triangle inequality to write
  \begin{equation}\label{eq:norm.equivalence:tnorm.lesssim.norm:1}
    \tnorm{\DIV\VDIV,T}{\utens{\tau}_T}
    \le\tnorm{\DIV\VDIV,T}{\IDivDiv{T}\TP\utens{\tau}_T - \utens{\tau}_T}
    + \tnorm{\DIV\VDIV,T}{\IDivDiv{T}\TP\utens{\tau}_T}
    \eqcolon\term_1 + \term_2.
  \end{equation}
  For the first term, we recall \eqref{eq:sT} to write $\term_1 = s_T(\utens{\tau}_T, \utens{\tau}_T)^{\frac12}\le\norm{\DIV\VDIV,T}{\utens{\tau}_T}$, where the conclusion follows from the definitions \eqref{eq:discrete.L2.norm} of $\norm{\DIV\VDIV,T}{{\cdot}}$ and \eqref{eq:L2.prod:T} of the local discrete $L^2$-product.
  For the second term, we use the second line of \eqref{eq:norm.equivalence:norm.lesssim.tnorm:2} and again the definitions recalled above to write $\term_2\le\norm{\DIV\VDIV,T}{\utens{\tau}_T}$.
  Plugging the above estimates into \eqref{eq:norm.equivalence:tnorm.lesssim.norm:1} concludes the proof of \eqref{eq:norm.equivalence:tnorm.lesssim.norm}.
\end{proof}

\begin{lemma}[Inf-sup condition on {$\DD{h}$}]\label{lem:inf-sup:DDh}
  The following inf-sup condition holds uniformly in $h$:
  \begin{equation}\label{eq:inf-sup:DDh}
    1\lesssim
    \inf_{v_h\in\Poly{k}(\Th)\setminus\{0\}}\sup_{\utens{\tau}_h\in\uHdivdiv{h}\setminus\{\utens{0}\}}
    \frac{\int_\Omega\DD{h}\utens{\tau}_h~v_h}{\norm{\DIV\VDIV,h}{\utens{\tau}_h}\norm{L^2(\Omega)}{v_h}}.
  \end{equation}
\end{lemma}

\begin{proof}
  From the boundedness property \eqref{eq:fortin:1} of the interpolator along with the uniform norm equivalence \eqref{eq:norm.equivalence} and $h_T\le h\lesssim 1$ for all $T\in\Th$, and \eqref{eq:fortin:2}, it can be inferred that
  \begin{equation}\label{eq:fortin}
    \text{
      $\norm{\DIV\VDIV,h}{\IDivDiv{h}\btens{\tau}}\lesssim\norm{\btens{H}^2(\Omega;\Matr)}{\btens{\tau}}$
      and
      $\DD{h}\IDivDiv{h}\btens{\tau} = \lproj{k}{h}(\DIV\VDIV\btens{\tau})$
      for all $\btens{\tau}\in\btens{H}^2(\Omega;\Symm)$,
    }
  \end{equation}
  where $\lproj{k}{h}$ denotes the $L^2$-orthogonal projector on $\Poly{k}(\Th)$.
  Since $\DIV\VDIV:\btens{H}^2(\Omega;\Symm)\to L^2(\Omega)$ is surjective (see \cite[Point (iv) in Theorem 3.25]{Pauly.Zulehner:20}), this shows that $\DD{h}$ is a $B$-compatible operator in the sense of \cite[Section 8.4.1]{Boffi.Brezzi.ea:13}.
  From the abstract theory therein, it can be inferred that \eqref{eq:inf-sup:DDh} holds.
\end{proof}

We are now ready to prove the well-posedness of the discrete problem.

\begin{proof}[Proof of Theorem \ref{thm:well-posedness}]
  By \eqref{eq:norm}, $(\cdot,\cdot)_{\DIV\VDIV,h}$ is coercive with respect to the norm $\norm{\DIV\VDIV,h}{{\cdot}}$ with coercivity constant equal to 1.
  In conjunction with the inf-sup condition \eqref{eq:inf-sup:DDh} on $\DD{h}$, this classicaly yields \eqref{eq:inf-sup:Ah}.
  The well-posedness of problem \eqref{eq:discrete} (or, equivalently, \eqref{eq:discrete:variational}) then follows from classical arguments (see, e.g., \cite[Proposition 7]{Di-Pietro.Droniou:18}).
\end{proof}

\subsection{Error estimate}\label{sec:error.estimate}

The goal of this section is to prove Theorem \ref{thm:error.estimate}.
To this purpose, we preliminarily need estimates of the discrete $L^2$-product and adjoint divdiv consistency errors.

\begin{lemma}[Estimate of the consistency error for the discrete $L^2$-product]
  Let $\btens{\upsilon}\in\btens{H}^2(\Omega;\Symm)$ and define the discrete $L^2$-product consistency error
  \begin{equation}\label{eq:Eprod}
    \Eprod(\btens{\upsilon};\utens{\tau}_h)
    \coloneq
    \int_\Omega\btens{\upsilon}:\TP[h]\utens{\tau}_h
    - (\IDivDiv{h}\btens{\upsilon},\utens{\tau}_h)_{\DIV\VDIV,h}.
  \end{equation}
  Then, additionally assuming $\btens{\upsilon}\in\btens{H}^{k+1}(\Th;\Symm)$, it holds
  \begin{equation}\label{eq:Eprod:estimate}
    \norm{\DIV\VDIV,h,*}{\Eprod(\btens{\upsilon};\cdot)}
    \lesssim h^{k+1}\seminorm{\btens{H}^{k+1}(\Th;\Matr)}{\btens{\upsilon}}.
  \end{equation}
\end{lemma}

\begin{proof}
  We start by decomposing \eqref{eq:Eprod} as follows:
  \[
  \Eprod(\btens{\upsilon};\utens{\tau}_h)
  = \sum_{T\in\Th}\left[
    \term_1(T) + \term_2(T) + \term_3(T)
    \right]
  \]
  where, recalling that $\TP\IDivDiv{T}(\tlproj{k}{T}\btens{\upsilon}) = \tlproj{k}{T}\btens{\upsilon}$ by \eqref{eq:TP:polynomial.consistency},
  \[
  \begin{gathered}
    \term_1(T)
    \coloneq\int_T(\btens{\upsilon} - \tlproj{k}{T}\btens{\upsilon}):\TP\utens{\tau}_T,
    \qquad
    \term_2(T)
    \coloneq\int_T\TP\IDivDiv{T}(\tlproj{k}{T}\btens{\upsilon} - \btens{\upsilon}):\TP\utens{\tau}_T,
    \\
    \term_3(T)
    \coloneq -s_T(\IDivDiv{T}\btens{\upsilon},\utens{\tau}_T).
  \end{gathered}
  \]
  We next proceed to estimate these terms one by one.
  
  The first term is readily treated using a Cauchy--Schwarz inequality along with the approximation properties of the $L^2$-orthogonal projector (see \cite[Lemma 3.1]{Di-Pietro.Droniou:17*1} and \cite[Theorem 1.45]{Di-Pietro.Droniou:20}) for the first factor and the definitions \eqref{eq:discrete.L2.norm} of $\norm{\DIV\VDIV,T}{{\cdot}}$ and \eqref{eq:L2.prod:T} of the local discrete $L^2$-product for the second:
  \[
  |\term_1(T)|
  \lesssim
  \norm{\btens{L}^2(T;\Matr)}{\btens{\upsilon} - \tlproj{k}{T}\btens{\upsilon}}\norm{\btens{L}^2(T;\Matr)}{\TP\utens{\tau}_T}
  \lesssim
  h^{k+1}\seminorm{\btens{H}^{k+1}(T;\Matr)}{\btens{\upsilon}}\norm{\DIV\VDIV,T}{\utens{\tau}_T}.
  \]

  For the second term, we preliminarily notice that
  \begin{equation}\label{eq:Eprod:T2(T):basic}
    \norm{\btens{L}^2(T;\Matr)}{\TP\IDivDiv{T}(\tlproj{k}{T}\btens{\upsilon} - \btens{\upsilon})}
    \overset{\eqref{eq:TP:boundedness}}\lesssim\tnorm{\DIV\VDIV,T}{\IDivDiv{T}(\tlproj{k}{T}\btens{\upsilon} - \btens{\upsilon})}
    \lesssim h^{k+1}\seminorm{\btens{H}^{k+1}(T;\Matr)}{\btens{\upsilon}},
  \end{equation}
  where the conclusion follows combining the boundedness \eqref{eq:fortin:1} of $\IDivDiv{T}$ written for $\btens{\tau} = \btens{\upsilon} - \tlproj{k}{T}\btens{\upsilon}$ 
  with the approximation properties of the $L^2$-orthogonal projector and $h_T\le h$ for all $T\in\Th$.
  We can then use a Cauchy--Schwarz inequality along with \eqref{eq:discrete.L2.norm} and \eqref{eq:L2.prod:T} as for $\term_1(T)$ to write
  \[
  |\term_2(T)|
  \lesssim
  \norm{\btens{L}^2(T;\Matr)}{\TP\IDivDiv{T}(\btens{\upsilon} - \tlproj{k}{T}\btens{\upsilon})}\norm{\btens{L}^2(T;\Matr)}{\TP\utens{\tau}_T}
  \lesssim
  h^{k+1}\seminorm{\btens{H}^{k+1}(T;\Matr)}{\btens{\upsilon}}\norm{\DIV\VDIV,T}{\utens{\tau}_T}.
  \]

  For the third term, recalling the polynomial consistency \eqref{eq:sT:polynomial.consistency} of the stabilization bilinear form, we can write, for all $\utens{\tau}_T\in\uHdivdiv{T}$,
  \[
  \begin{aligned}
    |\term_3(T)|
    &= |s_T(\IDivDiv{T}(\btens{\upsilon} - \tlproj{k}{T}\btens{\upsilon}),\utens{\tau}_T)|
    \\
    &\le\norm{\DIV\VDIV,T}{\IDivDiv{T}(\btens{\upsilon} - \tlproj{k}{T}\btens{\upsilon})}
    \norm{\DIV\VDIV,T}{\utens{\tau}_T}
    \\
    \overset{\eqref{eq:norm.equivalence}}&\lesssim
    \tnorm{\DIV\VDIV,T}{\IDivDiv{T}(\btens{\upsilon} - \tlproj{k}{T}\btens{\upsilon})}
    \norm{\DIV\VDIV,T}{\utens{\tau}_T}
    \overset{\eqref{eq:Eprod:T2(T):basic}}\lesssim
    h^{k+1}\seminorm{\btens{H}^{k+1}(T;\Matr)}{\btens{\upsilon}}\norm{\DIV\VDIV,T}{\utens{\tau}_T}.\qedhere
  \end{aligned}
  \]
\end{proof}

\begin{lemma}[Estimate of the adjoint consistency error for the divdiv operator]
  Let $v\in H^2(\Omega)$ be such that $v = \partial_{\normal} v = 0$ on $\partial\Omega$ and define the divdiv adjoint consistency error linear form $\Edivdiv(v;\cdot):\uHdivdiv{h}\to\Real$ such that, for all $\utens{\tau}_h\in\uHdivdiv{h}$,
  \begin{equation}\label{eq:Edivdiv}
    \Edivdiv(v;\utens{\tau}_h)
    \coloneq
    \int_\Omega\HESS v:\TP[h]\utens{\tau}_h
    - \int_\Omega v~\DD{h}\utens{\tau}_h.
  \end{equation}
  Then, additionally assuming $v\in H^{k+3}(\Th)$, it holds,
  \begin{equation}\label{eq:Edivdiv:estimate}
    \norm{\DIV\VDIV,h,*}{\Edivdiv(v;\cdot)}
    \lesssim h^{k+1}\seminorm{H^{k+3}(\Th)}{v}.
  \end{equation}
\end{lemma}

\begin{proof}
  Let, for all $T\in\Th$, $v_T\coloneq\lproj{k+2}{T}v$.
  Combining \eqref{eq:TP} for $(\btens{\tau},v) = (\btens{0}, v_T)$ with the definition \eqref{eq:Edivdiv} of the adjoint consistency error, we get, for any $\utens{\tau}_h\in\uHdivdiv{h}$, such that $\norm{\DIV\VDIV,h}{\utens{\tau}_h} = 1$,
  \[
  \begin{aligned}
    \Edivdiv(v;\utens{\tau}_h)
    &= \sum_{T\in\Th}\int_T (v_T - v)~\DD{T}\utens{\tau}_T
    + \sum_{T\in\Th}\int_T \HESS(v - v_T):\TP\utens{\tau}_T
    \\
    &\quad
    + \sum_{T\in\Th}\sum_{F\in\FT}\omega_{TF}\sum_{E\in\EF}\omega_{FE}\int_E(\normal_{FE}^\top\btens{\tau}_E\normal_F)(v_T - v)
    \\
    &\quad
    + \sum_{T\in\Th}\sum_{F\in\FT}\omega_{TF}\int_F\left[
      \gammaF\utens{\tau}_F~\partial_{\normal_F}(v_T - v)
      + D_{\btens{\tau},F}~(v_T - v)
      \right],
  \end{aligned}
  \]
  where the insertion of the face or edge traces of $v$ and of $\partial_{\normal} v$ into the boundary terms is justified by their single-valuedness along with the assumed boundary conditions.
  Using Cauchy--Schwarz inequalities along with the approximation properties of $v_T$, the definition \eqref{eq:tnorm} of $\tnorm{\DIV\VDIV,T}{{\cdot}}$, 
  and the boundedness \eqref{eq:DD:boundedness} of $\DD{T}$, \eqref{eq:TP:boundedness} of $\TP$, and \eqref{eq:gammaF:boundedness} of $\gammaF$, we infer  
  \[
  \left|\Edivdiv(v;\utens{\tau}_h)\right|\lesssim h^{k+1}\seminorm{H^{k+3}(\Th)}{v}\tnorm{\DIV\VDIV,h}{\utens{\tau}_h}
  \overset{\eqref{eq:norm.equivalence}}{\lesssim} h^{k+1}\seminorm{H^{k+3}(\Th)}{v}.\qedhere
  \]
\end{proof}

\begin{proof}[Proof of Theorem \ref{thm:error.estimate}]
  Accounting for the inf-sup condition \eqref{eq:inf-sup:Ah}, by \cite[Theorem 10]{Di-Pietro.Droniou:18}, it holds
  \begin{equation}\label{eq:error.estimate:basic}
    \norm{\mathcal{Z},h}{(\utens{\sigma}_h - \IDivDiv{h}\btens{\sigma}, u - \lproj{k}{h} u)}
    \lesssim\norm{\mathcal{Z},h,*}{\Err(\btens{\sigma},u;\cdot)},
  \end{equation}
  where $\norm{\mathcal{Z},h,*}{{\cdot}}$ denotes the norm dual to $\norm{\mathcal{Z},h}{{\cdot}}$ and the consistency error linear form $\Err(\btens{\sigma},u;\cdot):\mathcal{Z}_h^k\to\Real$ is such that
  \[
  \begin{aligned}
    \Err(\btens{\sigma},u;\utens{\tau}_h,v_h)
    &\coloneq
    \int_\Omega fv_h
    - \mathcal{A}_h((\IDivDiv{h}\btens{\sigma},\lproj{k}{h}u), (\utens{\tau}_h,v_h))
    \\
    &=
    \Eprod(\btens{\sigma};\utens{\tau}_h)
    + \cancel{\int_\Omega(\DIV\VDIV\btens{\sigma} - \DD{h}\IDivDiv{T}\btens{\sigma})~v_h}
    + \Edivdiv(u;\utens{\tau}_h),
  \end{aligned}
  \]
  where, to pass to the second line, we have used the fact that $f = -\DIV\VDIV\btens{\sigma}$ almost everywhere in $\Omega$, added the term $\int_\Omega(\btens{\sigma} + \HESS u):\TP[h]\utens{\tau}_h = 0$, and used the definitions \eqref{eq:Eprod} and \eqref{eq:Edivdiv} of the $L^2$-product and adjoint divdiv consistency errors,
  while the cancellation follows from \eqref{eq:fortin}.
  To prove \eqref{eq:error.estimate} it suffices to use \eqref{eq:Eprod:estimate} and \eqref{eq:Edivdiv:estimate} to estimate the terms in the right-hand side of the above expression and plugging the resulting bound into \eqref{eq:error.estimate:basic} after observing that $\norm{\DIV\VDIV,h}{\utens{\tau}_h}\le \norm{\mathcal{Z},h}{(\utens{\tau}_h, v_h)}$ by definition of this latter norm.
\end{proof}


\section{Local complex property}\label{sec:complex.property}

We collect in this section the proofs of the complex properties \eqref{eq:LC.DGSC} and \eqref{eq:LC.SCDD}.

\subsection{Proof of \eqref{eq:LC.DGSC}}

Let $\uvec{v}_T \in \uHdevgrad{T}$. We need to prove that the edge, face, and element components of $\SC{T}\DG{T}\uvec{v}_T$ (obtained plugging \eqref{eq:defuDG} into \eqref{eq:defuSC}) vanish.

\subsubsection{Edge components}

Given $E \in \ET$, and letting, for the sake of brevity, 
$\Gamma_{\tangent,E}\coloneq \DerE{k}(\bvec{v}_{V_1} \cdot \tE, \bvec{v}_{V_2} \cdot \tE, \doftE{v})$, 
we have that
\begin{align*}
  \SCE \DG{E} \uvec{v}_E
  &=
  \Ctensor \bigg[
    \begin{aligned}[t]
      &\DerE{k+1}\big(\trnnE{({\gdofV[V_1]{v}})},\trnnE{({\gdofV[V_2]{v}})}, \gdofE{v}\big)
      \\
      &- \DerE{k+1}\bigg(
      \trnnE{(\DEV {\gdofV[V_1]{v}})},
      \trnnE{(\DEV {\gdofV[V_2]{v}})},
      \gdofE{v} 
      - \frac13\left(
      \tr \gdofE{v} + \Gamma_{\tangent,E}
      \right) \ID{2}
      \bigg)
      \bigg]
    \end{aligned}
    \\
    &\eqcolon\Ctensor\left[\term_1 - \term_2\right].%
\end{align*}
Since $\DerE{k+1}$ acts component-wise and the off-diagonal entries of its arguments in $\term_1$ and $\term_2$ coincide, the off-diagonal entries of these terms coincide as well, showing that there exists $\lambda\in\Poly{k+1}(E)$ such that $\term_1 - \term_2 = \lambda\ID{2}$.
Hence,
\[
\SCE \DG{E} \uvec{v}_E
= \Ctensor(\lambda\ID{2}) = \btens{0},
\]
where the conclusion follows from the definition \eqref{eq:defCtensor} of $\Ctensor$.

\subsubsection{Face components}

Let now $F \in \FT$.
For all $r \in \Poly{k+1}(F)$,
using the definition \eqref{eq:defSCnnF} of $\SCnnF$ with $\utens{\tau}_F = \DG{F}\uvec{v}_F$,
invoking the definition \eqref{eq:defDGntF} of $\DGntF$ 
with test function $\bvec{w} = \CURLF r\in\Roly{k}(F)\subset\RT{k+1}(F)$ (cf. \eqref{eq:RT}),
and noticing that $\CURLF r \cdot \nFE = - \partial_{\tE} r$ on every edge $E\in\EF$, we have
\[
\begin{aligned}
  \int_F \SCnnF \DG{F} \uvec{v}_F \,r
  &=
  \int_F \DGntF \uvec{v}_F \cdot \CURLF r
  - \!\!\sum_{E \in \EF} \wFE \int_E
  \DerE{k+1}(\trnE{(\bvec{v}_{V_1})},\trnE{(\bvec{v}_{V_2})}, \dofnE{v})\cdot\nF
  \, r\\
  \overset{\eqref{eq:defDGntF}}&=
  -\int_F v_{\normal,F}\, \cancel{\DIVF \CURLF r}
  - \sum_{E \in \EF} 
  \wFE \int_E 
  (\dofnE{v} \cdot \nF) \,\partial_{\tE} r \\
  &\quad
  -\sumEF
  \DerE{k+1}(\trnE{(\bvec{v}_{V_1})},\trnE{(\bvec{v}_{V_2})}, \dofnE{v})\cdot \nF
  \, r
  \\
  \overset{\eqref{eq:IPP.E}}&=
  -\sum_{E\in\EF}\wFE\jump{E}{(\bvec{v}_{V} \cdot \nF) r}
  = 0,
\end{aligned}
\]
where 
we have concluded observing that, for any family $(\varphi_V)_{V \in \VF} \in \Real^{\VF}$,
\begin{equation}\label{eq:zero-jump}
  \sum_{E\in\EF}\wFE\jump{E}{\varphi_V} = 0
\end{equation}

Writing the definition \eqref{eq:defSCddofF} of $\SCddofF$ for $\uvec{\tau}_F = \DG{F} \uvec{v}_F$, we get, for all $r \in \Poly{k+1}(F)$,
\begin{align*}
  &\term\coloneq\int_F \SCddofF \DG{F} \uvec{v}_F\, r
  \\
  &\quad= 
  - \int_F\DGttF \uvec{v}_F \tdot \CURLF \GRADF r
  + \sumEF
  \DerE{k+1}(\trnE{(\bvec{v}_{V_1})},\trnE{(\bvec{v}_{V_2})}, \dofnE{v}) \cdot \nFE\,\partial_{\nFE} r
  \\
  &\qquad
  - \sumEF (2 \nFE^\top \gdofE{v} \nFE + \nF^\top \gdofE{v} \nF)\,\partial_{\tE}r
  \\
  &\qquad
  + \sumEF \cancelto{1}{\frac{2 \nFE^\top \ID{2} \nFE + \nF^\top \ID{2} \nF}{3}}
  \left[\tr \gdofE{v} +
    \DerE{k}(\bvec{v}_{V_1} \cdot \tE, \bvec{v}_{V_2} \cdot \tE, \doftE{v})
    \right]\,\partial_{\tE}r
  \\
  &\qquad
  - \sumEF
  \nFE^\top\DerE{k+1}(\trnnE{({\gdofV[V_1]{v}})},\trnnE{({\gdofV[V_2]{v}})}, \gdofE{v})\nFE\, r
  \\
  &\qquad
  + 
  \sum_{E\in\EF}\wFE\jump{E}{(\nFE^\top \DEV \gdofV{v} \nFE)\,r},
\end{align*}
where we have used the fact
that both $\normal_{FE}$ and $\normal_F$ have unit Euclidian norm for the fourth term.

We next expand, in the above expression, $\DGttF \uvec{v}_F$ according to \eqref{eq:defDGttF} with $\btens{\sigma} = \CURLF \GRADF r\in \CGoly{k-1}(F)\subset\CGtrimPoly{k}(F)$ to go on writing
\begin{align*}
  \term 
  &= \int_F \doftF{v} \cdot \cancel{\VDIVF \CURLF \GRADF r}
  + \frac13 \cancel{\int_F \gdofF{v} \tr(\CURLF \GRADF r)}
  \\
  &\quad
  + \sumEF \doftE{v} \, \partial_{\tE}^2 r
  + \sumEF (\dofnE{v} \cdot \nFE)\, \partial_{\tE}\partial_{\nFE} r
  \\ 
  &\quad + \sumEF\left(
  -\bcancel{2} \nFE^\top \gdofE{v} \nFE
  - \bcancel{\nF^\top \gdofE{v} \nF}
  + \bcancel{\tr \gdofE{v}}
  \right)\,\partial_{\tE}r
  \\
  &\quad
  + \sumEF\DerE{k}(\bvec{v}_{V_1} \cdot \tE, \bvec{v}_{V_2} \cdot \tE, \doftE{v})\, \partial_{\tangent_E}r
  \\
  &\quad
  + \sumEF \DerE{k+1}(\trnE{(\bvec{v}_{V_1})},\trnE{(\bvec{v}_{V_2})}, \dofnE{v})\cdot \nFE\,\partial_{\nFE} r
  \\
  &\quad
  -\sumEF\nFE^\top \DerE{k+1}(\trnnE{({\gdofV[V_1]{v}})},\trnnE{({\gdofV[V_2]{v}})}, \gdofE{v})\nFE\, r
  \\
  &\quad
  + \sum_{E\in\EF}\wFE\jump{E}{(\nFE^\top \gdofV{v} \nFE)\,r}
  - \frac13 \cancel{\sum_{E\in\EF}\wFE\jump{E}{r\, \tr \gdofV{v} \nFE^\top \nFE}},
\end{align*}
where we have used $\VDIVF \CURLF = \bvec{0}$ in the first cancellation,
the fact that $\tr(\CURLF \GRADF) = \ROTF \GRADF = 0$ to cancel the second term,
noticed that, for all $E \in\EF$,
$\tE^\top (\CURLF \GRADF r) \nFE = - \partial_{\tE}^2 r$
and $\nFE^\top (\CURLF \GRADF r) \nFE = - \partial_{\tE}\partial_{\nFE} r$
in the third and fourth terms,
and observed that, for all $E\in\EF$, $\normal_{FE}^\top\gdofE{v} \nFE + \nF^\top \gdofE{v} \nF = \tr \gdofE{v}$ since $(\normal_{FE},\normal_F)$ is an orthonormal basis in the fifth term.
The cancellation of the last term follows noticing that $\normal_{FE}^\top\normal_{FE} = 1$ and invoking \eqref{eq:zero-jump}.

Using the definition \eqref{eq:IPP.E} of $\DerE{\bullet}$ on the terms where this operator appears, we then get

\begin{align*}
  \term
  &=
  \sumEF \doftE{v}\, \partial_{\tE}^2 r
  + \sumEF (\dofnE{v}\cdot\nFE)\, \partial_{\tE}\partial_{\nFE} r
  \\
  &\quad
  - \sumEF (\nFE^\top\gdofE{v}\nFE)\, \partial_{\tE} r
  \\
  &\quad
  - \sumEF \doftE{v}\, \partial_{\tE}^2 r
  + \sum_{E\in\EF}\wFE \jump{E}{(\bvec{v}_V \cdot \tE)\, \partial_{\tE} r}
  \\
  &\quad
  - \sumEF (\dofnE{v} \cdot\nFE)\, \partial_{\tE}\partial_{\nFE} r
  + \sum_{E\in\EF}\wFE \jump{E}{(\bvec{v}_V \cdot \nFE) \, \partial_{\nFE}r}
  \\
  &\quad
  + \sumEF (\nFE^\top \gdofE{v} \nFE)\, \partial_{\tE} r
  = \sum_{E \in \EF} \wFE \jump{E}{\trtF{(\bvec{v}_V)} \cdot \GRADF r}
  \overset{\eqref{eq:zero-jump}}{=} 0,
\end{align*}
where the penultimate equality follows simplifying the terms involving integrals over $E$
and gathering together the edge jump terms after observing that
$(\bvec{v}_V \cdot \nFE)\, \partial_{\nFE} r + (\bvec{v}_V \cdot\tE)\, \partial_{\tE}r = \trtF{(\bvec{v}_V)} \cdot \GRADF r$.

\subsubsection{Element component}

To conclude the proof of \eqref{eq:LC.DGSC}, we it remains to show that the element component of $\SC{T}\DG{T}\uvec{v}_T$ vanishes.
Writing the definition \eqref{eq:defSCT} of $\SCT$ for $\utens{\tau}_T = \DG{T}\uvec{v}_T$, we get, for all $\btens{\sigma}\in\HtrimPoly{k}(T)$,
\begin{align*}
  \term\coloneq
  \int_T \SCT\DG{T}\uvec{v}_T \tdot \btens{\sigma}
  &=
  \int_T \DGT\uvec{v}_T \tdot \CURL \btens{\sigma}
  + \sum_{F \in \FT} \wTF \int_F \DGttF\uvec{v}_F \tdot \trttF{(\btens{\sigma} \times \nF)}\\
  &\quad + \sumFT \DGntF\uvec{v}_F \cdot \trntF{(\btens{\sigma}\times\nF)}.
\end{align*}
Next, we expand $\DGT\uvec{v}_T$ according to \eqref{eq:defDGT} (which is possible since $\CURL \btens{\sigma} \in \SRoly{k-1}(T)\subset\SRtrimPoly{k}(T)$, cf. \eqref{eq:SRtrimPoly}),
$\DGttF\uvec{v}_F$ according to \eqref{eq:defDGttF} (after noticing that, by \eqref{eq:poly.trace.H}, $\trttF{(\btens{\sigma}\times\nF)} \in \CGtrimPoly{k}(F)$),
and $\DGntF\uvec{v}_F$ according to \eqref{eq:defDGntF} (possible since $\trntF{(\btens{\sigma}\times\nF)} \in \tPoly{k}(F;\Real^2) \subset \RT{k+1}(F)$).
This gives
\begin{align*}
  \term
  &=
  - \int_T \bvec{v}_T \cdot \cancel{\VDIV \CURL \btens{\sigma}} 
  + \sumFT \left(
  \dofnF{v} \nF^\top \CURL \btens{\sigma} \nF + \hdoftF{v}^\top \CURL\btens{\sigma} \nF
  \right)
  \\
  &\quad
  - \sumFT \doftF{v} \cdot \VDIVF\trttF{(\btens{\sigma} \times \nF)}
  - \frac13\sumFT \gdofF{v} \cancel{\ID{2} \tdot \trttF{(\btens{\sigma} \times \nF)}}
  \\
  &\quad
  + \sumEFT 
  (\dofnE{v} \cdot \nFE)\,
  \nFE^\top\trttF{(\btens{\sigma} \times \nF)}\nFE \\
  &\quad
  + \sumEFT 
  \doftE{v}\,
  \tE^\top\trttF{(\btens{\sigma} \times \nF)}\nFE\\
  &\quad
  - \sumFT \dofnF{v} \DIVF \trntF{(\btens{\sigma}\times\nF)}\\
  &\quad
  + \sumEFT
  (\dofnE{v} \cdot \nF)
  (\trntF{(\btens{\sigma}\times\nF)}\cdot\nFE),
\end{align*}
where we have cancelled the first term using the identity $\VDIV \CURL \btens{\sigma} = 0$
and the sixth term using the fact that $\btens{\sigma}$ is symmetric, hence $\trttF{(\btens{\sigma} \times \nF)}$ is traceless.
Noticing that, by \eqref{eq:injection.3d}, $\hdoftF{v}^\top \CURL\btens{\sigma} \nF = \doftF{v} \cdot (\CURL \btens{\sigma})_{\tangent\normal,F}$, 
and rearranging the terms, we can go on writing
\[
\begin{aligned}
  \term
  &=
  \sumFT \dofnF{v}\, \cancel{\left[
      \nF^\top \CURL \btens{\sigma} \nF - \DIVF \trntF{(\btens{\sigma}\times\nF)}
      \right]
  }
  \\
  &\quad
  + \sumFT \doftF{v} \cdot \bcancel{\left[
      (\CURL \btens{\sigma})_{\tangent\normal,F} - \VDIVF\trttF{(\btens{\sigma} \times \nF)}
      \right]}
  \\
  &\quad
  + \sumEFT ( \dofnE{v} \cdot \nFE)\, \nFE^\top\trttF{(\btens{\sigma} \times \nF)}\nFE
  \\
  &\quad
  + \sumEFT \doftE{v} \,\tE^\top\trttF{(\btens{\sigma} \times \nF)}\nFE
  \\
  &\quad
  + \sumEFT ( \dofnE{v} \cdot \nF)(\trntF{(\btens{\sigma}\times\nF)}\cdot\nFE),
\end{aligned}
\]
where the first two terms cancel thanks to the identities
$\nF^\top \CURL \btens{\sigma} \nF = \DIVF \trntF{(\btens{\sigma}\times\nF)}$ and
$(\CURL \btens{\sigma})_{\tangent\normal,F} = \VDIVF\trttF{(\btens{\sigma} \times \nF)}$, respectively.
Gathering together the terms involving integrals over edges, and 
  using \eqref{eq:injection.3d} with $\bvec{u}_{\bvec{w}} = \dofnE{v}$ and $\bvec{w} = \{\normal_{FE}, \normal_F\}$,
we then have
\[
\begin{aligned}
  \term 
  &\quad
  = \sumEFT\left[
    \hdofnE{v}^\top(\btens{\sigma} \times \nF)\nFE 
    + \doftE{v} \,\tE^\top(\btens{\sigma} \times \nF)\nFE
    \right]
  \\
  &\quad
  = \sumEFT (\hdofnE{v} + \doftE{v}\, \tE)^\top \btens{\sigma} (\nF \times \nFE)
  \\
  &\quad
  = - \sumEFT (\hdofnE{v} + \doftE{v}\, \tE)^\top \btens{\sigma} \tE
  = 0,
\end{aligned}
\]
where we have used the vector triple product formula to pass to the second line
and the fact that $(\tE, \nFE, \nF)$ forms a right-handed system (i.e., $\nF \times \nFE = - \tE$) to pass to the third line.
Finally, the conclusion follows observing that, for any family of functions $(\phi_E)_{E\in\ET}$ such that $\phi_E\in L^2(E)$ for all $E\in\ET$,
\begin{equation} \label{eq:zero-edges}
  \sumEFT \phi_E = 0.
\end{equation}
This completes the proof of \eqref{eq:LC.DGSC}

\subsection{Proof of \eqref{eq:LC.SCDD}}

Let $\utens{\tau}_T \in \uHsymcurl{T}$.
We need to show that $\DD{T}\SC{T}\utens{\tau}_T = 0$.
Using the definition \eqref{eq:DT} of $\DD{T}$ with $\utens{\upsilon}_T = \SC{T} \utens{\tau}_T$, 
we have, for all $v \in \Poly{k}(T)$,
\begin{multline*}
  \term\coloneq
  \int_T \DD{T} \SC{T} \utens{\tau}_T \, v
  =
  \int_T \SCT\utens{\tau}_T:\HESS v
  - \sum_{F\in\FT}\omega_{TF}\sum_{E\in\EF}\omega_{FE}\int_E(\normal_{FE}^\top\SCE\utens{\tau}_E\normal_F)~v
  \\
  - \sum_{F\in\FT}\omega_{TF}\int_F\SCnnF\utens{\tau}_F~\partial_{\normal_F}v
  - \sum_{F\in\FT}\omega_{TF}\int_F\SCddofF\utens{\tau}_F~v
  \eqcolon\mathfrak{A} + \mathfrak{B} + \mathfrak{C} + \mathfrak{D}.
\end{multline*}
Next, we expand $\SCT$ according to \eqref{eq:defSCT} with $\btens{\sigma} = \HESS v\in\Holy{k-2}(T)\subset\HtrimPoly{k}(T)$ (cf. \eqref{eq:HtrimPoly}), 
and $\SCE$, $\SCnnF$, and $\SCddofF$ according to \eqref{eq:defSCE}, \eqref{eq:defSCnnF}, and \eqref{eq:defSCddofF}, respectively, to write  
\begin{align*} 
  &\left.\begin{aligned}
    \term & = \int_T \srvec{\tau} \tdot \cancel{\CURL \HESS v}
    \\
      {} &\quad
      + \sumFT\cgtvec[F]{\tau} \tdot \trttF{(\HESS v \times \nF)}
      + \sumFT\rttvec[F]{\tau} \cdot \trntF{(\HESS v\times\nF)}
  \end{aligned}
  \right\}\mathfrak{A}
  \\
  &\quad
  \left.\begin{aligned}
    &\quad - \sumEFT (\normal_{FE}^\top\Ctensor\cdofE{\tau}\normal_F)\,v\\
    &\quad + \sumEFT \left(
    \normal_{FE}^\top\Ctensor\DerE{k+1}(\trnnE{(\btens{\tau}_{V_1})},\trnnE{(\btens{\tau}_{V_2})}, \btens{\tau}_E)\normal_F
    \right)\,v
  \end{aligned}\right\}\mathfrak{B}
  \\
  &\quad  
  \left.\begin{aligned}
    &\quad - \sumFT \rttvec[F]{\tau} \cdot \CURLF \partial_{\nF} v
    + \sumEFT (\doftE{\tau} \cdot \nF)\, \partial_{\nF} v
  \end{aligned}
  \right\}\mathfrak{C}
  \\
  &\quad
  \left.\begin{aligned}
    &\quad
    - \sumFT \cgtvec[F]{\tau} \tdot \CURLF \GRADF v + \sumEFT (\doftE{\tau} \cdot \nFE)\,\partial_{\nFE} v
    \\
    &\quad
    - \sumEFT \left ( 2 \nFE^\top \btens{\tau}_E \nFE + \nF^\top \btens{\tau}_E \nF \right )\, \partial_{\tE} v
    \\
    &\quad
    - \sumEFT \nFE^\top \cdofE{\tau} \, \nFE\, v
    + \sum_{F\in\FT}\wTF\sum_{E\in\EF}\wFE \jump{E}{v\, \nFE^\top \, \btens{\tau}_V \, \nFE}
  \end{aligned}\right\}\mathfrak{D}
  \\
  &\quad\eqcolon 0 + \term_1 + \cdots + \term_{11}.
\end{align*}
We infer $\term_{1} = - \term_{7}$ and $\term_{2} = - \term_{5}$ from the identities
$\trttF{(\HESS v \times \nF)} = \CURLF\GRADF v$ and 
$\nF^\top(\HESS v \times \nF) = \CURLF \partial_{\nF} v$.
It remains to shows that the edge terms cancel.
From the definition \eqref{eq:defCtensor} of $\Ctensor$, we get
, for any $\btens{\eta}\in\Real^{2\times 2}$,
$\nFE^\top \Ctensor \btens{\eta} \nF
= -\frac{1}{2} \nFE^\top \btens{\eta} \nFE
+ \frac{1}{2} \nF^\top \btens{\eta} \nF$, hence
\begin{equation}\label{eq:decomp.Ceta}
  \nFE^\top \Ctensor \btens{\eta} \nF
  + \nFE^\top \btens{\eta} \nFE
  = \frac{1}{2} \tr \btens{\eta}.
\end{equation}
Applying the above equation to $\btens{\eta} = \cdofE{\tau}$ and using \eqref{eq:zero-edges} with $\phi_E = \frac12\tr\cdofE{\tau}$, we infer $\term_{3} + \term_{10} = 0$.
We can then merge some derivatives on edges. Since
$
(\doftE{\tau} \cdot \nF)\, \partial_{\nF} v +
(\doftE{\tau} \cdot \nFE)\, \partial_{\nFE} v =
\doftE{\tau} \cdot \trnE{(\GRAD v)},
$
we also have $\term_{6} + \term_{8} = 0$ from \eqref{eq:zero-edges}.
Letting $\term_9(E)\coloneq\int_E\left ( 2 \nFE^\top \btens{\tau}_E \nFE + \nF^\top \btens{\tau}_E \nF \right )\, \partial_{\tE} v$ and observing that the quantity in parenthesis is equal to $\tr\btens{\tau}_E\ID{2} + \nFE^\top\btens{\tau}_E\nFE$ since $(\nF,\nFE)$ form an orthonormal basis and the trace is an invariant, we get
\begin{align*}
  \term_9(E)
  &= \int_E \tr \btens{\tau}_E\,\partial_{\tE}v
  +\int_E \nFE^\top \btens{\tau}_E \nFE\,\partial_{\tE} v \\
  \overset{\eqref{eq:IPP.E}}&=
  \int_E \tr \btens{\tau}_E\,\partial_{\tE} v
  - \int_E \nFE^\top
  \DerE{k+1}(\trnnE{(\btens{\tau}_{V_1})},\trnnE{(\btens{\tau}_{V_2})}, \btens{\tau}_E)\nFE~v
  + \jump{E}{\nFE^\top \, \btens{\tau}_V \, \nFE\,v} \\
  \overset{\eqref{eq:decomp.Ceta}}&=
  \int_E \tr \btens{\tau}_E\,\partial_{\tE} v
  - \int_E \frac12\tr\DerE{k+1}(\trnnE{(\btens{\tau}_{V_1})},\trnnE{(\btens{\tau}_{V_2})}, \btens{\tau}_E)~v\\
  &\quad
  + \int_E \nFE^\top\Ctensor\DerE{k+1}(\trnnE{(\btens{\tau}_{V_1})},\trnnE{(\btens{\tau}_{V_2})}, \btens{\tau}_E)\nFE~v
  + \jump{E}{\nFE^\top \, \btens{\tau}_V \, \nFE\,v}.
\end{align*}
Therefore,
\begin{align*}
  \term_{9} + \term_{4} + \term_{11} 
  =& \sum_{F\in\FT}\wTF \sumEF\left(
  \frac12\tr\DerE{k+1}(\trnnE{(\btens{\tau}_{V_1})},\trnnE{(\btens{\tau}_{V_2})}, \btens{\tau}_E)~v
  -  \tr \btens{\tau}_E\,\partial_{\tE} v
  \right).
\end{align*}
Using \eqref{eq:zero-edges} with $\phi_E$ equal to the integrand in the above expression readily yields $\term_{4} + \term_{9} + \term_{11} = 0$,
thus showing that $\term = 0$ and therefore concluding the proof.


\section{Local exactness}\label{sec:exactness}

This section contains the proof of the relations \eqref{eq:exactness}, yielding the exactness of the local complex \eqref{eq:ddr.complex}.

\subsection{Proof of \eqref{eq:exact.RTDG}}

Having already proved \eqref{eq:IGrad.RT.subset.DGT}, we only need to show that
\begin{equation}\label{eq:Ker.DGT.subset.IDGrad.RT1}
  \Ker \DG{T}\subset \IDGrad{T} \RT{1}(T).
\end{equation}
To this purpose, we let $\uvec{v}_T \in \uHdevgrad{T}$ be such that $\DG{T} \uvec{v}_T = \uvec{0}$ and show the existence of $\bvec{w}\in\RT{1}(T)$ such that $\uvec{v}_T = \IDGrad{h}\bvec{w}$.
We start from the vertex and edge components, 
  which provide the expression for $\bvec{w}$, then show that the face and element components are also equal to the interpolate of $\bvec{w}$.

\subsubsection{Vertex and edge components}

Given an edge $E\in\ET$ with vertices $V_1$ and $V_2$, enforcing $\DG{E} \uvec{v}_E = \uvec{0}$, corresponds to the following conditions (cf. \eqref{eq:defuDG}):
\begin{subequations}\label{eq:DGE=0}
  \begin{alignat}{2}\label{eq:DGE=0:1}
    \DEV \gdofV[V_2]{v}
    =  \DEV \gdofV[V_1]{v}
    &= \btens{0},
    \\ \label{eq:DGE=0:2}
    \DerE{k+1}(\trnnE{(\gdofV[V_1]{v})},\trnnE{(\gdofV[V_2]{v})}, \gdofE{v}) &= \btens{0},
    \\ \label{eq:DGE=0:3}
    \gdofE{v} - \frac13\left(
    \tr \gdofE{v} +
    \DerE{k}(\bvec{v}_{V_1} \cdot \tE, \bvec{v}_{V_2} \cdot \tE, \doftE{v})
    \right)\ID{2}
    &= \btens{0},
    \\ \label{eq:DGE=0:4}            
    \DerE{k+1}(\trnE{(\bvec{v}_{V_1})},\trnE{(\bvec{v}_{V_2})}, \dofnE{v}) &= 0.
  \end{alignat}
\end{subequations}
From \eqref{eq:DGE=0:1} we infer, for all $V\in\VE$, the existence of $\lambda_V\in\Real$ such that $\gdofV{v} = \lambda_V \ID{3}$.
Condition \eqref{eq:DGE=0:2} then gives $\lambda_{V_1}\ID{2} = \lambda_{V_2}\ID{2} =  \gdofE{v}$ which implies, in particular, $\lambda_{V_1} = \lambda_{V_2}$.
Since this reasoning applies to all edges $E\in\ET$, this yields the existence of $\lambda\in\Real$ such that $\lambda_V = \lambda$ for all $V\in\VT$ and $\gdofE{v} = \lambda \ID{2}$ for all $E\in\ET$.
Substituting this value of $\gdofE{v}$ in \eqref{eq:DGE=0:3} results in $\DerE{k+1}(\bvec{v}_{V_1} \cdot \tE, \bvec{v}_{V_2} \cdot \tE, \doftE{v}) = \lambda$, which gives, accounting for \eqref{eq:IPP.E}, $(\bvec{v}_{V_2} - \bvec{v}_{V_1}) \cdot \tE = \lambda h_E$ and
$\doftE{v}(\bvec{x}) = \lproj{k-1}{E}\big[\bvec{v}_{V_1} + \lambda (\bvec{x} - \bvec{v}_{V_1})\big]\cdot\tangent_E$ 
for all $\bvec{x}\in E$.
Condition \eqref{eq:DGE=0:4}, on the other hand, gives $\trnE{(\bvec{v}_{V_1})} = \trnE{(\bvec{v}_{V_2})} = \dofnE{v}$.
Combining the above results on the tangential and normal components of the vertex values $\bvec{v}_V$ yields
\[
\bvec{v}_{V_2}
= \bvec{v}_{V_1} + \lambda h_E \tE
= \bvec{v}_{V_1} + \lambda\, 
  \left[
    (\bvec{v}_{V_2} - \bvec{v}_{V_1})\cdot\tE
    \right]\tE
\qquad\forall E\in\ET.
\]
The only possibility for this condition to hold is that there exists $\RT{1}(T)\ni\bvec{w}:T\ni\bvec{x}\mapsto\bvec{a} + \lambda(\bvec{x} - \bvec{x}_T)\in\Real^3$ with $\bvec{a}\in\Real^3$ 
such that $\bvec{v}_V = \bvec{w}(\bvec{x}_V)$ for a given vertex $V\in\VT$ 
(which is sufficient for $\bvec{v}_{V'} = \bvec{w}(\bvec{x}_{V'})$ to hold also for all $V'\in\VT\setminus\{V\}$).
We can easily check, recalling the definition of the interpolator on $\uHdevgrad{E}$ (which corresponds to the restriction to $E$ of \eqref{eq:IDGrad}), that the above conditions on the components of $\uvec{v}_E$ amount to
\begin{equation} \label{eq:exact.RTDG.proof1}
  \uvec{v}_E = \IDGrad{E} \bvec{w}\qquad\forall E \in \ET.
\end{equation}

\subsubsection{Face components}

Let now $F\in\FT$. 
Enforcing $\DG{F}\uvec{v}_F = \uvec{0}$ amounts to the following conditions, in addition to \eqref{eq:DGE=0}:
\[
\text{
  $\DGntF\uvec{v}_F = \bvec{0}$\quad and\quad
  $\DGttF\uvec{v}_F = \btens{0}$.
}
\]
Enforcing $\DGntF\uvec{v}_F = \bvec{0}$ in \eqref{eq:defDGntF} written for $\bvec{w} = \bvec{z}\in\vPoly{k+1}(F;\Real^2)$, and accounting for \eqref{eq:exact.RTDG.proof1} gives,
\[
- \int_F \dofnF{v} \DIVF \bvec{z}
+ \sum_{E \in \EF} \wFE \int_E (\bvec{w} \cdot \nF)(\bvec{z}\cdot\nFE)
= 0\qquad\forall \bvec{z} \in \vPoly{k+1}(F;\Real^2).
\]
Integrating by parts the boundary terms and noticing that $\GRADF(\bvec{w} \cdot \nF) = \bvec{0}$ since the function 
$T\ni\bvec{x}\mapsto \bvec{w}(\bvec{x})\cdot\nF\in\Real$
is constant on $F$ (see \cite[Proposition~8]{Di-Pietro.Droniou:23} for a proof of this result on general meshes), 
the above condition translates to
$\int_F ( \bvec{w}\cdot \nF - \dofnF{v} ) \DIVF \bvec{z} = 0$
for all $\bvec{z} \in \vPoly{k+1}(F;\Real^2)$.
Since $\DIVF : \vPoly{k+1}(F;\Real^2) \rightarrow \Poly{k}(F)$ is onto, this implies
\begin{equation}\label{eq:exactness.dofnF.v}
  \dofnF{v} = \bvec{w}_{\vert F}\cdot\nF.
\end{equation}

Enforcing then $\DGttF\uvec{v}_F = \btens{0}$ in \eqref{eq:defDGttF}, 
removing projectors according to their respective definition,
and using the integration by parts formula \eqref{eq:IPP.DG.Ftt}, we get, 
for all $\btens{\sigma} \in \CGtrimPoly{k}(F)$,
\[
\int_F ( \trtF{\bvec{w}} - \doftF{v}) \cdot \VDIVF\btens{\sigma}
+ \int_F \left(\lambda - \frac13\gdofF{v}\right) \tr \btens{\sigma}
= 0.
\]
Taking $\btens{\sigma} \in \CGoly{\compl,k}(F)$ and using the fact that $\VDIVF : \CGoly{\compl,k}(F) \rightarrow \vPoly{k-1}(F;\Real^2)$ is onto (cf. Lemma \ref{lemma:DIVFCG}) along with $\Tr \btens{\sigma} = 0$, this condition yields
\begin{equation}\label{eq:exactness.doftF.v}
  \doftF{v} = \vlproj{k-1}{F}\trtF{\bvec{w}}.
\end{equation}
Taking $\btens{\sigma} \in \CGoly{k-1}(F)$, using the fact that $\VDIVF \btens{\sigma} = 0$ and that $\Tr \CGoly{k-1}(F) \rightarrow \Poly{k-1}(F)$ is onto, we have, on the other hand
\begin{equation}\label{eq:exactness.gdofF.v}
  \gdofF{v} = 3 \lambda.
\end{equation}
Gathering \eqref{eq:exact.RTDG.proof1}, \eqref{eq:exactness.dofnF.v}, \eqref{eq:exactness.doftF.v}, and \eqref{eq:exactness.gdofF.v}, and recalling that the above reasoning holds for any $F\in\FT$, we have thus proved that
\begin{equation} \label{eq:exact.RTDG.proof2}
  \uvec{v}_F = \IDGrad{F} \bvec{w}\qquad\forall F\in\FT.
\end{equation}  

\subsubsection{Element component}

To conclude the proof of \eqref{eq:Ker.DGT.subset.IDGrad.RT1}, it only remains to show that $\DGT\uvec{v}_T = \btens{0}$ implies
\begin{equation}\label{eq:exactness.dofT.v}
  \bvec{v}_T = \vlproj{k-1}{T}\bvec{w}.
\end{equation}
This relation reduces to the trivial identity $\bvec{0} = \bvec{0}$ for $k=0$.
Let us then consider the case $k\ge 1$.
Enforcing $\DGT\uvec{v}_T = \btens{0}$ in \eqref{eq:defDGT} and accounting for \eqref{eq:exact.RTDG.proof2}, \eqref{eq:defDGT} gives, for all $\btens{\sigma} \in \SRoly{\compl,k}(T) \subset \SRtrimPoly{k}(T)$,
\[
0 =
-\int_T \bvec{v}_T \cdot \VDIV\btens{\sigma}
+ \sum_{F \in \FT} \wTF 
\int_F\left(
(\bvec{w}\cdot\nF)\, \nF^\top \btens{\sigma} \nF  + \cancel{\vlproj{k-1}{F}}\hdoftF{w}^\top \btens{\sigma} \nF
\right)
=
\int_T (\bvec{w} - \bvec{v}_T) \cdot \VDIV\btens{\sigma},
\]
where we have used \eqref{eq:poly.trace.SR} to remove the projector and the integration by parts formula \eqref{eq:IPP.DG.T} (after noticing that $\btens{\sigma} \in \SRoly{\compl,k}(T)$ is traceless) to conclude.
Since $\VDIV : \SRoly{\compl,k}(T) \rightarrow \vPoly{k-1}(T;\Real^3)$ is onto, this relation implies \eqref{eq:exactness.dofT.v}, thus concluding the proof of \eqref{eq:Ker.DGT.subset.IDGrad.RT1}.

\subsection{Proof of \eqref{eq:exact.DGSC}}

Let $\uvec{\tau}_T \in \uHsymcurl{T}$ be such that $\SC{T} \uvec{\tau}_T = \uvec{0}$, i.e., recalling \eqref{eq:defSCE} and \eqref{eq:defuSC}:
\begin{alignat}{2}\label{eq:SCT=0:1} 
  \Ctensor\left(
  \cdofE{\tau} -
  \DerE{k+1}(\trnnE{(\btens{\tau}_{V_1})},\trnnE{(\btens{\tau}_{V_2})}, \btens{\tau}_E)
  \right) &= \btens{0}
  &\qquad& \forall E \in \ET,\\ \label{eq:SCT=0:2}
  \SCddofF \utens{\tau}_F &= 0
  &\qquad& \forall F \in \FT, \\ \label{eq:SCT=0:3}
  \SCnnF \utens{\tau}_F &= 0
  &\qquad& \forall F\in\FT, \\ \label{eq:SCT=0:4}
  \SCT \utens{\tau}_T &= \btens{0} .
\end{alignat}
In order to show that $\Ker \SC{T} \subset \Image \DG{T}$, starting from the above conditions we will explicitly construct $\uvec{v}_T \in \uHdevgrad{T}$ such that
\begin{equation}\label{eq:DT.vT=tauT}
  \DG{T} \uvec{v}_T = \uvec{\tau}_T
\end{equation}
  determining, in this order, its vertex components (cf. \eqref{eq:exact.DGSC.defGV}),
  edge components (cf. \eqref{eq:exact.1.E}),
  face components (cf. \eqref{eq:ex.1.F.def1}, \eqref{eq:ex.1.F.def2}, and \eqref{eq:ex.1.F.def3}),
  and element component (cf. \eqref{eq:ex.1.T.1}).%

\subsubsection{Vertex components}
We infer from \eqref{eq:SCT=0:1} and from the definition \eqref{eq:defCtensor} of $\Ctensor$
the existence of $\lambda_E \in \Poly{k+1}(E)$ such that
\begin{equation} \label{eq:exact.DGSC.proofdefLE}
  \cdofE{\tau} = 
  \DerE{k+1}(\trnnE{(\btens{\tau}_{V_1})},\trnnE{(\btens{\tau}_{V_2})}, \btens{\tau}_E)
  + \lambda_E\ID{2}.
\end{equation}
Evaluating then \eqref{eq:defSCnnF} for any $F\in\FT$ and $r \equiv 1 \in \Poly{k+1}(F)$ and enforcing \eqref{eq:SCT=0:3} gives
\begin{equation} \label{eq:ex.DGSC.nFloop}
  0 = -\sum_{E\in\EF}\wFE\int_E \doftE{\tau} \cdot \nF.
\end{equation}
On the other hand, enforcing \eqref{eq:SCT=0:2} in the definition \eqref{eq:defSCddofF} of $\SCddofF$ 
written for a generic $r \in \Poly{k+1}(F)$ gives, using \eqref{eq:decomp.Ceta} with $\btens{\eta} = \btens{\tau}_E$ 
and expressing $\cdofE{\tau}$ according to \eqref{eq:exact.DGSC.proofdefLE},%
\begin{equation} \label{eq:ex.DGSC.tFloop}
  \begin{aligned}
    0 
    &=
    -\int_F \cgvec{\btens{\tau}} \tdot \CURL_F \GRAD_F r
    + \sum_{E\in\EF}\wFE \int_E (\doftE{\tau} \cdot \nFE) \partial_{\nFE}\, r
    - \sumEF \tr \btens{\tau}_E \, \partial_{\tE} r
    \\
    &\quad
    \cancel{- \sumEF \nFE^\top \btens{\tau}_E \nFE\, \partial_{\tE}r}
    \cancel{+ \sum_{E\in\EF}\wFE\jump{E}{\nFE^\top \btens{\tau}_V \nFE\, r}}
    \\
    &\quad 
    \cancel{- \sumEF \nFE^\top
      \DerE{k+1}(\trnnE{(\btens{\tau}_{V_1})},\trnnE{(\btens{\tau}_{V_2})}, \btens{\tau}_E)\nFE  \, r
    }
    - \sumEF {\lambda}_E  \, r,
  \end{aligned}
\end{equation}
where the cancellations follows from \eqref{eq:IPP.E}.
Taking $r \equiv 1$ in the above expression gives $\sumEF \lambda_E = 0$ for all $F\in\FT$.
Since the first Betti number of $T$
is equal to $0$, we infer from 
this relation that the piecewise function equal to $\int_E\lambda_E$ on each $E\in\ET$ can be regarded as the gradient of a piecewise affine function on the edge skeleton of $T$, i.e., there exists a family 
$(\lambda_V)_{V\in\VT} \in \Real^{\VT}$ such that
\begin{equation}\label{eq:exact.DGSC.defLV}
  \int_E \lambda_E = \jump{E}{\lambda_V}
  \qquad\forall E\in\ET.
\end{equation}
We infer from \eqref{eq:ex.DGSC.nFloop} and \eqref{eq:ex.DGSC.tFloop} that, 
for all $r \in \Poly{1}(T)$ and all $F\in\FT$,
\begin{equation}\label{eq:ex.DGSC.1}
\sum_{E\in\EF}\wFE
\int_E \left(
\hdoftE{\tau} \cdot \GRAD r
- \tr \btens{\tau}_E\, \partial_{\tE} r
-  \lambda_E\, r
\right) = 0.
\end{equation}
We construct a family $(\bvec{z}_V)_{V\in\VT} \in (\Real^3)^{\VT}$ such that, for all $E\in\ET$,
\begin{alignat}{2} \label{eq:ex.DGSC.defvEn}
  \jump{E}{\trnE{(\bvec{z}_V)}\cdot \overline{\bvec{w}}}
  &=
  \int_E \doftE{\tau} \cdot \overline{\bvec{w}}
  &\qquad& \forall \overline{\bvec{w}} \in \vPoly{0}(E;\Real^2),
  \\ \label{eq:ex.DGSC.defvEt}
  \jump{E}{(\bvec{z}_V \cdot \tE) \dotp{r}}
  &=
  - \int_E \left(
  \tr \btens{\tau}_E \, \dotp{r} + \lambda_E \, r
  \right)
  + \jump{E}{\lambda_V\, r}
  &\qquad& \forall r \in \Poly{1}(E)
\end{alignat}
as follows: we first set an arbitrary value on a vertex $V_0$, then choose the value on neighboring vertices according to the relations \eqref{eq:ex.DGSC.defvEn} and \eqref{eq:ex.DGSC.defvEt}.
The relation \eqref{eq:ex.DGSC.1} ensures that this construction is consistent.
Indeed, any path leading to a given vertex will give the same value for that vertex, because the difference between two paths is a closed loop which can be realized as the boundary of the union of some faces $F\in\FT$ 
(since the first Betti number of $T$ is zero).

We conclude fixing the vertex components of the sought $\uvec{v}_T\in\uHdevgrad{T}$ as follows:
\begin{equation} \label{eq:exact.DGSC.defGV}
  \text{
    $\bvec{v}_V \coloneq \bvec{z}_V$\quad and\quad
    $\gdofV{v} \coloneq \btens{\tau}_V + \lambda_V \ID{3}$
    \quad for all $V\in\VT$.
  }
\end{equation}
With this choice it holds, for all $V\in\VE$,
\begin{equation}\label{eq:exact.dev.gdofV=tauV}
  \DEV \gdofV{v}
  = \DEV \btens{\tau}_V + \lambda_V \DEV \ID{3}
  = \DEV \btens{\tau}_V = \btens{\tau}_V,
\end{equation}
where the conclusion follows observing that $\btens{\tau}_V$ is traceless.

\subsubsection{Edge components}

We next identify suitable edge components for $\uvec{v}_T$ satisfying \eqref{eq:DT.vT=tauT}.
Specifically, for any $E\in\ET$, we define $\gdofE{v} \in \tPoly{k}(E;\Real^{2\times2})$,
$\dofnE{v}\in\tPoly{k}(E;\Real^2)$,
and $\doftE{v} \in \Poly{k-1}(E)$
such that,
for all $\widetilde{\btens{\sigma}}\in\tPoly[0]{k+1}(E;\Real^{2\times2})$,
all $\widetilde{\bvec{w}}\in\vPoly[0]{k+1}(E;\Real^2)$,
and all $r\in \Poly{k+1}(E)$ such that $\lproj{1}{E} r = 0$,
\begin{subequations}\label{eq:exact.1.E}
  \begin{gather}\label{eq:exact.DGSC.defGE}
    \int_E \gdofE{v} \tdot \dotp{\widetilde{\btens{\sigma}}}
    =
    -\int_E \cdofE{\tau}\tdot\widetilde{\btens{\sigma}}
    + \jump{E}{(\trnnE{(\btens{\tau}_V)} + \lambda_V\ID{2})\tdot \widetilde{\btens{\sigma}}},
    \\ \label{eq:exact.DGSC.proof3}
    \int_E \dofnE{v} \cdot \dotp{\widetilde{\bvec{w}}}
    =
    - \int_E \doftE{\tau} \cdot \widetilde{\bvec{w}}
    + \jump{E}{\trnE{(\bvec{v}_V)} \cdot \widetilde{\bvec{w}}},
    \\ \label{eq:exact.DGSC.proof5} 
    \int_E \doftE{v} \,\dotpp{r}
    = \jump{E}{(\bvec{v}_V \cdot \tE) \dotp{r}}
    + \int_E \tr \btens{\tau}_E\, \dotp{r}
    + \int_E \lambda_E \, r - \jump{E}{\lambda_V\, r}.
  \end{gather}
\end{subequations}
Let us check that, for any edge $E\in\ET$, $\DG{E}\uvec{v}_E = \utens{\tau}_E$, where we remind the reader that the components of $\DG{E}\uvec{v}_E$ are extracted from \eqref{eq:defuDG}.
The fact that the vertex components coincide is expressed by \eqref{eq:exact.dev.gdofV=tauV}, so we only need to consider the edge components.

It holds, for all $\widetilde{\btens{\sigma}} \in \tPoly[0]{k+1}(E;\Real^{2\times 2})$, 
\begin{multline} \label{eq:exact.DGSC.gdofEk+1}
  \int_E
  \DerE{k+1}(\trnnE{({\gdofV[V_1]{v}})},\trnnE{({\gdofV[V_2]{v}})}, \gdofE{v})
  \tdot \widetilde{\btens{\sigma}}
  \\
  \overset{\eqref{eq:IPP.E},\,\eqref{eq:exact.DGSC.defGV}}=
  - \int_E \gdofE{v} \tdot \dotp{\widetilde{\btens{\sigma}}}
  + \jump{E}{(\trnnE{(\btens{\tau}_V)}  + \lambda_V\ID{2})\tdot \widetilde{\btens{\sigma}}}
  \overset{\eqref{eq:exact.DGSC.defGE}}= \int_E \cdofE{\tau}\tdot\widetilde{\btens{\sigma}}.
\end{multline}
On the other hand, for all $\overline{\btens{\sigma}} \in \tPoly{0}(E;\Real^{2\times2})$, \eqref{eq:IPP.E} together with \eqref{eq:exact.DGSC.defGV} gives
\begin{equation}\label{eq:exact.DGSC.gdofE0}
  \begin{aligned}
    &\int_E
    \DerE{k+1}(\trnnE{({\gdofV[V_1]{v}})},\trnnE{({\gdofV[V_2]{v}})}, \gdofE{v})
    \tdot \overline{\btens{\sigma}}
    \\
    &\qquad
    =
    \jump{E}{\trnnE{(\btens{\tau}_V)}\tdot \overline{\btens{\sigma}}} + \jump{E}{\lambda_V \ID{2}\tdot \overline{\btens{\sigma}}}
    \overset{\eqref{eq:exact.DGSC.defLV}}{=}
    \jump{E}{\trnnE{(\btens{\tau}_V)}\tdot \overline{\btens{\sigma}}} + \int_E \lambda_E \ID{2}\tdot\overline{\btens{\sigma}}
    \\
    &\qquad
    = \int_E \left[
      \DerE{k+1}(\trnnE{(\btens{\tau}_{V_1})},\trnnE{(\btens{\tau}_{V_2})}, \btens{\tau}_E)
      + \lambda_E \ID{2}
      \right]\tdot\overline{\btens{\sigma}}
    \overset{\eqref{eq:exact.DGSC.proofdefLE}}{=} \int_E \cdofE{\tau} \tdot \overline{\btens{\sigma}},
  \end{aligned}
\end{equation}
where we have used \eqref{eq:IPP.E} along with the fact that $\dotp{\overline{\btens{\sigma}}} = \btens{0}$ (since $\overline{\btens{\sigma}}$ is constant) to write 
$\jump{E}{\trnnE{(\btens{\tau}_V)}\tdot \overline{\btens{\sigma}}} = \int_E \DerE{k+1}(\trnnE{(\btens{\tau}_{V_1})},\trnnE{(\btens{\tau}_{V_2})}, \btens{\tau}_E) \tdot \overline{\btens{\sigma}}$ in the third equality.
Summing \eqref{eq:exact.DGSC.gdofEk+1} and \eqref{eq:exact.DGSC.gdofE0} and noticing that $\widetilde{\sigma} + \overline{\sigma}$ spans $\Poly{k+1}(E;\Real^{2\times 2})$ as $(\widetilde{\sigma},\overline{\sigma})$ spans $\tPoly[0]{k+1}(E;\Real^{2\times 2})\times\tPoly{0}(E;\Real^{2\times2})$, we conclude that
\begin{equation} \label{eq:ex.1.E.0}
  \DerE{k+1}(\trnnE{({\gdofV[V_1]{v}})},\trnnE{({\gdofV[V_2]{v}})}, \gdofE{v}) =  \cdofE{\tau}.
\end{equation}

Next, for all $\bvec{w}\in\Poly{k+1}(E;\Real^2)$, writing $\bvec{w} = \overline{\bvec{w}} + \widetilde{\bvec{w}}$ with $\overline{\bvec{w}}\coloneq\vlproj{0}{E}\bvec{w}$, we have
\[
\begin{aligned}
  \int_E
  \DerE{k+1}(\trnE{(\bvec{v}_{V_1})},\trnE{(\bvec{v}_{V_2})}, \dofnE{v}) \cdot \bvec{w}
  \overset{\eqref{eq:IPP.E}}&=
  - \int_E \dofnE{v} \cdot \dotp{\widetilde{\bvec{w}}}
  + \jump{E}{\trnE{(\bvec{v}_V)} \cdot \widetilde{\bvec{w}}}
  + \jump{E}{\trnE{(\bvec{v}_V)} \cdot \overline{\bvec{w}}}
  \\
  \overset{\eqref{eq:exact.DGSC.proof3},\,\eqref{eq:ex.DGSC.defvEn}}&=
  \int_E \doftE{\btens{\tau}} \cdot \bvec{w},
\end{aligned}
\]
so that
\begin{equation} \label{eq:ex.1.E.ttau}
  \DerE{k+1}(\trnE{(\bvec{v}_{V_1})},\trnE{(\bvec{v}_{V_2})}, \dofnE{v}) = \doftE{\tau}.
\end{equation}

To conclude the equality of the edge components, we have to prove that
\begin{equation} \label{eq:ex.1.E.2}
  \gdofE{v}
  - \frac13\left(
  \tr \gdofE{v} + \DerE{k}(\bvec{v}_{V_1} \cdot \tE, \bvec{v}_{V_2} \cdot \tE, \doftE{v})
  \right) \ID{2} 
  = \btens{\tau}_E.
\end{equation}
To this purpose, we start by noticing that, for all $\bvec{\sigma} \in \tPoly[0]{k+1}(E;\Real^{2\times2})$ letting, for the sake of brevity, $\check{\btens{\sigma}}\coloneq\btens{\sigma} - \frac13(\tr\btens{\sigma})\ID{2}$,
\begin{equation}\label{eq:ex.DGSC.dofEIPP}
  \begin{aligned}
    \term
    &\coloneq\int_E \left[
      \gdofE{v}
      - \frac13\left(
      \tr \gdofE{v}
      + \DerE{k}(\bvec{v}_{V_1} \cdot \tE, \bvec{v}_{V_2} \cdot \tE, \doftE{v})
      \right)\ID{2}
      \right] 
    \tdot \dotp{\btens{\sigma}}
    \\
    &=
    \int_E \gdofE{v} \tdot\dotp{\check{\btens{\sigma}}}
    - \frac13 \int_E  \DerE{k}(\bvec{v}_{V_1} \cdot \tE, \bvec{v}_{V_2} \cdot \tE, \doftE{v}) \, \tr(\dotp{\btens{\sigma}})
    \\
    \overset{\eqref{eq:exact.DGSC.defGE},\,\eqref{eq:IPP.E}}&=
    -\int_E \cdofE{\tau} \tdot \check{\btens{\sigma}}
    + \jump{E}{\big(\trnnE{(\btens{\tau}_V)} + \lambda_V \ID{2} \big) \tdot \check{\btens{\sigma}}}
    + \frac13 \int_E \doftE{v} \, \tr(\dotpp{\btens{\sigma}})
    \\
    &\quad
    - \frac13 \jump{E}{(\bvec{v}_V \cdot \tE) \tr(\dotp{\btens{\sigma}})}
    \\
    &=
    -\int_E \cdofE{\tau} \tdot \check{\btens{\sigma}}
    + \jump{E}{\trnnE{(\btens{\tau}_V)}  \tdot \check{\btens{\sigma}}}
    + \frac13 \jump{E}{\lambda_V\,\tr \btens{\sigma}}
    + \frac13 \int_E \doftE{v} \, \tr(\dotpp{\btens{\sigma}})
    \\
    &\quad
    - \frac13 \jump{E}{(\bvec{v}_V \cdot \tE) \tr(\dotp{\btens{\sigma}})}
    \\
    \overset{\eqref{eq:exact.DGSC.proofdefLE},\,\eqref{eq:IPP.E}}&=
    \int_E \btens{\tau}_E \tdot \dotp{\check{\btens{\sigma}}}
    - \frac13 \int_E \lambda_E\, \tr \btens{\sigma}
    + \frac13 \jump{E}{\lambda_V\,\tr \btens{\sigma}}
    + \frac13 \int_E \doftE{v} \, \tr(\dotpp{\btens{\sigma}})
    \\
    &\quad   
    - \frac13 \jump{E}{(\bvec{v}_V \cdot \tE) \tr(\dotp{\btens{\sigma}})},
  \end{aligned}
\end{equation}
where we have additionally used the fact that $\ID{2} \tdot \check{\btens{\sigma}} = \frac13 \tr \btens{\sigma}$ in the fourth equality.
Taking $\btens{\sigma} = \btens{\upsilon}$ with $\btens{\upsilon}\in\tPoly[0]{k+1}(E;\Real^{2\times2})$ such that $\tr \btens{\upsilon} = 0$, \eqref{eq:ex.DGSC.dofEIPP} yields
\begin{equation} \label{eq:ex.2.E.1}
  \term = \int_E \btens{\tau}_E \tdot \dotp{\btens{\upsilon}} .
\end{equation}
For $\btens{\sigma}$ such that $\btens{\sigma} = r \ID{2}$ for some $r \in \Poly[0]{k+1}(E)$, \eqref{eq:ex.DGSC.dofEIPP} gives
\begin{equation} \label{eq:ex.2.E.2}
  \term
  = \frac13 \int_E \tr {\btens{\tau}_E} \, \dotp{r}
  + \frac23 \left(
  \int_E \doftE{v}\, \dotpp{r} - \jump{E}{(\bvec{v}_V\cdot\tE) \dotp{r}}
  - \int_E \lambda_E\, r + \jump{E}{\lambda_V \, r}
  \right).
\end{equation}
If, in particular, we take $r = \widetilde{r}$ such that $\lproj{1}{E}\widetilde{r} = 0$, plugging the definition \eqref{eq:exact.DGSC.proof5} of $\doftE{v}$ into \eqref{eq:ex.2.E.2} yields
\begin{equation} \label{eq:ex.2.E.2.0}
  \term = \int_E \tr \btens{\tau}_E \dotp{\widetilde{r}}.
\end{equation}
On the other hand, taking $q = \overline{r} \in \Poly[0]{1}(E)$, 
we infer from \eqref{eq:exact.DGSC.defGV}, \eqref{eq:ex.DGSC.defvEt}, and \eqref{eq:ex.2.E.2} that
\begin{equation} \label{eq:ex.2.E.2.1}
  \term
  = \frac13 \int_E \tr \btens{\tau}_E \, \dotp{\overline{r}}
  + \frac23 \int_E\tr \btens{\tau}_E \, \dotp{\overline{r}}
  =  \int_E \tr \btens{\tau}_E \, \dotp{\overline{r}}
\end{equation}
Noticing that $\dotp{\left[\btens{\upsilon} + \left(\widetilde{r} + \overline{r}\right)\ID{2}\right]}$ spans $\tPoly{k}(E;\Real^{2\times2})$ when $\btens{\upsilon}$ spans the zero-trace subspace of $\tPoly[0]{k+1}(E;\Real^{2\times 2})$,
$\widetilde{r}$ spans the subspace of functions in $\Poly{k+1}(E)$ with zero $L^2$-orthogonal projection on $\Poly{1}(E)$,
and $\overline{r}$ spans $\Poly[0]{1}(E)$, we
conclude from \eqref{eq:ex.2.E.1}, \eqref{eq:ex.2.E.2.0}, and \eqref{eq:ex.2.E.2.1} that \eqref{eq:ex.1.E.2} holds.
Combining this relation with \eqref{eq:exact.dev.gdofV=tauV}, \eqref{eq:ex.1.E.0}, and \eqref{eq:ex.1.E.ttau} gives
\[
\DG{E} \uvec{v}_E = \utens{\tau}_E
\qquad\forall E \in \ET.
\]

\subsubsection{Face components}

Let $F\in\FT$.
Since $\DIVF:\Roly{\compl,k+1}(F)\to\Poly{k}(F)$ is an isomorphism, there exists a unique $\dofnF{v} \in \Poly{k}(F)$ such that, 
for all $\bvec{w}\in\Roly{\compl,k+1}(F)$,
\begin{equation} \label{eq:ex.1.F.def1}
  \int_F \dofnF{v} \DIVF \bvec{w}
  = - \int_F \rttvec[F]{\tau} \cdot \bvec{w}
  + \sum_{E \in \EF} \wFE \int_E 
  (\dofnE{v} \cdot \nF)
  (\bvec{w}\cdot\nFE),
\end{equation}
with $\dofnE{v}$ defined by \eqref{eq:exact.DGSC.proof3}.
Plugging this $\dofnF{v}$ into the definition \eqref{eq:defDGntF} of $\DGntF$
with test function $\bvec{w}$ in $\Roly{\compl,k+1}(F)\subset\RT{k+1}(F)$, we get
\begin{equation}\label{eq:cRproj.DGntF=cRproj.tauRF}
  \cRproj[F]{k+1}\DGntF \uvec{v}_F\, =\, \cRproj[F]{k+1}{\rttvec[F]{\tau}}.
\end{equation}

Recalling the definition \eqref{eq:Poly=CGoly+cCGoly} of $\CGoly{k-1}(F)$ and using
the decomposition 
$\vPoly{k}(F;\Real^2) = \GRAD_F \Poly{k+1}(F) \oplus \bvec{x}^\perp \Poly{k-1}(F)$ (see \cite{Arnold:18}), 
we can write 
\begin{equation} \label{eq:ex.1.F.dec1}
  \CGoly{k-1}(F) = \CURL_F \GRAD_F \Poly{k+1}(F) \oplus \CURL_F \bvec{x}^\perp \Poly{k-1}(F).
\end{equation}
Since $\tr:\CURL_F \bvec{x}^\perp \Poly{k-1}(F)\to\Poly{k-1}(F)$ is an isomorphism
(notice that $\Poly{k-1}(F) = \tr \vPoly{k-1}(F;\Real^2) \overset{\eqref{eq:Poly=CGoly+cCGoly}}
{=} \tr \CGoly{k-1}(F) \overset{\eqref{eq:ex.1.F.dec1}}{=} 
\tr (\CURL_F \bvec{x}^{\perp} \Poly{k-1}(F))$ since $\tr \cGoly{k-1}(F) = 0$ 
and $\tr \CURL_F \GRAD_F = \ROTF \GRAD_F = 0$ and count the dimensions), 
we can define uniquely $\gdofF{v} \in \Poly{k-1}(F)$ enforcing the following condition:
For all $\btens{\sigma} \in \CURL_F\bvec{x}^\perp \Poly{k-1}(F)$,
\begin{equation} \label{eq:ex.1.F.def2}
  \begin{aligned}
  \frac{1}{3} \int_F \gdofF{v}\,\tr \btens{\sigma}
  &= \sumEF (\dofnE{v} \cdot \nFE)\, \nFE^\top\btens{\sigma}\nFE
  \\
  &\quad
  + \sumEF \doftE{v}\,\tE^\top\btens{\sigma}\nFE
  -\int_F \cgvec{\btens{\tau}} \tdot \btens{\sigma}.
  \end{aligned}
\end{equation}

Likewise, since $\VDIVF:\CGoly{\compl,k}(F)\to\vPoly{k-1}(F;\Real^2)$ is an isomorphism 
(see Lemma~\ref{lemma:DIVFCG} below), $\eqref{eq:defDGttF}$ yields a unique 
\begin{equation} \label{eq:ex.1.F.def3}
  \text{
    $\doftF{v} \in \vPoly{k-1}(F;\Real^2)$ such that
    $\cCGproj[F]{k}\DGttF \uvec{v}_F\, =\, \cCGproj[F]{k}\cgtvec[F]{\tau}$.
  }
\end{equation}

We next check that the face components defined above (along with the vertex components defined by \eqref{eq:exact.DGSC.defGV} and the edge components defined by \eqref{eq:exact.1.E}) yield the equality of the face components in \eqref{eq:DT.vT=tauT}.
Enforcing $\SCnnF\utens{\tau}_F = 0$ (cf. \eqref{eq:SCT=0:3}) in the definition \eqref{eq:defSCnnF} of $\SCnnF$, we get
\begin{equation}\label{eq:rtt.F.tau=sum.E.doftE.tau}
  \int_F \rttvec[F]{\tau}\cdot \CURLF r
  = \sum_{E\in\EF}\wFE\int_E (\doftE{\tau} \cdot \nF) \, r
  \qquad \forall r\in \Poly{k+1}(F).
\end{equation}
  Letting $r \in \Poly{k+1}(F)$, writing
the definition \eqref{eq:defDGntF} of $\DGntF\uvec{v}_F$ for $\bvec{w} = \CURLF r\in\vPoly{k}(F;\Real^2)\subset\RT{k+1}(F)$, 
and using the fact that $\DIVF \CURLF r = 0$ and $\CURLF r \cdot \nFE = - \partial_{\tE} r$ gives
\begin{equation*}
  \begin{aligned}
    &\int_F \DGntF\uvec{v}_F \cdot \CURLF r
    \\
    &\quad=
    -\sum_{E\in\EF}\wFE\int_E (\dofnE{v} \cdot \nF)\,\dotp{r}
    \\
    \overset{\eqref{eq:IPP.E}}&\quad=
    \sum_{E\in\EF}\wFE\int_E\DerE{k+1}(\trnE{(\btens{v}_{V_1})},\trnE{(\btens{v}_{V_2})}, \dofnE{v}) \cdot \nF
    \, r -
    \cancel{\sum_{E\in\EF}\wFE\jump{E}{(\bvec{v}_V \cdot\nF) \, r}} \\
    \overset{\eqref{eq:ex.1.E.ttau}}&\quad=
    \sum_{E\in\EF}\wFE\int_E (\doftE{\tau} \cdot \nF) r
    \overset{\eqref{eq:rtt.F.tau=sum.E.doftE.tau}}= \int_F \rttvec[F]{\tau}\cdot \CURLF r,
  \end{aligned}
\end{equation*}
where we have invoked \eqref{eq:zero-jump} with $\varphi_V = (\bvec{v}_V\cdot\normal_F) r$ in the cancellation.
Hence, we have $\Rproj[F]{k}\DGntF\uvec{v}_F = \Rproj[F]{k}\rttvec[F]{\tau}$, 
which, combined with \eqref{eq:cRproj.DGntF=cRproj.tauRF}, gives, 
after recalling the definition \eqref{eq:RT} of $\RT{k+1}$ and using \cite[Eq.~(2.14)]{Di-Pietro.Droniou:23},
\begin{equation} \label{eq:exact.DGSC.proofdefRT}
  \DGntF\uvec{v}_F = \rttvec[F]{\tau}.
\end{equation}

For any $r \in \Poly{k+1}(F)$, writing the definition \eqref{eq:defDGttF} of $\DGttF\uvec{v}_F$ with $\btens{\sigma} = \CURLF \GRAD_F r$ and noticing that $\VDIVF\CURLF \GRAD_F r = 0$, $\tr \CURLF \GRAD_F r = \ROTF \GRAD_F r = 0$, 
  and that $\tangent_E^\top(\CURLF\GRADF r)\normal_{FE} = -\dotpp{r}$ 
  and $\normal_{FE}^\top(\CURLF\GRADF r)\normal_{FE} = -\partial_{\normal_{FE}}\dotp{r}$ for all $E\in\EF$,
we get
\begin{multline} \label{eq:ex.1.F.2.0}
  \int_F \DGttF\uvec{v}_F\tdot\CURLF\GRADF{r}
  = - \sumEF (\dofnE{v} \cdot \nFE)\, \partial_{\nFE}\dotp{r}
  \\
  - \sumEF \doftE{v}\,\dotpp{r}    
  \eqcolon\term_1 + \term_2.
\end{multline}
Writing $\term_1 = \sum_{E\in\EF}\wFE\term_1(E)$
and using the 
the definition \eqref{eq:IPP.E} of $\DerE{k+1}$, 
we have, for all $E\in\EF$,
\begin{equation}\label{eq:ex.1.F.2.1}
  \begin{aligned}
    \term_1(E)
    &= \int_E \DerE{k+1} (\bvec{v}_{V_1} \cdot \nFE, \bvec{v}_{V_2} \cdot \nFE, \dofnE{v} \cdot \nFE)\, \partial_{\nFE}r
    - \jump{E}{(\bvec{v}_{V} \cdot \nFE) \partial_{\nFE}r}
    \\
    \overset{\eqref{eq:ex.1.E.ttau}}
    &= \int_E (\doftE{\tau}\cdot\nFE)\,\partial_{\nFE}r 
    - \jump{E}{(\bvec{v}_{V} \cdot \nFE) \partial_{\nFE}r}.
  \end{aligned}
\end{equation}
To treat the second term, we start by noticing that, for all $r\in\Poly{k+1}(F)$,
\begin{multline*} 
  \sumEF \tr \gdofE{v}\, \dotp{r}
  \\
  \begin{aligned}
    \overset{\eqref{eq:IPP.E}}&=
    -\sumEF \tr \DerE{k+1}(\trnnE{(\gdofV[V_1]{v})},\trnnE{(\gdofV[V_2]{v})}, \gdofE{v})\, r
    + \sum_{E\in\EF}\omega_{FE}\jump{E}{\tr \trnnE{(\gdofV{v})} \, r}
    \\
    \overset{\eqref{eq:ex.1.E.0}}&=
    -\sumEF \tr \cdofE{\tau} \, r
    + \sum_{E\in\EF}\omega_{FE}\jump{E}{\tr \trnnE{(\gdofV{v})} \, r}
    \\
    \overset{\eqref{eq:exact.DGSC.proofdefLE},\,\eqref{eq:exact.DGSC.defGV}}&=
    -\sumEF \tr \DerE{k+1}(\trnnE{(\btens{\tau}_{V_1})},\trnnE{(\btens{\tau}_{V_2})}, \btens{\tau}_E) \, r
    - 2 \sumEF \lambda_E\, r
    \\
    &\quad
    + \sum_{E\in\EF}\omega_{FE} \jump{E}{\tr \trnnE{(\btens{\tau}_V)} \, r}
    + \cancel{2\sum_{E\in\EF}\omega_{FE}\jump{E}{\lambda_V \, r}}
    \\
    \overset{\eqref{eq:IPP.E}}&=
    \sumEF \tr \btens{\tau}_E\, \dotp{r}
    - 2 \sumEF \lambda_E\, r,
  \end{aligned}
\end{multline*}
where the cancellation is a consequence of \eqref{eq:zero-jump} 
  with $\varphi_V = \lambda_V \, r$.
Combining this relation with \eqref{eq:ex.1.E.2}, we get
\[
\sumEF \DerE{k} (\bvec{v}_{V_1} \cdot \tE, \bvec{v}_{V_2} \cdot \tE,\doftE{v})\,\dotp{r} 
= - \sumEF \tr \btens{\tau}_E \, \dotp{r} - \sumEF \lambda_E\, r.
\]
Therefore, using the definition \eqref{eq:IPP.E} of $\DerE{k}$, we have that
\begin{equation} \label{eq:ex.1.F.2.3}
  \begin{aligned}
    \term_2
    &=
    \sumEF \DerE{k} (\bvec{v}_{V_1} \cdot \tE, \bvec{v}_{V_2} \cdot \tE,\doftE{v})\,\dotp{r}
    - \sum_{E\in\EF}\omega_{FE}\jump{E}{(\bvec{v}_V \cdot \tE) \dotp{r}}\\
    &=
    - \sumEF \tr \btens{\tau}_E \, \dotp{r}
    - \sumEF \lambda_E\, r
    - \sum_{E\in\EF}\omega_{FE} \jump{E}{(\bvec{v}_V \cdot \tE) \dotp{r}}.
  \end{aligned}
\end{equation}
Plugging \eqref{eq:ex.1.F.2.1} and \eqref{eq:ex.1.F.2.3} into \eqref{eq:ex.1.F.2.0} gives
\begin{equation}\label{eq:ex.1.F.2.4}
  \begin{aligned}
    &\int_F \DGttF\uvec{v}_F\tdot\CURLF\GRADF{r}
    \\
    &\quad=
    \sumEF(\doftE{\tau}\cdot\nFE)\,\partial_{\nFE}r
    - \sumEF\tr\btens{\tau}_E\,\dotp{r}
    - \sumEF \lambda_E\,r
    \\
    &\qquad - \sum_{E\in\EF}\wFE \left(
    \jump{E}{(\bvec{v}_{V} \cdot \nFE) \partial_{\nFE}r} + 
    \jump{E}{(\bvec{v}_V \cdot \tE) \dotp{r}}
    \right)
    \\
    &\quad=
    \int_F \cgtvec{\tau} \tdot \CURLF\GRADF{r}
    - \cancel{\sum_{E\in\EF}\wFE \jump{E}{\bvec{v}_V \cdot \GRAD_F r}},
  \end{aligned}
\end{equation}
where we have used \eqref{eq:ex.DGSC.tFloop} in the second equality,
enforced \eqref{eq:SCT=0:2} to cancel the term involving $\SCddofF \utens{\tau}_F$, 
and invoked \eqref{eq:zero-jump} with $\varphi_V = \bvec{v}_V \cdot \GRADF r$ to cancel the sum over the edges.
The definition \eqref{eq:ex.1.F.def2} of $\gdofF{v}$ readily gives,
for all $\btens{\sigma} \in \CURL_F\bvec{x}^\perp \Poly{k-1}(F)$,
\begin{equation} \label{eq:ex.1.F.2.5}
  \int_F \DGttF\uvec{v}_F\tdot\btens{\sigma} = \int_F \cgtvec{\tau} \tdot \btens{\sigma}.
\end{equation}
Recalling the definition \eqref{eq:ex.1.F.def3} of $\doftF{v}$ and using \eqref{eq:ex.1.F.2.4} and \eqref{eq:ex.1.F.2.5} together with the decomposition \eqref{eq:ex.1.F.dec1} 
to infer $\CGproj{k-1} \DGttF \uvec{v}_F = \CGproj{k-1} \utens{\tau}_F$, we finally get, 
after recalling \eqref{eq:CGtrimPoly} and using \cite[Eq.~(2.14)]{Di-Pietro.Droniou:23},
\begin{equation}\label{eq:ex.1.F.DGtt}
  \DGttF \uvec{v}_F = \cgtvec{\tau}.
\end{equation}

\subsubsection{Element component}

Finally, for the element component, we use the fact that
$\VDIV : \SRoly{\compl,k}(T) \rightarrow \vPoly{k-1}(T;\Real^3)$ is an isomorphism 
to find, from \eqref{eq:defDGT},
\begin{equation} \label{eq:ex.1.T.1}
\text{%
  $\bvec{v}_T \in \vPoly{k-1}(T;\Real^3)$ such that
  $\cSRproj{k} \DGT \uvec{v}_T = \cSRproj{k}\srtvec[T]{\tau}$.
}
\end{equation}
  Recalling \cite[Eq.~(2.14)]{Di-Pietro.Droniou:23}, in order to prove that $\SRtrimproj{k} \DGT \uvec{v}_T = \srtvec[T]{\tau}$, it
only remains to check that
\begin{equation}\label{eq:DGT=projSR-.tau}
  \SRproj{k-1} \DGT \uvec{v}_T = \SRproj{k-1}\srtvec[T]{\tau}.
\end{equation}
To prove \eqref{eq:DGT=projSR-.tau}, we start writing the definition \eqref{eq:defDGT} of $\DGT$ with $\btens{\sigma} \in \Holy{\compl,l}(T)$ and using the fact that $\VDIV \CURL \btens{\sigma} = 0$ 
and that $(\CURL \btens{\sigma}) \nF = \VDIV(\btens{\sigma} \times \nF)$ for all $F\in\FT$
to infer
\begin{multline*}
  \term\coloneq
  \int_T \DGT \uvec{v}_T \tdot \CURL\btens{\sigma} 
  \\
  = \sumFT \dofnF{v}\, \VDIV(\btens{\sigma}\times\nF) \cdot \nF
  + \sumFT \hdoftF{v} \cdot \VDIV(\btens{\sigma}\times\nF)
  \eqcolon\mathfrak{A} + \mathfrak{B}.
\end{multline*}
Recalling \eqref{eq:poly.trace.H} and observing that $\vPoly{k}(F;\Real^2) \subset \RT{k+1}(F)$, we can invoke the definitions \eqref{eq:defDGntF} of $\DGntF$ and \eqref{eq:defDGttF} of $\DGttF$ to continue as follows:
\[
\begin{aligned}
  &\left.\begin{aligned}
    \term
    &=
    -\sumFT \DGntF\uvec{v}_F \cdot \trntF{(\btens{\sigma}\times\nF)}
    \\
    &\quad
    + \sumEFT 
    (\dofnE{v} \cdot \nF)
    \,
    \nF^\top (\btens{\sigma} \times \nF) \nFE
  \end{aligned}
  \right\}\mathfrak{A}
  \\
  &\quad
  \left.\begin{aligned}
    &\quad
    -\sumFT \DGttF\uvec{v}_F \tdot \trttF{(\btens{\sigma}\times\nF)}
    - \frac13 \sumFT \gdofF{v} \, \cancel{\ID{2}\tdot \trttF{(\btens{\sigma}\times\nF)}}
    \\
    &\quad+ \sumEFT 
    (\dofnE{v} \cdot \nFE)
    \, \nFE^\top(\btens{\sigma} \times \nF)\nFE\\
    &\quad
    + \sumEFT 
    \doftE{v}
    \,\tE^\top (\btens{\sigma}\times\nF)\nFE,
  \end{aligned}\right\}\mathfrak{B}
\end{aligned}
\]
where the cancellation on the third line occurs because $\btens{\sigma}$ is symmetric, therefore 
$\tr \trttF{ ( \btens{\sigma}\times \nF ) } = 0$.
We continue using the relations \eqref{eq:exact.DGSC.proofdefRT} and \eqref{eq:ex.1.F.DGtt} 
to replace, respectively, $\DGntF\uvec{v}_F$ with $\rttvec[F]{\tau}$ and $\DGttF \uvec{v}_F$ with $\cgtvec{\tau}$ in the first and third terms in the right-hand side,
combining the second and fifth terms and 
using the injection \eqref{eq:injection.3d} in $\Real^3$
(additionally using the fact that $(\btens{\sigma}\times\nF)\nFE = - \btens{\sigma}\tE$),
and noticing that $\tE^\top (\btens{\sigma}\times\nF)\nFE = - \trttE{\sigma}$ in the last term to obtain
\begin{equation} \label{eq:ex.1.T.2}
  \begin{aligned}
    \term
    &=
    - \sumFT \rttvec[F]{\tau} \cdot \trntF{(\btens{\sigma}\times\nF)} 
    - \sumFT \cgtvec[F]{\tau} \tdot \trttF{(\btens{\sigma}\times\nF)}\\ 
    &\quad - \cancel{%
      \sumEFT 
      \hdofnE{v}
      \btens{\sigma}\tE
    }
    - \cancel{\sumEFT 
      \doftE{v}
      \, \trttE{\sigma}}\\
    \overset{\eqref{eq:defSCT}}&= \int_T \srvec{\tau} \tdot \CURL \btens{\sigma} - \cancel{\int_T \SCT \utens{\tau}_T \tdot \btens{\sigma}},
  \end{aligned}
\end{equation}
where the cancellations of the edge terms follows from \eqref{eq:zero-edges}, while the conclusion is a consequence of the zero-sym curl condition \eqref{eq:SCT=0:4}.
By Lemma \ref{lemma:poly.SRfromH}, \eqref{eq:ex.1.T.2} implies \eqref{eq:DGT=projSR-.tau}.

\subsection{Proof of \eqref{eq:exact.SCDD}}

We conclude by counting the dimensions of each space, which are explicitly known and can be expressed in terms of the number of geometric entity of each dimension.
Specifically, we have for all $k \geq 1$
\begin{equation*}
  \begin{aligned}
    \DIM \RT{1}(T)
    &= 4
    \\
    \DIM \uHdevgrad{T}
    &=
    12 \vert \VT \vert
    + (7 k + 6) \vert \ET \vert
    + (2 k^2 + 3 k + 1) \vert \FT \vert
    + \frac12 (k^3 + 3 k^2+ 2 k),
    \\
    \DIM \uHsymcurl{T}
    &=
    8 \vert \VT \vert
    + (10 k + 16) \vert \ET \vert
    + (3 k^2 + 8 k + 3) \vert \FT \vert
    +\frac16 (8 k^3 + 33 k^2+ 25 k),
    \\
    \DIM \uHdivdiv{T}
    &=
    -3
    + (3 k + 6) \vert \ET \vert
    + (k^2 + 5 k + 6) \vert \FT \vert
    + (k^3 + 5 k^2+ 5 k),
    \\
    \DIM \Poly{k}(T)
    &= 1 + \frac16(k^3 + 6 k^2 + 11 k).
  \end{aligned}
\end{equation*}
Moreover, by the exactness properties proved above, we have
\begin{equation*}
  \begin{aligned}
    \DIM \Image \DG{T}
    &=
    \dim \uHdevgrad{T} - \dim \Ker \DG{T}
    \overset{\eqref{eq:exact.RTDG}}=
    \dim \uHdevgrad{T} - \dim \RT{1}(T)
    \\
    &=
    -4
    + 12 \vert \VT \vert
    + (7 k + 6) \vert \ET \vert
    + (2 k^2 + 3 k + 1) \vert \FT \vert
    + \frac12(k^3 + 3 k^2+ 2 k),
    \\
    \DIM \Image \SC{T}
    &=
    \dim \uHsymcurl{T} - \DIM \Ker \SC{T}
    \overset{\eqref{eq:exact.DGSC}}=
    \dim \uHsymcurl{T} - \DIM \Image \DG{T}    
    \\
    &=
    4
    - 4 \vert \VT \vert
    + (3 k + 10) \vert \ET \vert
    + (k^2 + 5 k + 2) \vert \FT \vert
    + \frac16(5 k^3 + 24 k^2+ 19 k),
    \\
    \DIM \Ker \DD{T}
    &= \dim \uHdivdiv{T} - \dim \Image \DD{T}
    \overset{\eqref{eq:exact.SCDD}}=
    \dim \uHdivdiv{T} - \dim \Poly{k}(T)
    \\
    &=
    -4
    + (3 k + 6) \vert \ET \vert
    + (k^2 + 5 k + 6) \vert \FT \vert
    + \frac16(5 k^3 + 24 k^2+ 19 k).
  \end{aligned}
\end{equation*}
Therefore,
\begin{equation*}
  \DIM \Ker \DD{T} - \DIM \Image \SC{T} = 4 \left(
  \vert\VT\vert - \vert\ET\vert + \vert\FT\vert - 2
  \right).
\end{equation*}
The Euler characteristic for an element with trivial topology gives the identity
$\vert\VT\vert - \vert\ET\vert + \vert\FT\vert = 2$.
Therefore, $\DIM \Ker \DD{T} = \DIM \Image \SC{T}$.
We conclude using the local complex property \eqref{eq:LC.SCDD}.

\begin{remark}[The case $k=0$] \label{rem:SCDD.k=0}
  The formulas above fail when $k = 0$.
  Indeed, they give a negative dimension (of $-3$) on the cell.
  The problem stems from the fact that $\DIM \Holy{-1}(T) = \DIM \Holy{-2}(T)$.
  Correcting the formulas, we find
  \begin{equation*}
    \DIM \Ker DD_T^0 - \DIM \Image \utens{SC}_T^0 = 4 \left(
    \vert\VT\vert - \vert\ET\vert + \vert\FT\vert \right)
    - 5 
    = 3,
  \end{equation*}
    showing that exactness does not hold for $k=0$.
\end{remark}

\appendix

\section{Results on local polynomial spaces}\label{sec:results.polynomial.spaces}

\begin{lemma}[Isomorphism of the face divergence between polynomial spaces]\label{lemma:DIVFCG}
  The operator $\VDIVF:\CGoly{\compl,\ell}(F)\to\vPoly{\ell-1}(F;\Real^2)$ is an isomorphism.
\end{lemma}

\begin{proof}
  Take $\bvec{x}_F = \bvec{0}$, and write $\bvec{x} = (x,y)^\top$.
  By \eqref{eq:Poly=CGoly+cCGoly}, it holds
  \[
  \CGoly{\compl,\ell}(F) =
  \left\lbrace 
  \btens{A}(P_1,P_2) \coloneq 
  \begin{pmatrix}
    x\,P_1 & y\,P_1 \\
    x\,P_2 & y\,P_2
  \end{pmatrix}
  -
  \begin{pmatrix}
    y\,P_2 & -y\,P_1 \\
    -x\,P_2 & x\,P_1
  \end{pmatrix}
  \st P_1, P_2 \in \Poly{\ell-1}(F)\right\rbrace .
  \]
  The divergence of a generic $\btens{A}(P_1,P_2)$ is thus given by
  \[
  \VDIVF \btens{A}(P_1,P_2)
  = \begin{pmatrix} (x\partial_x + 2 y\partial_y + 3) P_1 - y\partial_x P_2 \\
    (2 x\partial_x + y\partial_y + 3) P_2 - x\partial_yP_1
  \end{pmatrix}.
  \]
  This expression behaves well on monomials:
  Given two couples $(i_1,j_1)$ and $(i_2,j_2)$ of non-negative integers, the above expression for $P_1(x,y) = x^{i_1}y^{j_1}$ and $P_2(x,y) = \lambda x^{i_2}y^{j_2}$, $\lambda\in\Real$, becomes
  \begin{equation} \label{eq:proof.DIVFCG}
    \VDIVF \btens{A}(x^{i_1}y^{j_1},\lambda x^{i_2}y^{j_2})
    = \begin{pmatrix} (i_1 + 2 j_1 + 3) x^{i_1}y^{j_1} - \lambda i_2x^{i_2 - 1}y^{j_2 +1} \\
      \lambda (2 i_2 +  j_2 + 3) x^{i_2}y^{j_2} - j_1x^{i_1 + 1}y^{j_1 - 1} \end{pmatrix}.
  \end{equation}  

  To prove the injectivity of $\VDIVF$, let us show that $\VDIVF\btens{A}(P_1,P_2)\equiv 0$ implies that both $P_1$ and $P_2$ are identically zero. We
  see from \eqref{eq:proof.DIVFCG} that each monomial $x^iy^j$ of $P_1$ must be cancelled by a monomial $x^{i+1}y^{j-1}$ in $P_2$ and vice-versa.
  If $j = 0$,
  no contribution from $P_2$ appears on the first component in \eqref{eq:proof.DIVFCG}, 
  and we must have $(i + 3)x^i = 0$ (which is impossible since $i\ge 0$), 
  else there must be $\lambda\in\Real$ such that
  \[
  \VDIVF \btens{A}(x^iy^j,\lambda x^{i+1}y^{j-1})
  = \begin{pmatrix}
    \left[i + 2j + 3 - (i+1)\lambda\right] x^iy^j
    \\
    \left[\lambda(2i + j + 4) - j\right] x^{i+1}y^{j-1}
  \end{pmatrix}\equiv 0.
  \]
  This condition requires $\lambda = \frac{i+2j+3}{i+1} = \frac{j}{2i+j+4}$, i.e.,
  \[
  (i+2j+3)(2i+j+4) = (i+1)j
  \implies
  2 (i+j)^2 + 10(i+j) + 10 = 0,
  \]
  which is impossible to satisfy, 
  showing that the only possibility for $\VDIVF \btens{A}(P_1,P_2)\equiv 0$ to hold is that $P_1 = P_2 \equiv 0$, 
  i.e., $\VDIVF$ is injective on $\CGoly{\compl,\ell}(F)$.

  Let us now prove its surjectivity by showing that every vector-valued field can be obtained as a divergence of an element of $\CGoly{\ell}(F)$.
  To this end, it suffices to consider the case where one component is a monomial and the other is zero.
  Letting $(i,j)$ denote a couple of non-negative integers such that $i+j\le\ell$, the
  above computation gives for $j > 0$ and $\lambda = \frac{j}{2i+j+4}$,
  \[
  \VDIVF \btens{A}\left(
  x^iy^j,\frac{j}{2i+j+4}x^{i+1}y^{j-1}
  \right)
  = \frac{2(i+j)^2 + 10(i+j) + 12}{2 i + j + 4}
  \begin{pmatrix} x^iy^j \\ 0
  \end{pmatrix}
  \]
  and, by symmetry,
  \[
  \VDIVF \btens{A}\left(
  \frac{i}{i + 2j + 4}x^{i-1}y^{j+1},x^{i}y^{j}
  \right)
  = \frac{2(i+j)^2 + 10(i+j) + 12}{i + 2 j + 4}
  \begin{pmatrix} 0 \\ x^iy^j \end{pmatrix},
  \]
  which concludes the proof since $\left\{x^iy^j\st \text{$i\ge 0$, $j\ge 0$, and $i+j$}\le\ell-1\right\}$ is a basis of $\Poly{\ell-1}(F)$ and its tensorization a basis of $\vPoly{\ell-1}(F;\Real^2)$.
\end{proof}

\begin{proof}[Proof of Lemma~\ref{lemma:Poly=CGoly+cCGoly}]
  Lemma \ref{lemma:DIVFCG} gives 
  $\CGoly{\ell}(F) \cap \cCGoly{\ell}(F) = \lbrace 0 \rbrace$.
  We only have to check that 
  $\DIM \tPoly{\ell}(F;\Real^{2\times 2}) = 2\ell^2 + 6\ell+4 = \DIM \CGoly{\ell}(F) + \DIM \cCGoly{\ell}(F)$.
  We can compute the dimension of $\CGoly{\ell}(F)$ from the isomorphism $\CURLF:\Poly[0]{\ell+1}(F)\to\Roly{\ell}(F)$ as follows:
  $
  \DIM \CGoly{\ell}(F)
  = 2 \big(\DIM \Poly{\ell+1}(F) - 1 \big)
  = \ell^2 + 5\ell+4.
  $
  On the other hand, the dimension of $\cCGoly{\ell}(F)$ is given by Lemma \ref{lemma:DIVFCG}:
  $
  \DIM \cCGoly{\ell}(F) = \DIM \vPoly{\ell-1}(F;\Real^2)
  = \ell^2 + \ell.
  $
  Summing the above expressions, the result follows.
\end{proof}


\section*{Acknowledgements}

The partial support of \emph{Agence Nationale de la Recherche} (ANR) and I-Site MUSE through the grant ANR-16-IDEX-0006 ``RHAMNUS'' is gratefully acknowledged.
Daniele Di Pietro acknowledges the partial support of ANR through the grant ANR-20-MRS2-0004 ``NEMESIS''.


\printbibliography

\end{document}